\documentclass[12pt,a4paper]{article}
\usepackage{amsfonts,latexsym,amsmath, amssymb, amscd, amsthm}
\usepackage{mathrsfs}

\usepackage{bm}
\usepackage{appendix}
\usepackage{color}
\usepackage[top=1.2in, bottom=1.2in, left=1.25in, right=1.25in]{geometry}
\usepackage{cite}
\usepackage{cases}
\usepackage[utf8]{inputenc}
\usepackage{graphicx,float,enumerate}
\usepackage[T1]{fontenc}
\usepackage{authblk}
\numberwithin{equation}{section}

\usepackage[colorlinks=true,citecolor=red,linkcolor=blue,urlcolor=RubineRed,pdfpagetransition=Blinds,pdftoolbar=false,pdfmenubar=false]{hyperref}

\newcommand{\baa}{\begin{array}}
\newcommand{\eaa}{\end{array}}
\def\be{\begin{equation}}
\def\ee{\end{equation}}
\def\epsilon{\varepsilon}

\def\R{\mathbb{R}}
\def\N{\mathbb{N}}

\def\tilde{\widetilde}
\def\varep{\varepsilon}
\def\ds{\displaystyle}

\newtheorem{theorem}{Theorem}[section]
\newtheorem{lemma}[theorem]{Lemma}
\newtheorem{proposition}[theorem]{Proposition}
\newtheorem{corollary}[theorem]{Corollary}
\newtheorem{definition}[theorem]{Definition}
\newtheorem{remark}[theorem]{Remark}
\newtheorem{conjecture}[theorem]{Conjecture}

\allowdisplaybreaks

\title{\bf{Reaction-diffusion fronts in funnel-shaped domains}}
\author{Fran\c cois Hamel\thanks{Aix-Marseille Univ, CNRS, Centrale Marseille, I2M, Marseille, France (\texttt{francois.hamel@univ-amu.fr}). This work has received funding from Excellence Initiative of Aix-Marseille University~-~A*MIDEX, a French ``Investissements d'Avenir'' programme, and from the ANR NONLOCAL (ANR-14-CE25-0013) and RESISTE (ANR-18-CE45-0019) projects.} \ and Mingmin Zhang\thanks{Aix-Marseille Univ, CNRS, Centrale Marseille, I2M, Marseille, France, and School of Mathematical Sciences, University of Science and Technology of China, Hefei, Anhui 230026, China (\texttt{mingmin.zhang.math@gmail.com}). M. Zhang is  supported by the China Scholarship Council for 2 years of study at Aix-Marseille Universit\'{e}.}}

\date{}
\begin{document}

\maketitle

\begin{abstract} 
We consider bistable reaction-diffusion equations in funnel-shaped domains of $\mathbb{R}^N$ made up of straight parts and conical parts with positive opening angles. We study the large time dynamics of entire solutions emanating from a planar front in the straight part of such a domain and moving into the conical part. We show a dichotomy between blocking and spreading, by proving especially some new Liouville type results on stable solutions of semilinear elliptic equations in the whole space $\R^N$. We also show that any spreading solution is a transition front having a global mean speed, which is the unique speed of planar fronts, and that it converges at large time in the conical part of the domain to a well-formed front whose position is approximated by expanding spheres. Moreover, we provide sufficient conditions on the size~$R$ of the straight part of the domain and on the opening angle $\alpha$ of the conical part, under which the solution emanating from a planar front is blocked or spreads completely in the conical part. We finally show the openness of the set of parameters~$(R,\alpha)$ for which the propagation is complete.
\end{abstract}


\section{Introduction and main results}\label{Sec-1}

This paper is devoted to the study of propagation phenomena of time-global (entire) bounded solutions $u=u(t,x)$ of reaction-diffusion equations of the type
\begin{align}
\label{1}
\begin{cases}
u_t=\Delta u+f(u), & t\in\mathbb{R},\ x\in\overline{\Omega},\\
\nu\cdot \nabla u=0,   & t\in\mathbb{R},\ x\in\partial\Omega,
\end{cases}
\end{align}
in certain unbounded smooth domains $\Omega\subset\mathbb{R}^N$ with $N\ge 2$. Here $u_t$ stands for $\frac{\partial u}{\partial t}$, and $\nu=\nu(x)$ is the outward unit normal on the boundary $\partial\Omega$, that is, Neumann boundary conditions are imposed on $\partial\Omega$. Equations of type \eqref{1} arise especially in the fields of population dynamics, mathematical ecology, physics and also medicine and biology. The function $u$ typically stands for the temperature or the concentration of a species. It is assumed to be bounded, then with no loss of generality we suppose that it takes values in $[0,1]$. The reaction term $f$ is assumed to be of class $C^{1,1}([0,1],\mathbb{R})$ and such that 
\begin{equation}
\label{f-bistable-1}
f(0)=f(1)=0, \ \ f'(0)<0, \ \ f'(1)<0,
\end{equation}
which means that both 0 and 1 are stable zeros of $f$. Moreover, we assume that $f$ is of the bistable type with positive mass, that is, there exists $\theta\in(0,1)$ such that 
\begin{equation}
\label{f-bistable-2}
f<0~\text{in}~ (0,\theta),\quad f>0~\text{in}~(\theta,1),\quad f'(\theta)>0,\quad \int_{0}^{1}f(s)ds>0.
\end{equation}
The fact that $f$ has a positive mass over $[0,1]$ means the state $1$ is in some sense more stable than~$0$.\footnote{If the integral of $f$ over $[0,1]$ were negative, the study would be similar, after changing $u$ into $1-u$ and $f(s)$ into $-f(1-s)$. If the integral of $f$ over $[0,1]$ were equal to $0$, the analysis of the propagation phenomena would be very different, since then no front connecting $0$ and $1$ with nonzero speed can exist in the one-dimensional version of~\eqref{1}.} A typical example of a function~$f$ satisfying \eqref{f-bistable-1}-\eqref{f-bistable-2} is the cubic nonlinearity $f(u)=u(1-u)(u-\theta)$ with $\theta\in (0,1/2)$. For mathematical purposes, we extend~$f$ in $\mathbb{R}\backslash[0,1]$ to a $C^{1,1}(\R,\R)$ function as follows: $f(s)=f'(0)s$ for $s<0$, and $f(s)=f'(1)(s-1)$ for $s>1$.

One main question of interest for the solutions of~\eqref{1} is the description of their dynamical properties as $t\to\pm\infty$. The answer to this question depends strongly on the geometry of the underlying domain $\Omega$. In the one-dimensional real line $\R$, a prominent role is played by a class of particular solutions, namely the traveling fronts. More precisely, with assumptions \eqref{f-bistable-1}-\eqref{f-bistable-2} above, equation~\eqref{1} in $\R$ admits a unique planar traveling front $\phi(x-ct)$ solving 
\begin{align}
\label{TW}
\begin{cases}
\phi''+c\phi'+f(\phi)=0\ \text{in}\ \mathbb{R},\vspace{3pt}\\
\phi(-\infty)=1,\ \ \phi(+\infty)=0,\ \ 0<\phi<1\ \text{in}\ \mathbb{R},\ \ \phi(0)=\theta,
\end{cases}
\end{align}
see, for instance, \cite{AW1978,FM1977,K1962}. The profile $\phi$ is then a connection between the stable steady states~$1$ and~$0$. Moreover, $\phi'<0$ in $\mathbb{R}$, and $c$ is positive since $f$ has a positive integral over $[0,1]$. The traveling front $\phi(x-ct)$ is invariant in the moving frame with speed $c$, and it attracts as $t\to+\infty$ a large class of front-like solutions of the associated Cauchy problem, see~\cite{FM1977}. It is also known that $\phi$ (resp. $1-\phi$) decays exponentially fast at $+\infty$ (resp. $-\infty$), that is,
\begin{align}
\label{phi}
\left\{\baa{lll}
c_1e^{-\mu^* z}\le \phi(z)\le C_1e^{-\mu^* z}, &z\ge0, &\ds\text{with }\mu^*=\frac{c+\sqrt{c^2-4f'(0)}}{2}>0,\vspace{3pt}\\
c_2e^{\mu_* z}\le 1-\phi(z)\le C_2e^{\mu_* z}, &z<0, &\ds\text{with }\mu_*=\frac{-c+\sqrt{c^2-4f'(1)}}{2}>0,
\eaa\right.
\end{align}
where $c_1,c_2, C_1$ and $C_2$ are positive constants. The derivative $\phi'(z)$ also satisfies
\begin{align}
\label{phi'}
\begin{cases}
c_3e^{-\mu^* z}\le -\phi'(z)\le C_3e^{-\mu^* z}, &z\ge0,\\
c_4e^{\mu_* z}\le -\phi'(z)\le C_4e^{\mu_* z}, &z<0,
\end{cases}
\end{align}
with positive constants $c_3,c_4, C_3$ and $C_4$. Such planar fronts exist under the assumptions \eqref{f-bistable-1}-\eqref{f-bistable-2}, whereas if $f$ satisfies~\eqref{f-bistable-1} only, fronts connecting $0$ and $1$ do not exist in general, see~\cite{FM1977} for more precise conditions for the existence and non-existence. \textit{Throughout this paper, we assume that $f$ satisfies \eqref{f-bistable-1}-\eqref{f-bistable-2} and that $\phi$ and $c>0$ are uniquely defined as in~\eqref{TW}.}


\subsection{Notations}

We focus in this paper on the case of equation~\eqref{1} set in unbounded domains of $\mathbb{R}^N$, made up of a straight part and a conical part: we assume that the left (say, with respect to the direction~$x_1$) part of~$\Omega$, namely $\Omega^-=\Omega\cap\{x\in\R^N:x_1\le 0\}$, is a straight half-cylinder in the direction $-x_1$ with cross section of radius $R>0$, while the right part, namely $\Omega^+=\Omega\backslash\Omega^-$, is a cone-like set with respect to the $x_1$-axis and with opening angle~$\alpha\ge0$. More precisely, we assume that $\Omega$ is rotationally invariant with respect to the $x_1$-axis, that is,
\begin{equation}
\label{defOmega}
\Omega=\big\{x=(x_1,x')\in\mathbb{R}^N :\ x_1\in\mathbb{R},\ |x'|<h(x_1)\big\},
\end{equation}
where $|\ |$ denotes the Euclidean norm, and that $h: \mathbb{R}\to \mathbb{R}^+$ is a $C^{2,\beta}(\R)$ (with $0<\beta<1$) function satisfying the following properties:
\be\label{h}
\left\{\baa{ll}
0\le h '\le \tan \alpha\text{ in }\mathbb{R}, & \text{for some angle }\alpha\in[0,\pi/2),\vspace{3pt}\\
h=R\text{ in }(-\infty,0], & \text{for some radius}~R>0,\vspace{3pt}\\
h(x_1)=x_1\,\tan\alpha\text{ in}\  [L\cos \alpha,+\infty), & \text{for some} \ L>R,\eaa\right.
\ee
see Figure~1. Such a domain is then called ``funnel-shaped''. In the particular limit case $\alpha=0$, the domain $\Omega$ amounts to a straight cylinder in $\mathbb{R}^N$ with cross section of radius $R$.  Notice that, when $\alpha>0$, the cross section is unbounded as $x_1\to+\infty$. To emphasize the dependence on $R$ and $\alpha$, we will also use the notation $\Omega_{R,\alpha}$ for convenience. The domains $\Omega_{R,\alpha}$ are not uniquely defined by~\eqref{defOmega}-\eqref{h}, and they also depend on the parameter $L$ in~\eqref{h}, but only the parameters $R>0$ and $\alpha\in[0,\pi/2)$ will play an important role in our study (except in Theorem~\ref{thm6} below). Other domains which have a globally similar shape, but may be only asymptotically straight in the left part or asymptotically conical in the right part could have been considered, at the expense of less precise estimates and more technical calculations. Since the domains satisfying~\eqref{defOmega}-\eqref{h} lead to a variety of interesting and non-trivial phenomena, we restrict ourselves to~\eqref{defOmega}-\eqref{h} throughout the paper. 

\begin{figure}[H]
	\centering
	\includegraphics[scale = 1.1]{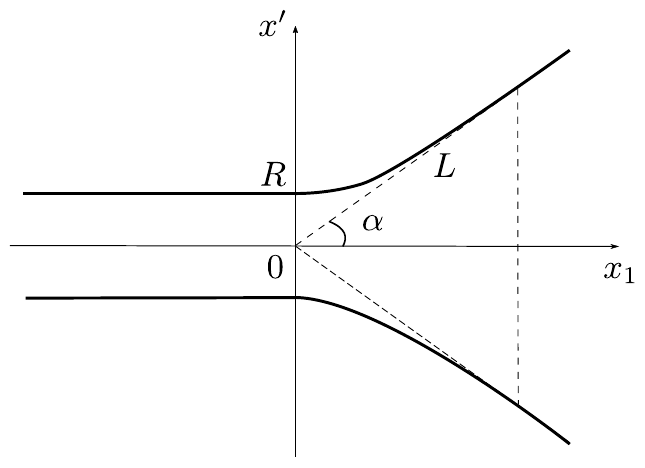}
	\caption{\small Schematic figure of the domain $\Omega_{R,\alpha}$ for $R>0$ and $\alpha\in(0,\pi/2)$.}
\end{figure}

If the domain is a straight cylinder in the direction $x_1$ (this happens in the case $\alpha=0$), then the planar front $\phi(x_1-ct)$ given by~\eqref{TW} solves~\eqref{1} (furthermore, up to translation, any transition front connecting $0$ and $1$ in the sense of Definition~\ref{Def1} below is equal to that front, see~\cite{GHS2020,H2016}). Here a domain $\Omega=\Omega_{R,\alpha}$ satisfying~\eqref{defOmega}-\eqref{h} is straight in its left part only, and the standard planar front $\phi(x_1-ct)$ does not fulfill the Neumann boundary conditions when $\alpha>0$. But it is still very natural to consider solutions of~\eqref{1} behaving in the past like the planar front $\phi(x_1-ct)$ coming from the left part of the domain, and to investigate the outcome of these solutions as they move into the right part of the domain. More precisely, we consider time-global solutions of~\eqref{1} emanating from the planar front $\phi(x_1-ct)$, that is,
\begin{equation}\label{initial}
u(t,x)-\phi(x_1-ct)\to 0\ \text{as}\ t\to-\infty,\ \text{uniformly with respect to}\ x\in\overline\Omega
\end{equation}
(notice that, in the right part $\Omega^+$ of $\Omega$, this condition simply means that $u(t,\cdot)\to0$ as $t\to-\infty$ uniformly in $\overline{\Omega^+}$). We will see that such solutions exist and are unique, and the main goal of the paper is to study their behavior as $t\to+\infty$, in terms of the parameters $R$ and $\alpha$.


\subsection{Background}

To describe the dynamical properties of the solutions of~\eqref{1} satisfying~\eqref{initial}, we use the unifying notions of generalized traveling fronts, called transition fronts, introduced in \cite{BH2007,BH2012}. In order to recall these notions of transition fronts and that of global mean speed, let us introduce some notations. Let $d_{\Omega}$ be the geodesic distance in $\overline\Omega$ (with respect to the Euclidean distance~$d$ in $\R^N$). For any two subsets $A$ and $B$ of $\overline{\Omega}$, we set
\begin{equation*}
d_\Omega(A,B)=\inf\big\{d_\Omega(x,y):(x,y)\in A\times B\big\},
\end{equation*}
and $d_\Omega(x,A)=d_\Omega(\{x\},A)$ for $x\in\overline{\Omega}$. We also use similar definitions with $d$, instead of $d_\Omega$, for the Euclidean distance between subsets of $\R^N$. Consider now two families $(\Omega^-_t)_{t\in\mathbb{R}}$ and $(\Omega^+_t)_{t\in\mathbb{R}}$ of open non-empty subsets of $\Omega$ such that 
\begin{equation}
\label{1.7}
\begin{aligned}
\forall\,t\in\mathbb{R},\ 
\begin{cases}
\Omega^-_t\cap\Omega^+_t=\emptyset,\ \ \partial\Omega^-_t\cap\Omega=\partial\Omega^+_t\cap\Omega=:\Gamma_t\neq\emptyset,\ \ \Omega^-_t\cup\Gamma_t\cup\Omega^+_t=\Omega,\vspace{3pt}\\
\sup\big\{d_\Omega(x,\Gamma_t):x\in\Omega^+_t\big\}=\sup\big\{d_\Omega(x,\Gamma_t):x\in\Omega^-_t\big\}=+\infty
\end{cases}
\end{aligned}
\end{equation}
and
\begin{equation}
\label{1.8}
\begin{aligned}
\begin{cases}
\inf\big\{\sup\{d_\Omega(y,\Gamma_t):y\in\Omega^+_t,d_\Omega(y,x)\le r\}:t\in\mathbb{R},x\in\Gamma_t\big\}\to+\infty\vspace{3pt}\\
\inf\big\{\sup\{d_\Omega(y,\Gamma_t):y\in\Omega^-_t,d_\Omega(y,x)\le r\}:t\in\mathbb{R},x\in\Gamma_t\big\}\to+\infty
\end{cases}
~\text{as}~r\to+\infty.
\end{aligned}
\end{equation}
Condition~\eqref{1.8} says that for any $M>0$, there is $r_M>0$ such that for every $t\in\mathbb{R}$ and $x\in\Gamma_t$, there are $y^\pm=y^\pm_{t,x}\in\mathbb{R}^N$ such that 
\begin{equation}
\label{1.9}
y^\pm\in\Omega^\pm_t,~d_\Omega(x,y^\pm)\le r_M~\text{and}~d_\Omega(y^\pm,\Gamma_t)\ge M.
\end{equation}
In other words, any point on $\Gamma_t$ is not too far from the centers of two large balls (in the sense of the geodesic distance in $\overline{\Omega}$) included in $\Omega^-_t$ and $\Omega^+_t$, this property being uniform with respect to $t$ and to the point on $\Gamma_t$. Moreover, in order to avoid interfaces with infinitely many twists, the sets~$\Gamma_t$ are assumed to be included in finitely many graphs: there is an integer $n\ge 1$ such that, for each $t\in\mathbb{R}$, there are $n$ open subsets $\omega_{i,t}\subset\mathbb{R}^{N-1}$ (for $1\le i\le n$), $n$ continuous maps $\psi_{i,t}:\omega_{i,t}\to \mathbb{R}$ and $n$ rotations $R_{i,t}$ of $\mathbb{R}^N$ with
\begin{equation}
\label{1.10}
\Gamma_t\subset \bigcup\limits_{1\le i\le n} R_{i,t}\big(\big\{x=(x',x_N)\in\mathbb{R}^N:x'\in\omega_{i,t},x_N=\psi_{i,t}(x')
\big\}\big).
\end{equation}

\begin{definition}[$\!\!$\cite{BH2007,BH2012}] 
\label{Def1}	
For problem \eqref{1}, a transition front connecting $1$ and $0$ is a classical solution $u: \mathbb{R}\times\overline\Omega\to (0,1)$ for which there exist some sets $(\Omega^\pm_t)_{t\in\mathbb{R}}$ and $(\Gamma_t)_{t\in\mathbb{R}}$ satisfying \eqref{1.7}-\eqref{1.10} and for every $\varep>0$ there exists $M_\varep>0$ such that 
\begin{equation}
\label{1.11}
\begin{aligned}
\begin{cases}
\forall\,t\in\mathbb{R},~\forall\,x\in\overline{\Omega^+_t},~~d_\Omega(x,\Gamma_t)\ge M_\varep~\Longrightarrow ~u(t,x)\ge 1-\varep,\vspace{3pt}\\
\forall\,t\in\mathbb{R},~\forall\,x\in\overline{\Omega^-_t},~~d_\Omega(x,\Gamma_t)\ge M_\varep~\Longrightarrow ~u(t,x)\le \varep.
\end{cases}
\end{aligned}
\end{equation}
Furthermore, $u$ is said to have a global mean speed $\gamma\in[0,+\infty)$ if 
\begin{equation*}
\frac{d_\Omega(\Gamma_t,\Gamma_s)}{|t-s|}\to\gamma~~~\text{as}~~|t-s|\to+\infty.
\end{equation*}
\end{definition}

This definition has been shown in \cite{BH2007,BH2012, H2016} to cover and unify all classical cases of traveling fronts in various situations. Condition~\eqref{1.11} means that the transition between the steady states $1$ and $0$ takes place in some uniformly-bounded-in-time neighborhoods of $\Gamma_t$. For a given transition front connecting~$1$ and~$0$, the families $(\Omega^\pm_t)_{t\in\mathbb{R}}$ and $(\Gamma_t)_{t\in\mathbb{R}}$ satisfying~\eqref{1.7}-\eqref{1.11} are not unique, but the global mean speed $\gamma$, if any, does not depend on the choice of the families~$(\Omega^\pm_t)_{t\in\mathbb{R}}$ and $(\Gamma_t)_{t\in\mathbb{R}}$, see~\cite{BH2012}.

Before stating the main results of this paper, let us recall here some related works on the role of the geometry of $\Omega$ on propagation phenomena for equations of the type~\eqref{1}. It was shown in \cite{CG2005} that, for $\Omega$ being a succession of two semi-infinite straight cylinders with square cross sections of different sizes $r$ and $R$, the solution $u$ emanating from the planar front $\phi(x_1-ct)$ in the left half-cylinder with smaller section and going to the right one with larger section can be blocked, in the sense that
\begin{equation}
\label{blocking}
u(t,x)\to u_\infty(x) ~\text{as}~t\to+\infty ~\text{locally uniformly in}~x\in\overline\Omega,~~\text{with}~u_\infty(x)\to 0 ~\text{as}~x_1\to+\infty.
\end{equation}
Later, propagation and blocking phenomena for different kinds of cylindrical domains with uniformly bounded cross sections were investigated in~\cite{BBC2016}.\footnote{The existence and uniqueness of time-global solutions emanating from planar fronts in more general asymptotically straight cylindrical domains was also proved in~\cite{P2018}.}  Especially, if the section of the cylindrical domain is non-increasing with respect to $x_1$, or if it is non-decreasing, large enough, and axially star-shaped, then the solution of~\eqref{1} emanating from the planar front $\phi(x_1-ct)$ propagates completely in the sense that
\begin{equation}
\label{complete}
u(t,x)\to 1 ~~\text{as}~t\to+\infty ~\text{locally uniformly in}~x\in\overline{\Omega}.
\end{equation}
However, under some other geometrical conditions (when typically, the cross section is narrow and then becomes abruptly much wider), blocking phenomena can occur, in the sense of~\eqref{blocking}. Further propagation and/or blocking phenomena were also shown for bistable equations set in the real line $\R$ (with periodic heterogeneities~\cite{DHZ2015,DLL2018,DGM2014,HFR2010,X1993,XZ1995,Z2017}, with local defects~\cite{BRR2013,CS2011,LK2000,N2015,P1981,SS1997}, or with asymptotically distinct left and right environments~\cite{E2018}), as well as in straight infinite cylinders with non-constant drifts~\cite{E2019,E2020}, and in some periodic domains~\cite{DR2018} or the whole space with periodic coefficients~\cite{D2016,GR2020}. In~\cite{RRBK2008}, a reaction-diffusion model was considered to analyse the effects on population persistence of simultaneous changes in the position and shape of a climate envelope. Recently, the existence and characterization of the global mean speed of transition fronts in domains with multiple cylindrical branches were investigated in~\cite{GHS2020}. It was proved that the front-like solutions emanating from planar fronts in some branches and propagating completely are transition fronts moving with the planar speed $c$ and eventually converging to planar fronts in the other branches. The classification of such fronts in domains with multiple asymptotically cylindrical branches was shown in~\cite{G2020}.

Meanwhile, the interaction between smooth compact obstacles $K\subset\R^N$ and a bistable planar front $\phi(x_1-ct)$ was studied in \cite{BHM2009}. An entire solution $u(t,x)$ converging to $\phi(x_1-ct)$ as $t\to-\infty$ uniformly in $\overline{\Omega}=\mathbb{R}^N\backslash\mathring{K}$ was constructed in \cite{BHM2009}. It was also proved that if the obstacle~$K$ is star-shaped or directionally convex with respect to some hyperplane, then the solution passes the obstacle in the sense that $u(t,x)$ converges to $\phi(x_1-ct)$ as $t\!\to\!+\infty$ uniformly in $\overline{\Omega}$. In particular, the propagation is then complete in the sense of~\eqref{complete}. Furthermore, the solution is a transition front connecting $0$ and $1$, in the sense of Definition~\ref{Def1}, and one can choose $\Gamma_t=\{x\in\Omega=\mathbb{R}^N\backslash K: x_1=ct\}$ in~\eqref{1.7} (the transition front is then said to be almost planar). Moreover, the authors constructed non-convex obstacles $K$ for which the solution~$u$ emanating from the bistable planar front $\phi(x_1-ct)$ as $t\to-\infty$ does not pass the obstacle completely, in the sense that~\eqref{complete} is not fulfilled. Furthermore, it follows from~\cite{GHS2020} that all transition fronts connecting $1$ and $0$ propagate completely and have a global mean speed equal to the planar speed $c$ (examples of such fronts are the almost-planar fronts given in~\cite{BHM2009} and the $V$-shaped fronts constructed in~\cite{GM2019}). The solutions which do not propagate completely are still transition fronts, but they connect $0$ and a steady state less than $1$ in $\overline{\Omega}$, see~\cite{GHS2020}.
 
Unlike the cylindrical domains with two branches considered in~\cite{BBC2016,CG2005,E2019,E2020} or with multiple branches considered in~\cite{G2020,GHS2020}, the domains $\Omega=\Omega_{R,\alpha}$ given by~\eqref{defOmega}-\eqref{h} have sections which are not uniformly bounded, as soon as $\alpha>0$. Natural questions are to derive estimates, as~$t\to+\infty$, on the location and shape of the level sets of the solutions of~\eqref{1} satisfying~\eqref{initial}, and also to know whether the solutions remain front-like in the sense of Definition~\ref{Def1}. We also study in this paper the role of the geometrical parameters $R$ and $\alpha$ on the propagation or blocking phenomena. Since standard planar traveling fronts do not exist anymore in such domains (as soon as $\alpha>0$), the analysis of the spreading properties of the solutions of~\eqref{1} is much more complex than in the one-dimensional case or the case of straight cylinders. First of all, the existence and uniqueness of the entire solution $u$ of~\eqref{1} satisfying~\eqref{initial} is derived as in~\cite{BBC2016,BHM2009,P2018}. Then, we will show that the blocking or complete propagation properties,~\eqref{blocking} or~\eqref{complete}, are the only possible outcomes of the solution $u$ at large time. We will see that $u$ is always a transition front connecting $1$ and $0$ and that it has a global mean speed, equal to $c$, if the propagation is complete. It is worth to mention that the solution can never go ahead of the planar front $\phi(x_1-ct)$, as that planar front is a supersolution for~\eqref{1}. We will actually show that, if $\alpha>0$, and even if the propagation is complete, the solution lags far behind the planar front $\phi(x_1-ct)$ in the direction of $x_1$ at $t\to+\infty$, in the sense that any level set of $u$ is well approximated by the expanding spherical surface of radius $ct-((N-1)/c)\ln t+O(1)$ and is asymptotically locally planar. Then, we will give some sufficient conditions related to the parameters~$(R,\alpha)$ so that $u$ will propagate completely or be blocked. Moreover, we will also prove the openness of the set of parameters $(R,\alpha)\in(0,+\infty)\times(0,\pi/2)$ for which~$u$ propagates completely. In short, our results will then give a refined picture of the spatial shape and temporal dynamics of the level sets of front-like solutions in funnel-shaped domains, a geometrical configuration which had not been investigated before.


\subsection{General properties for any given $(R,\alpha)$}\label{Sec-1.3}

Our first result is the well-posedness of problem~\eqref{1} with the asymptotic past condition~\eqref{initial} as $t\to-\infty$, for any given $R>0$ and $\alpha\in[0,\pi/2)$.

\begin{proposition}
\label{thm1}
For any $R>0$ and $\alpha\in[0,\pi/2)$, problem~\eqref{1} admits a unique entire solution $u(t,x)$ emanating from the planar front $\phi(x_1-ct)$, in the sense of~\eqref{initial}. Moreover, $u_t(t,x)>0$ and $0<u(t,x)<1$ for all $(t,x)\in\mathbb{R}\times\overline{\Omega}$, and there exists $u_\infty(x)=\lim\limits_{t\to+\infty}u(t,x)$ in~$C^2_{loc}(\overline\Omega)$ satisfying $0<u_\infty(x)\le 1$ in~$\overline{\Omega} $ and
\begin{align}
\label{u-infty}
\begin{cases}
\Delta u_\infty+f(u_\infty)=0 &\text{in}\ \overline{\Omega},\vspace{3pt}\\
\nu\cdot\nabla u_\infty =0   &\text{on}\ \partial\Omega.
\end{cases}
\end{align}
Lastly, for each $t\in\R$, the function $u(t,\cdot)$ is axisymmetric with respect to the $x_1$-axis, that is, it only depends on $x_1$ and $|x'|$, with $x'=(x_2,\cdots,x_N)$.
\end{proposition}

From the strong maximum principle, one has either $u_\infty\equiv 1$ in $\overline{\Omega}$, or $u_\infty<1$ in $\overline{\Omega}$. Notice also that, from~\eqref{initial} and the monotonicity in $t$, there holds $u_\infty(x)\to 1$ as $x_1\to-\infty$ uniformly in $|x'|\le R$. The proof of Proposition~\ref{thm1} follows from the construction of a sequence of Cauchy problems and of some suitable sub- and supersolutions, as in~\cite{BBC2016,BHM2009,E2018,P2018}. It will be just sketched in Section~\ref{Sec-2}.

Once the well-posedness of~\eqref{1} with the past condition~\eqref{initial} is established, we then focus on the large time dynamics of the solution $u$ given in Proposition~\ref{thm1}. It turns out that the complete propagation in the sense of~\eqref{complete} or the blocking in the sense of~\eqref{blocking} are the only two possible outcomes. Namely, we will show that the following dichotomy holds.

\begin{theorem}\label{thm2}
For any $R>0$ and $\alpha\in[0,\pi/2)$, let $u$ be the solution of~\eqref{1} and~\eqref{initial} given in Proposition~$\ref{thm1}$. Then, either $u$ propagates completely in the sense of~\eqref{complete}, or it is blocked in the sense of~\eqref{blocking} and then the convergence of $u(t,\cdot)$ to $u_\infty$ as $t\to+\infty$ in~\eqref{blocking} is uniform in $\overline{\Omega}$.
\end{theorem}

\begin{remark}
When $\alpha=0$ in~\eqref{defOmega}-\eqref{h}, $\Omega$ amounts to a straight cylinder and, by uniqueness, the solution $u$ given in Proposition~$\ref{thm1}$ is nothing but the planar front $\phi(x_1-ct)$, hence the propagation is complete in this very particular case.
\end{remark}

Theorem~\ref{thm2} means that, under the notations of Proposition~\ref{thm1}, either $u_\infty\equiv 1$ in $\overline{\Omega}$, or $u_\infty(x)\to0$ as $x_1\to+\infty$. Any other more complex behavior is impossible. Theorem~\ref{thm2} is a consequence of the stability of the solution $u_\infty$ and of some Liouville type results for the stable solutions of some semilinear elliptic equations in the two-dimensional plane, or in a two-dimensional half-plane, or in the whole space $\R^N$ with axisymmetry. In order to give a flavor of these properties and results, which are also of independent interest, let us state here the definition of stability\footnote{For a thorough study of stable solutions of elliptic equations, we refer to the book~\cite{D2011}.} as well as one of the typical results shown in Section~\ref{Sec-3.2}. So, for a non-empty open connected set $\omega\subset\R^N$, we say that a $C^2(\overline{\omega})$ solution $U$ of $\Delta U+f(U)=0$ in~$\overline{\omega}$ is stable if
\be\label{stable}
\int_{\overline{\omega}}|\nabla\psi|^2-f'(U)\psi^2\ge0
\ee
for every $\psi\in C^1(\overline{\omega})$ with compact support (for instance, it turns out that the solution $u_\infty$ of~\eqref{u-infty} in $\overline{\Omega}$, given in Proposition~\ref{thm1}, is stable, see Lemma~\ref{lemstable} below). The following result, concerned with stable axisymmetric solutions, is also shown in Section~\ref{Sec-3.2}.

\begin{proposition}\label{proliouville}
Let $0\le U\le 1$ be a $C^2(\R^N)$ stable solution of $\Delta U+f(U)=0$ in $\R^N$. Assume that $U$ is axisymmetric with respect to the $x_1$-axis, that is, $U$ depends on $x_1$ and $|x'|$ only, with~$x'=(x_2,\cdots,x_N)$. Then, either $U\equiv0$ in $\R^N$ or $U\equiv1$ in $\R^N$.
\end{proposition}

Coming back to problem~\eqref{1} in funnel-shaped domains, we then turn to the study of the spreading properties and the behavior of the level sets of the solutions under the complete propagation condition~\eqref{complete} when $\alpha\in(0,\pi/2)$. In the sequel, we denote the level sets and the upper level sets of $u$ by:
\begin{equation}
\label{def-level set}
\mathcal{E}_\lambda(t)=\big\{x\in\overline{\Omega}: u(t,x)=\lambda\big\},\ \mathcal{U}_\lambda(t)=\big\{x\in\overline{\Omega}:u(t,x)>\lambda\big\},\hbox{ for }\lambda\in(0,1)\hbox{ and }t\in\R.
\end{equation}

\begin{theorem}
\label{thm3}
For any $R>0$ and $\alpha\in(0,\pi/2)$, let $u$ be the solution of~\eqref{1} and~\eqref{initial} given in Proposition~$\ref{thm1}$. If $u$ propagates completely in the sense of~\eqref{complete}, then it is a transition front connecting $1$ and $0$ with global mean speed $c$, and $(\Gamma_t)_{t\in\mathbb{R}}$, $(\Omega^{\pm}_t)_{t\in\mathbb{R}}$ in Definition~$\ref{Def1}$ can be defined by 
\begin{equation}
\label{Gamma_t}
\begin{cases}
\Gamma_t=\big\{x\in\Omega:x_1=ct\big\} & \hbox{for } t\le t_0,\vspace{3pt}\\
\ds\Gamma_t=\Big\{x\in\Omega:x_1>0\hbox{ and }|x|=ct-\frac{N-1}{c}\ln t\Big\} & \hbox{for }t>t_0,
\end{cases}
\end{equation}
and
\begin{align}
\label{Omega_t}
\begin{cases}
\Omega^{\pm}_t=\big\{x\in\Omega: \pm(x_1-ct)<0\big\} & \hbox{for }t\le t_0,\vspace{3pt}\\
\ds\Omega^+_t=\Big\{x\in\Omega: x_1\le0,\hbox{ or }x_1>0\hbox{ and }|x|<ct-\frac{N-1}{c}\ln t\Big\} & \hbox{for }t>t_0,\vspace{3pt}\\
\ds\Omega^-_t=\Big\{x\in\Omega: x_1>0\hbox{ and }|x|>ct-\frac{N-1}{c}\ln t\Big\} & \hbox{for }t>t_0,
\end{cases}
\end{align}
with $t_0>0$ large enough such that $ct-((N-1)/c)\ln t>L$ for all $t>t_0$.\footnote{We recall that $L$ is given in~\eqref{h}, with $L>R$.} Moreover, $u$ converges to planar fronts locally along its level sets as $t\to+\infty$: for any $\lambda\in(0,1)$, any sequence $(t_n)_{n\in\N}$ diverging to $+\infty$ and any sequence $(x_n)_{n\in\N}$ in $\overline{\Omega}$ such that $u(t_n,x_n)=\lambda$, then
\be\label{convtaun}
u(t+t_n,x+x_n)-\phi\Big(x\cdot\frac{x_n}{|x_n|}-ct+\phi^{-1}(\lambda)\Big)\longrightarrow 0\hbox{ in }C^{1,2}_{(t,x);loc}(\R\times\R^N)\hbox{ as $n\to+\infty$}
\ee
if $d(x_n,\partial\Omega)\to+\infty$ as $n\to+\infty$, and the same limit holds with the additional restriction $x+x_n\in\overline{\Omega}$ if $\limsup_{n\to+\infty}d(x_n,\partial\Omega)<+\infty$. Lastly, for every $\lambda\in(0,1)$, there exists $r_0>0$ such that the upper level set $\mathcal{U}_\lambda(t)$ satisfies
\begin{equation}
\label{upper level set}
S_{r(t)-r_0}\subset\mathcal{U}_\lambda(t)\subset S_{r(t)+r_0}
\end{equation}
for all $t$ large enough $($see Figure~$2$$)$, with $S_r$ and $r(t)$ given by
$$S_r=\overline{\Omega^-}\cup\big\{x\in\overline{\Omega}:|x|\le r\big\},\quad r(t)=ct-\frac{N-1}{c}\ln t.$$
\end{theorem}

\begin{figure}[H]
\centering
\includegraphics[scale=0.7]{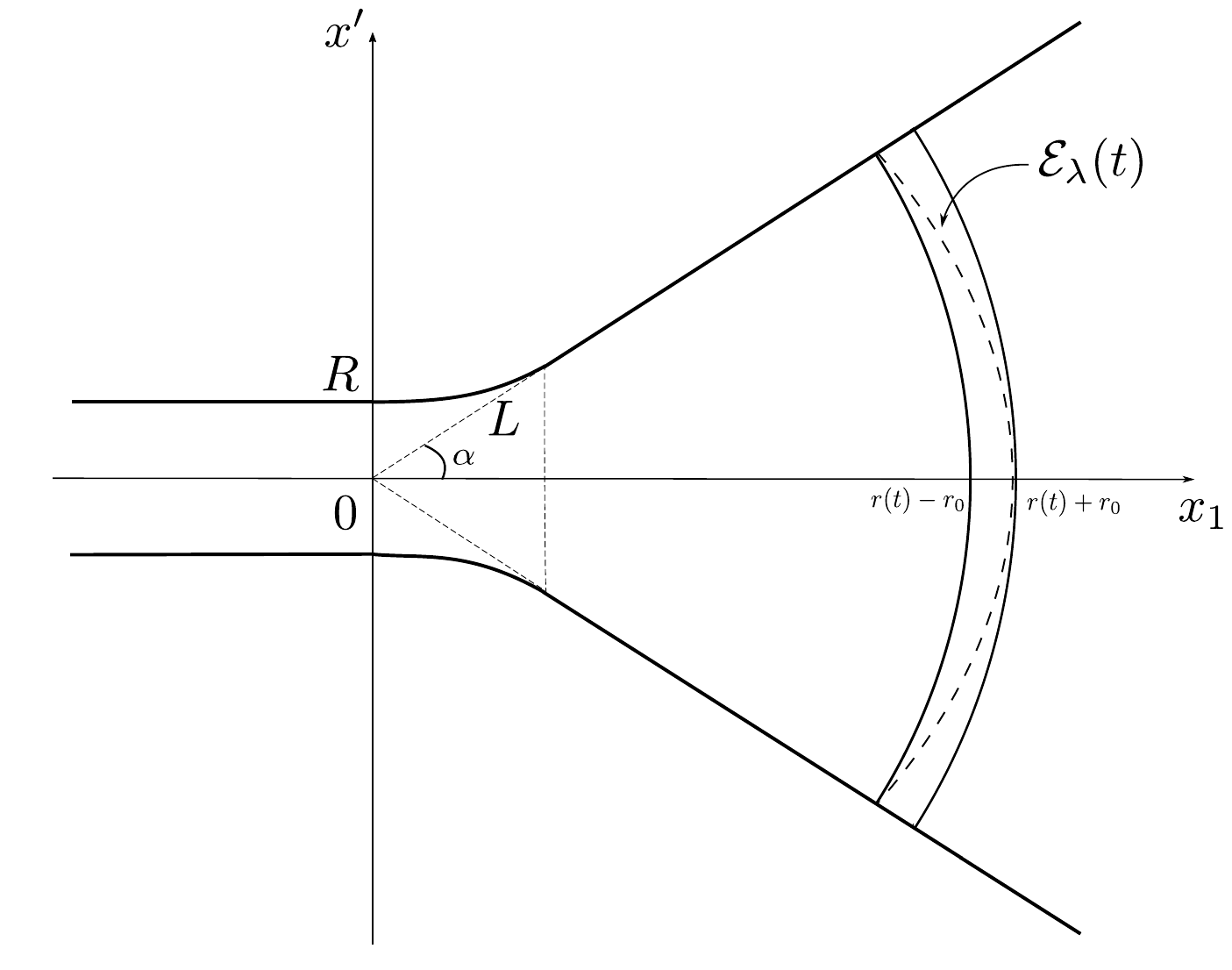}
\caption{\small Possible location of the level set $\mathcal{E}_\lambda(t)$ for $\lambda\in(0,1)$ and $t>0$ large.}
\end{figure}

In other words, the past condition~\eqref{initial} and the complete propagation condition~\eqref{complete} guarantee the spreading of the solution $u$ and the propagation with global mean speed $c$. Furthermore, the width of the transition between the limit states $1$ and $0$ is uniformly bounded in time in the sense of Definition~\ref{Def1} and the solution locally converges to planar fronts as $t\to+\infty$. The estimates of the location of the level sets as $t\to+\infty$ are established by constructing sub- and supersolutions whose level sets have roughly expanding spherical shapes of radii $ct-((N-1)/c)\ln t+O(1)$, see Lemma~\ref{lemma-super sub} below. The logarithmic gap $((N-1)/c)\ln t$ is due to the curvature of the level sets, and these estimates are similar to those obtained in~\cite{U1985} for the solutions of the Cauchy problem in $\R^N$ with compactly supported initial conditions and complete propagation. In our case, at time $t=0$ (as at any other time), the function $x\mapsto u(t,x)$ converges to $0$ as $x_1\to+\infty$, but it then invades the right part of the domain, a situation similar to the case of invading solutions with initial compact support in $\R^N$. The proof of the asymptotic planar property is based on compactness arguments and a Liouville-type theorem for entire solutions of the bistable equation in the whole space given in~\cite[Theorem~3.1]{BH2007}.

Theorem~\ref{thm3} shows that the solutions $u$ of Proposition~\ref{thm1} that propagate completely are transition fronts connecting $0$ and $1$, with global mean speed equal to $c$. It also turns out, this time immediately from Proposition~\ref{thm1}, that the solutions $u$ that are blocked are still transition fronts connecting $1$ and $0$, but they do not have any global mean speed.

\begin{theorem}
\label{thm4}
For any $R>0$ and $\alpha\in(0,\pi/2)$, let $u$ be the solution of~\eqref{1} and~\eqref{initial} given in Proposition~$\ref{thm1}$. If $u$ is blocked in the sense of~\eqref{blocking}, then it is a transition front connecting~$1$ and~$0$ without any global mean speed, and $(\Gamma_t)_{t\in\mathbb{R}}$, $(\Omega^{\pm}_t)_{t\in\mathbb{R}}$ can be defined by 
\begin{equation}
\label{Gamma_tbis}
\left\{\baa{lll}
\Gamma_t=\big\{x\in\Omega:x_1=ct\big\}  & \hbox{and }\ \Omega^{\pm}_t=\big\{x\in\Omega: \pm(x_1-ct)<0\big\} & \hbox{for } t\le 0,\vspace{3pt}\\
\ds\Gamma_t=\big\{x\in\Omega:x_1=0\big\} & \hbox{and }\ \Omega^{\pm}_t=\big\{x\in\Omega: \pm\,x_1<0\big\} & \hbox{for }t>0.\eaa\right.
\end{equation}
\end{theorem}


\subsection{Complete propagation for large $R$}

From now on, we investigate the effect of the parameters $R$ and $\alpha$ of the funnel-shaped domains~$\Omega=\Omega_{R,\alpha}$ on the propagation phenomena of the front-like solution $u$ of \eqref{1} satisfying the past condition~\eqref{initial}. We first recall that, when $\alpha=0$, $u(t,x)\equiv\phi(x_1-ct)$ and the propagation is complete, whatever $R>0$ may be. Our next result provides some sufficient conditions on the size $R>0$ to ensure the complete propagation condition~\eqref{complete} when $\alpha>0$.

\begin{theorem}
\label{thm5}
There is $R_0>0$ such that, if $R\ge R_0$ and $\alpha>0$, then the unique solution $u$ of~\eqref{1} satisfying \eqref{initial} propagates completely in the sense of \eqref{complete}, and therefore all the conclusions of Theorem~$\ref{thm3}$ are valid.
\end{theorem}

This theorem shows that  the invasion always occurs no matter the size of the opening angle in the right part is, provided the left part of the domain is not too thin (see Figure 3).  The proof relies on the existence of a compactly supported subsolution, with maximum larger than~$\theta$, to the elliptic problem~\eqref{u-infty}, and on the sliding method used to compare $u_\infty$ with some shifts of this subsolution.  

\begin{figure}[H]
\centering
\includegraphics[scale=0.45]{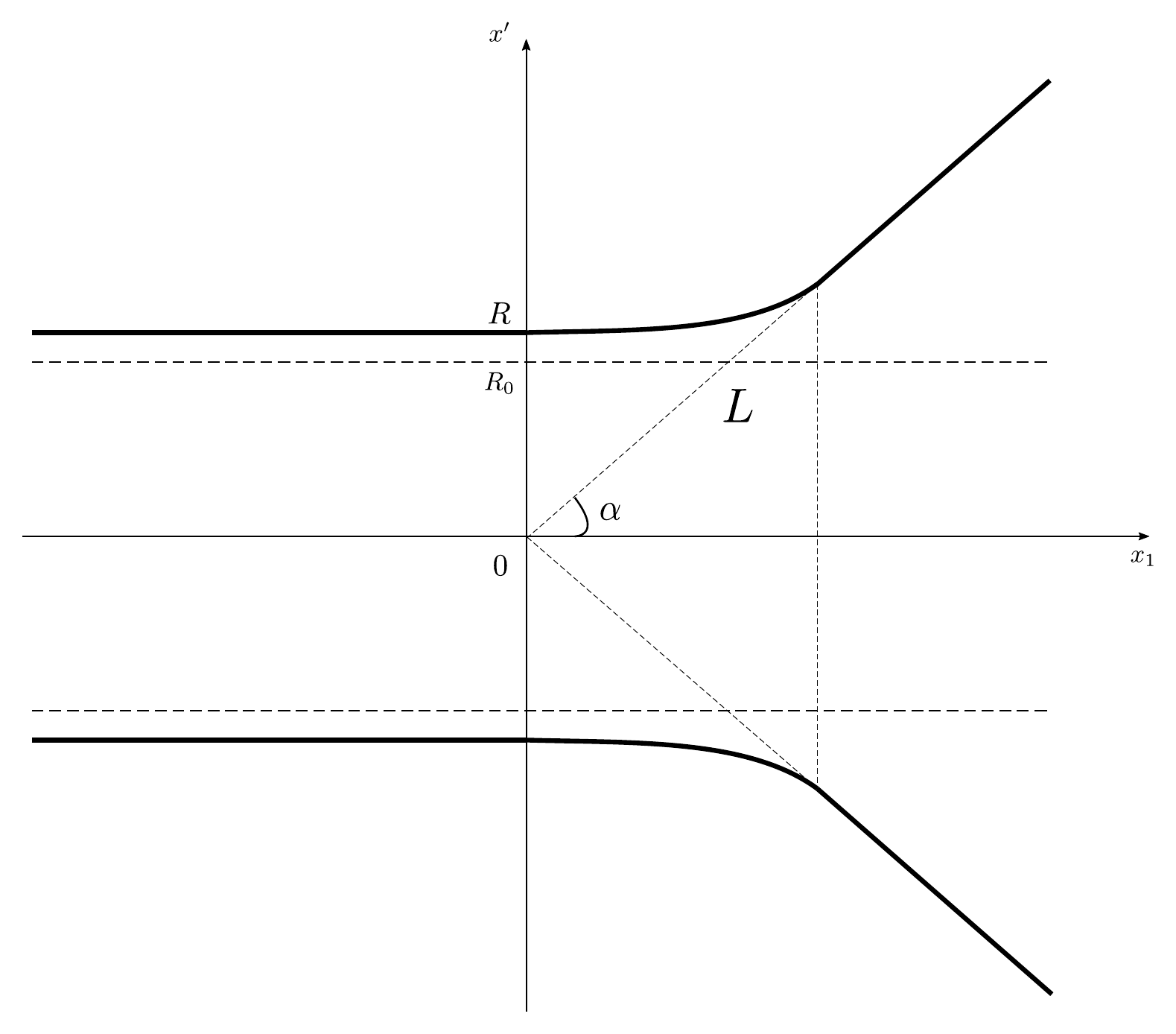}
\caption{\small Schematic figure of the domain $\Omega_{R,\alpha}$ for $\alpha\in(0,\pi/2)$ and $R> R_0$.}
\end{figure}


\subsection{Blocking for $R\ll 1$ and $\alpha$ not too small}

The next result is concerned with blocking phenomena. We prove that the  solution $u$ of~\eqref{1} in $\Omega_{R,\alpha}$ with past condition~\eqref{initial} is blocked if $R$ is sufficiently small and $\alpha$ is sufficiently close to $\pi/2$ (see Figure~4).

\begin{figure}[H]
\centering
\includegraphics[scale=1.35]{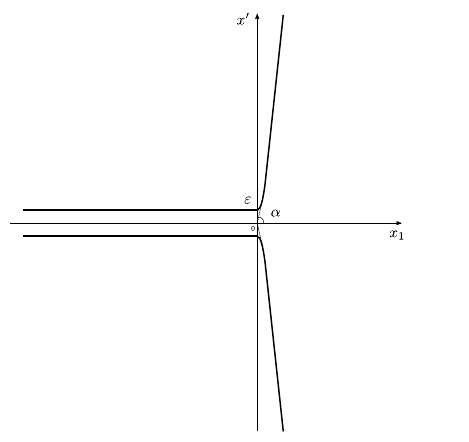}
\caption{\small Schematic figure of the domain $\Omega_{\varep,\alpha}$ for $R=\varep\ll 1$.}
\end{figure}

\begin{theorem}
\label{thm6}
Assume that $N\ge3$ and let $L_*>0$ and $\alpha_*\in(0,\pi/2)$ be given. Then there is~$R_*>0$ such that, if $0<R\le R_*$, $\alpha_*\le\alpha<\pi/2$ and $L\le L_*$ in~\eqref{defOmega}-\eqref{h}, then the solution~$u$ of~\eqref{1} in $\Omega$ with past condition~\eqref{initial} is blocked, in the sense of~\eqref{blocking}.
\end{theorem}

From a biological point of view, Theorem \ref{thm6} says that as the species goes from a very narrow passage into a suddenly wide open space, the diffusion disperses the population to lower density where the reaction behaves adversely. That prevents the species from rebuilding a strong enough basis to invade the right part of the domain. This phenomenon is similar to the problem studied in~\cite{CG2005}, although the proof given here, based on the construction of suitable supersolutions, is completely different.

Let us now make some further remarks on the effect of the geometry of the domain on invasion or blocking phenomena. In population dynamics, where $u$ stands for the population density, one can think of the invasion of fishes from mountain streams into an endless ocean, and more generally speaking the invasion of plants or animals subject to an Allee effect and going from an isthmus into a large area. In medical sciences, the bistable reaction-diffusion equation is used to model the motion of depolarization waves in the brain, in which the domain can be thought of as a portion of grey matter of the brain with different thickness: here $u$ represents the degree of depolarization, and the Neumann boundary condition means that the grey matter is assumed to be isolated. Equations of the type~\eqref{1} can also be used to study ventricular fibrillations. Ventricular fibrillation is a state of electrical anarchy in part of the heart that leads to rapid chaotic contractions, which are fatal unless a normal rhythm can be restored by defibrillation. When excitation waves enter the circular area of cardiac tissue, they are trapped and their propagation triggers off ventricular fibrillations~\cite{AH1945}. Therefore, understanding how the geometrical properties of the cardiac fibres or fibre bundles affect or even block the propagation of excitation waves is of vital importance. For more detailed backgrounds and explanations from biological view point, we refer to \cite{CG2005,BBC2016,GL1991} and the references therein.


\subsection{The set of parameters $(R,\alpha)$ with complete propagation is open in $(0,+\infty)\times(0,\pi/2)$
}

In the final main result, we show that if the front-like solution $u$ emanating from the planar traveling front satisfies the  complete propagation property~\eqref{complete} in $\Omega_{R,\alpha}$ for some $R>0$ and $\alpha\in(0,\pi/2)$, then, with a slight perturbation of $R$ and $\alpha$, the solution $u$ will still propagate completely in the perturbed domain. For this result, we use an additional assumption on the continuous dependence of $\Omega_{R,\alpha}$ with respect to $(R,\alpha)$.

\begin{theorem}
\label{thm7}
Assume that the functions $h$ given in~\eqref{defOmega}-\eqref{h} depend continuously on the parameters $(R,\alpha)\in(0,+\infty)\times(0,\pi/2)$ in the $C^{2,\beta}_{loc}(\R)$ sense, with $0<\beta<1$. Then the set of parameters $(R,\alpha)$ such that the solution $u$ of~\eqref{1} in $\Omega_{R,\alpha}$ with past condition~\eqref{initial} propagates completely, in the sense of~\eqref{complete}, is open in $(0,+\infty)\times(0,\pi/2)$.
\end{theorem}

The continuity of the functions $h$ given in~\eqref{defOmega}-\eqref{h} implies the local continuity of the domains $\Omega_{R,\alpha}$ in the sense of the Hausdorff distance. This continuity holds only in a local sense, since actually the Hausdorff distance between $\Omega_{R,\alpha}$ and $\Omega_{R',\alpha'}$ is infinite as soon as $\alpha\neq\alpha'$. But the local continuity is sufficient to guarantee the validity of~\eqref{complete} under small perturbations of $(R,\alpha)$. The proof of Theorem~\ref{thm7} is done by way of contradiction and it uses, as that of Theorem~\ref{thm5}, the existence of a compactly supported subsolution, with maximum larger than $\theta$, to the elliptic problem~\eqref{u-infty}.

From Theorems~\ref{thm2} and~\ref{thm7}, the next corollary follows immediately.

\begin{corollary}
Under the assumptions of Theorem~$\ref{thm7}$, the set of parameters $(R,\alpha)\in(0,+\infty)\times(0,\pi/2)$ such that the solution $u$ of~\eqref{1} in $\Omega_{R,\alpha}$ with past condition~\eqref{initial} is blocked, in the sense of~\eqref{blocking}, is relatively closed in $(0,+\infty)\times(0,\pi/2)$.
\end{corollary}

We finally conjecture that, under the assumptions of Theorem~\ref{thm7}, the set of parameters~$(R,\alpha)$ for which the solution $u$ of~\eqref{1} in $\Omega_{R,\alpha}$ with past condition~\eqref{initial} propagates completely is actually convex in both variables $R$ and $\alpha$, and that this property is stable by making $\alpha$ decrease or $R$ increase. This conjecture can be formulated as follows.

\begin{conjecture}
Assume that the functions $h$ given in~\eqref{defOmega}-\eqref{h} depend continuously on the parameters $(R,\alpha)\in(0,+\infty)\times(0,\pi/2)$ in the $C^{2,\beta}_{loc}(\R)$ sense, with $0<\beta<1$. We say that complete propagation $($resp. blocking$)$ holds in $\Omega_{R,\alpha}$ if the solution $u$ of~\eqref{1} in $\Omega_{R,\alpha}$ with past condition~\eqref{initial} satisfies~\eqref{complete} $($resp.~\eqref{blocking}$)$. Then,
\begin{itemize}
\item for every $R>0$, there is $\alpha_R\in(0,\pi/2]$ such that complete propagation holds in $\Omega_{R,\alpha}$ for all $\alpha\in(0,\alpha_R)$, and blocking holds for all $\alpha\in[\alpha_R,\pi/2)$ if $\alpha_R<\pi/2$;
\item for every $\alpha\in[0,\pi/2)$, there is $\rho_\alpha\in[0,+\infty)$ such that complete propagation holds in $\Omega_{R,\alpha}$ for all $R>\rho_\alpha$, and blocking holds for all $R\in(0,\rho_\alpha]$ if $\rho_\alpha>0$;
\end{itemize}
\end{conjecture}

From Theorem~\ref{thm5} one knows that $\alpha_R$ exists and $\alpha_R=\pi/2$ when $R\ge R_0$ (with the notations of Theorem~\ref{thm5}). Furthemore, $\rho_0$ exists and $\rho_0=0$. On the other hand, Theorem~\ref{thm6} implies that, in dimension $N\ge3$, for any given $\alpha_*\in(0,\pi/2)$ and $L_*>0$, the angle $\alpha_R$, if any, satisfies $\alpha_R\le\alpha_*$ when $R\in(0,R_*]$ (with the notations of Theorem~\ref{thm6}), and that $\rho_\alpha$, if any, satisfies $\rho_\alpha\ge R_*$ when $\alpha\in[\alpha_*,\pi/2)$.

\vskip 0.3cm
\noindent\textbf{Outline of the paper.} This article is organized as follows. The proof of Proposition~\ref{thm1} on the existence and uniqueness of the entire solution $u$ emanating from the planar front in the left part of a given domain $\Omega$ satisfying~\eqref{defOmega}-\eqref{h} is sketched in Section~\ref{Sec-2}. The proof of Theorem~\ref{thm2} on the dichotomy between complete propagation and blocking is shown in Section~\ref{Sec-3}, as are various Liouville type results for the solutions of~\eqref{u-infty} in funnel-shaped domains and in the whole space. Section~\ref{Sec-4} is devoted to the proof of Theorem~\ref{thm3} on the spreading properties in case of complete propagation. The immediate proof of Theorem~\ref{thm4} is also done in Section~\ref{Sec-4}. In Section~\ref{Sec-5}, we prove Theorem~\ref{thm5} on the existence of a threshold~$R_0>0$ such that the solution~$u$ propagates completely if $R\ge R_0$, and Theorem~\ref{thm6} on blocking when $R$ is small enough and~$\alpha$ is not too small, by constructing a suitable stable non-constant stationary solution of~\eqref{1}. Lastly, Section~\ref{Sec-6} is devoted to the proof of Theorem~\ref{thm7} on the openness of the set of parameters for which complete propagation holds. 


\section{Existence and uniqueness for problem~\eqref{1} with past condition~\eqref{initial}}\label{Sec-2}

This section is devoted to the sketch of the proof of Proposition~\ref{thm1} on the well-posedness of problem~\eqref{1} in $\overline{\Omega}$ with the past condition~\eqref{initial}, for any given $R>0$ and $\alpha\in[0,\pi/2)$. From the construction of the solution of~\eqref{1} and~\eqref{initial}, we also deduce another comparison result which will be used later in the proof of Theorem~\ref{thm7}. The proof of Proposition~\ref{thm1} is inspired from \cite{BHM2009,BBC2016,E2018,P2018}, so we just sketch it here. However, some important elements of the construction of the solution $u$ to~\eqref{1} satisfying~\eqref{initial} and several auxiliary estimates are pointed out since they will be used in the proofs of other main results in the following sections.

The main steps of the proof of Proposition~\ref{thm1} are the following:
\begin{itemize}
\item for $\mu^*>0$ defined as in~\eqref{phi}, there exist $M>0$ and
$$T'\le T:=\frac{1}{\mu^*c}\,\ln\frac{c}{c+M}<0$$
such that the function $w^-$ defined in $(-\infty,T]\times\overline{\Omega}$ by:
\be\label{defw-}
w^-(t,x)=\begin{cases}
\phi(x_1-ct+\xi(t))-\phi(-x_1-ct+\xi(t)), & t\le T,\ x\in\overline{\Omega}\hbox{ with }x_1<0,\vspace{3pt}\\
0, & t\le T,\ x\in\overline{\Omega}\hbox{ with }x_1\ge 0,
\end{cases}
\ee
with $\xi(t)=(1/\mu^*)\ln(c/(c-M\,e^{\mu^*ct}))$, is a generalized subsolution of~\eqref{1} in $(-\infty,T']\times\overline{\Omega}$, and it satisfies~\eqref{initial} (notice that $\xi(-\infty)=0$); furthermore, the real numbers $M$, $T$ and~$T'$ can be chosen independently of $R>0$ and $\alpha\in[0,\pi/2)$ (these coefficients depend on~$f$ and~$\phi$ only, and thus actually on $f$ only);
\item thanks to~\eqref{defOmega}-\eqref{h}, the function $w^+$ defined in $\R\times\overline{\Omega}$ by
\be\label{defw+}
w^+(t,x)=\phi(x_1-ct)
\ee
is a supersolution of~\eqref{1} in $\R\times\overline{\Omega}$, and it satisfies~\eqref{initial} and $w^-\le w^+$ in $(-\infty,T]\times\overline{\Omega}$, since $\phi$ is positive decreasing and $\xi>0$ in $(-\infty,T]$;
\item for each $n\in\N$ with $n>-T'$, let $u_n$ be the solution of the Cauchy problem associated to~\eqref{1} in $(-n,+\infty)\times\overline{\Omega}$, with initial (at time $-n$) condition defined by
\be\label{defunn}
u_n(-n,x)=\sup_{s\le-n}w^-(s,x)\,\in[0,1],\ \hbox{ for }x\in\overline{\Omega};
\ee
each function $u_n(-n,\cdot)$ only depends on $x_1$ and, from the strong parabolic maximum principle and the well-posedness of this Cauchy problem and the axisymmetry of $\Omega$ with respect to the $x_1$-axis, one has $0<u_n<1$ in $(-n,+\infty)\times\overline{\Omega}$ and, for each $t\ge-n$, $u_n(t,\cdot)$ is axisymmetric with respect to the $x_1$-axis, that is, it depends only on $x_1$ and $|x'|$, with $x'=(x_2,\cdots,x_N)$; furthermore, the maximum principle again and the fact that $w^-$ is a subsolution in $(-\infty,T']\times\overline{\Omega}$, imply that $u_n(t,\cdot)\ge u_n(-n,\cdot)$ in $\overline{\Omega}$ for all $t\in[-n,T']$, hence $u_n$ is non-decreasing with respect to the variable $t$ in $[-n,+\infty)\times\overline{\Omega}$;
\item the maximum principle also implies that $u_{n+1}(-n,\cdot)\ge u_n(-n,\cdot)$ in $\overline{\Omega}$ for each $n>-T'$, hence $u_{n+1}\ge u_n$ in $[-n,+\infty)\times\overline{\Omega}$ for each $n>-T'$; from standard parabolic estimates, the functions $u_n$ converge in $C^{1,2}_{(t,x);loc}(\R\times\overline{\Omega})$ to a classical solution $u$ of~\eqref{1} such that $0\le u\le 1$ and $u_t\ge0$ in $\R\times\overline{\Omega}$; furthermore, for each $t\in\R$, the function $u(t,\cdot)$ is axisymmetric with respect to the $x_1$-axis;
\item one has $1\ge u_n(-n,\cdot)\ge w^-(-n,\cdot)$ in $\overline{\Omega}$ for each $n>-T'$, hence $1\ge u_n(t,\cdot)\ge w^-(t,\cdot)$ in $\overline{\Omega}$ for all $t\in[-n,T']$ and
\be\label{uw}
1\ge u(t,\cdot)\ge w^-(t,\cdot)\ \hbox{ in $\overline{\Omega}$ for all $t\le T'$};
\ee
\item since $w^-\le w^+$ in $(-\infty,T]\times\overline{\Omega}$ and $w^+$ is increasing with respect to the variable $t$, one has $u_n(-n,\cdot)\le w^+(-n,\cdot)$ in $\overline{\Omega}$ for each $n>-T'$, hence $u_n(t,\cdot)\le w^+(t,\cdot)$ in $\overline{\Omega}$ for all $t\ge-n$, and
\be\label{inequw+}
u(t,x)\le w^+(t,x)=\phi(x_1-ct)\ \hbox{ for all $(t,x)\in\R\times\overline{\Omega}$};
\ee
\item from the inequalities $w^-\le u\le w^+$ in $(-\infty,T']\times\overline{\Omega}$, the past condition~\eqref{initial} follows immediately; one also gets that
$$0<u<1\ \hbox{ and }\ u_t>0\ \hbox{ in }\R\times\overline{\Omega}$$
from the strong parabolic maximum principle;
\item from standard parabolic estimates and the monotonicity in $t$, one has $u(t,\cdot)\to u_\infty$ as $t\to+\infty$ in $C^2_{loc}(\overline{\Omega})$, and $0<u_\infty\le1$ solves~\eqref{u-infty}; furthermore, $1\ge u_\infty(x)>u(t,x)\ge w^-(t,x)$ for all $(t,x)\in(-\infty,T']\times\overline{\Omega}$; in particular,
$$1\ge u_\infty(x)\ge w^-(T',x)=\phi(x_1-cT'+\xi(T'))-\phi(-x_1-cT'+\xi(T'))$$
for all $x\in\overline{\Omega}$ with $x_1<0$; since $\xi$ and $T'$ do not depend on $R>0$ and $\alpha\in[0,\pi/2)$, one gets that
\be\label{cv1unif}
u_\infty(x)\to1\hbox{ as $x_1\to-\infty$ uniformly with respect to $R>0$ and $\alpha\in[0,\pi/2)$};
\ee
\item for each $\eta\in(0,1/2)$, the past condition~\eqref{initial} and the monotonicity of $u$ in $t$, together with the strong parabolic maximum principle, yield $\liminf_{t\to-\infty,\,u(t,x)\in[\eta,1-\eta]}u_t(t,x)>0$;
\item for any solution $v$ of~\eqref{1} satisfying~\eqref{initial}, there are $\beta>0$ and $\sigma>0$ such that, for every $\epsilon>0$ small enough, there is $T_\epsilon<0$ such that $v<u+\epsilon$ in $(-\infty,T_\epsilon]\times\overline{\Omega}$ and the function
$$(t,x)\mapsto\min\big(u(t+\sigma\epsilon(1-e^{-\beta(t-t_0)}),x)+\epsilon e^{-\beta(t-t_0)},1\big)$$
is a supersolution of~\eqref{1} in $[t_0,T_\epsilon]\times\overline{\Omega}$ for all $t_0<T_\epsilon$; as this supersolution is larger than $v$ at time $t_0$, with any $t_0<T_\epsilon$, so is it in $[t_0,T_\epsilon]\times\overline{\Omega}$, hence $u(t+\sigma\epsilon,x)\ge v(t,x)$ in $(-\infty,T_\epsilon]\times\overline{\Omega}$ at the limit $t_0\to-\infty$, and finally $v\le u(\cdot+\sigma\epsilon,\cdot)$ in $\R\times\overline{\Omega}$ from the comparison principle; since this holds for all $\epsilon>0$ small enough, one gets $v\le u$ in $\R\times\overline{\Omega}$; similarly, the inequality $v\ge u$ holds, leading to the uniqueness for problem~\eqref{1} with the past condition~\eqref{initial}.
\end{itemize}
This completes the proof of Proposition~\ref{thm1}.\hfill$\Box$\break

From the proof of Proposition~\ref{thm1}, an important corollary follows, that will be used later in the proof of Theorem~\ref{thm7}.

\begin{corollary}\label{cor2}
For any $R>0$ and $\alpha\in[0,\pi/2)$, let $u$ be the solution of~\eqref{1} and~\eqref{initial} given in Proposition~$\ref{thm1}$. If there is a $C^2(\overline{\Omega})$ solution $U$ of the elliptic problem~\eqref{u-infty} such that $0<U\le1$ in $\overline{\Omega}$ and $U(x)\to1$ as $x_1\to-\infty$, then $u(t,x)\le U(x)$ for all $(t,x)\in\R\times\overline{\Omega}$.
\end{corollary}

\begin{proof} We recall that $f(1)=0$, $f'(1)<0$, and $f$ is extended by $f(s)=f'(1)(s-1)$ for $s>1$. Let $\delta>0$ be such that $f'<0$ in $[1-\delta,+\infty)$, and let $A>0$ be such that
\be\label{defA}
1-\delta\le U(x)\le1\ \hbox{ for all $x\in\overline{\Omega}$ with $x_1\le-A$}.
\ee
Since $U$ is positive and continuous in $\overline{\Omega}$, it follows from the definitions of $\Omega$ and $w^-$ in~\eqref{defOmega}-\eqref{h} and~\eqref{defw-} that there exists $T_1\in(-\infty,T']\subset(-\infty,T]\subset(-\infty,0)$ such that
\be\label{defT1}
w^-(t,x)\le U(x)\ \hbox{ for all $t\le T_1$ and $x\in\overline{\Omega}$ with $x_1\ge-A$}.
\ee

We now claim that $w^-(t,x)\le U(x)$ for all $t\le T_1$ and $x\in\overline{\Omega}$, an inequality that will easily lead to the desired conclusion. To show this inequality, define
$$\epsilon^*=\min\big\{\epsilon\ge0:w^-(t,x)\le U(x)+\epsilon\hbox{ for all $t\le T_1$ and $x\in\overline{\Omega}$}\big\}.$$
Since $w^-$ and $U$ are globally bounded and continuous, $\epsilon^*$ is a well-defined nonnegative real number, and one has $w^-(t,x)\le U(x)+\epsilon^*$ for all $t\le T_1$ and $x\in\overline{\Omega}$. One shall show that $\epsilon^*=0$. Assume by way of contradiction that $\epsilon^*>0$. Notice that $w^-(t,\cdot)\to0$ as $t\to-\infty$ locally uniformly in $\overline{\Omega}$, and remember that $w^-\le1$ in $(-\infty,T_1]\times\overline{\Omega}$ and $U>0$ in $\overline{\Omega}$ with $\lim_{x_1\to-\infty}U(x)=1$. It  then follows from~\eqref{defT1} and the definition of $\epsilon^*$ that there is $(t^*,x^*)\in(-\infty,T_1]\times\overline{\Omega}$ with $x^*_1<-A$ such that
$$w^-(t^*,x^*)=U(x^*)+\epsilon^*.$$
But the function $U+\epsilon^*$ is a supersolution of~\eqref{u-infty} in $\{x\in\overline{\Omega}:x_1\le-A\}$, owing to~\eqref{defA} and the definitions of $\delta$ and $A$ (one has $f(U(x)+\epsilon^*)\le f(U(x))$ for all $x\in\overline{\Omega}$ with $x_1\le-A$). On the other hand, the function $w^-$ is a generalized subsolution of~\eqref{1} in $(-\infty,T_1]\times\overline{\Omega}$ (remember that $T_1\le T'$). The strong parabolic maximum principle (namely, the interior version if $x^*\in\Omega$ with $x^*_1<-A$, or the strong parabolic Hopf lemma if $x^*\in\partial\Omega$, still with $x^*_1<-A$) then imply that $w^-(t,x)=U(x)+\epsilon^*$ for all $t\le t^*$ and $x\in\overline{\Omega}$ with $x_1\le-A$. This is clearly ruled out, since $w^-\le1$ and $U(x)+\epsilon^*\to1+\epsilon^*>1$ as $x_1\to-\infty$. Therefore, $\epsilon^*=0$, hence
$$w^-(t,x)\le U(x)\ \hbox{ for all $t\le T_1$ and $x\in\overline{\Omega}$}.$$
In particular, owing to~\eqref{defunn}, there holds $u_n(-n,\cdot)=\sup_{s\le-n}w^-(s,\cdot)\le U$ in $\overline{\Omega}$ for all $n\in\N$ with $n\ge-T_1$. Hence, from the parabolic maximum principle, one has $u_n(t,\cdot)\le U$ in $\overline{\Omega}$ for all $n\in\N$ with $n\ge-T_1$ and for all $t\ge-n$. Therefore, $u(t,\cdot)\le U$ in $\overline{\Omega}$ for all $t\in\R$, which is the desired conclusion.
\end{proof}


\section{Dichotomy between complete propagation and blocking: proof of Theorem~\ref{thm2}}
\label{Sec-3}

This section is devoted to the proof of the dichotomy between complete propagation and blo\-cking for the solutions $u$ of~\eqref{1} and~\eqref{initial} constructed in Proposition~\ref{thm1}, for any given $R>0$ and $\alpha\in[0,\pi/2)$. The proof of this dichotomy relies itself on several Liouville type results of independent interest for the solutions of elliptic equations $\Delta U+f(U)=0$ in certain domains of~$\R^N$. We start in Section~\ref{Sec-3.1} with Liouville type results for~\eqref{u-infty} in funnel-shaped domains~$\Omega$, and we then continue in Section~\ref{Sec-3.2} with such results for stable solutions of $\Delta U+f(U)=0$ in the plane, a half-plane and the whole space. Theorem~\ref{thm2} is finally proved in Section~\ref{Sec-3.3}.


\subsection{Auxiliary Liouville type results for~\eqref{u-infty} in funnel-shaped domains}\label{Sec-3.1}

The first two auxiliary Liouville type results used in the proof of Theorem~\ref{thm2}, as well as in other main results, are Lemmas~\ref{lemliouville0} and~\ref{lemliouville1} below for the solutions~$u_\infty$ of~\eqref{u-infty} in funnel-shaped domains $\Omega$. They rely themselves on the existence of some not-too-small solutions of the same equation in large balls with Dirichlet boundary conditions. In the sequel, we call~$B_r(x)$ the open Euclidean ball of center $x\in\R^N$ and radius $r>0$, and we denote~$B_r:=B_r(0)$.

\begin{lemma}\label{lemsub}
There are $R_0>0$ and a $C^2(\overline{B_{R_0}})$ solution $\psi$ of the semilinear elliptic equation
\begin{align}
\label{sub sol in ball}
\begin{cases}
\Delta \psi+f(\psi)=0 &\text{in}~\overline{B_{R_0}},\\
0\le \psi<1 &\text{in}~ \overline{B_{R_0}},\\
\psi=0 &\text{on}~\partial B_{R_0},\\
\displaystyle\mathop{\max}_{\overline{B_{R_0}}}\,\psi=\psi(0)>\theta.
\end{cases}
\end{align}
\end{lemma}

\begin{proof}
The proof is standard and is therefore omitted. In short, it can be done by using variational arguments (see e.g.~\cite[Theorem~A]{BL1980} and~\cite[Problem~(2.25)]{GHS2020}): such a solution~$\psi$ is obtained as a minimizer in $H^1_0(B_{R_0})$ of the functional $\varphi\mapsto\int_{B_{R_0}}\!\!|\nabla\varphi|^2/2-\int_{B_{R_0}}\!\!F(\varphi)$, with~$F'\!=\!f$. Furthermore, such a minimizer is radially symmetric and decreasing in $|x|$ as soon as it is not identically $0$ (see~\cite{GNN1979}), and its maximal value, which is the value at the origin, converges to $1$ as the radius of the ball converges to $+\infty$, thanks to~\eqref{f-bistable-1}-\eqref{f-bistable-2}.
\end{proof}

In Proposition~\ref{thm1}, the constructed solutions $u$ of~\eqref{1} and~\eqref{initial} converge as $t\to+\infty$ to a stationary solution $u_\infty$ of~\eqref{u-infty}. By construction, $u_\infty$ satisfies $0<u_\infty\le1$ in $\overline{\Omega}$, and $u_\infty(x)\to1$ as $x_1\to-\infty$ (and this limit actually holds uniformly with respect to the parameters $(R,\alpha)$). But this limit is not enough to guarantee that $u_\infty\equiv1$ in $\overline{\Omega}$ in general: Theorems~\ref{thm5} and~\ref{thm6} provide some conditions for $u_\infty$ to be equal to $1$ or not, according to the values of~$R$ and~$\alpha$. We now prove in the next result (which will be used in the proof of Theorem~\ref{thm2}) that, whatever~$R$ and $\alpha$ may be, if $u_\infty(x)$ is assumed to converge to $1$ (or is assumed to be not too small) as~$x_1\to+\infty$, then $u_\infty$ is identically equal to $1$.

\begin{lemma}\label{lemliouville0}
Let $\Omega$ be a funnel-shaped domain satisfying~\eqref{defOmega}-\eqref{h}, and let $0<u_\infty\le1$ be a solution of~\eqref{u-infty} in $\overline{\Omega}$. If $\liminf_{x\in\overline{\Omega},\,x_1\to+\infty}u_\infty(x)>\theta$ with $\theta$ as in~\eqref{f-bistable-2}, then $u_\infty\equiv1$ in $\overline{\Omega}$.
\end{lemma}

\begin{proof}
First of all, if $(x^n)_{n\in\N}=(x^n_1,x^n_2,\cdots,x^n_N)_{n\in\N}=(x^n_1,(x^n)')_{n\in\N}$ is any sequence in $\overline{\Omega}$ such that $x^n_1\to+\infty$ and $h(x^n_1)-|(x^n)'|\to+\infty$ as $n\to+\infty$, then, from standard elliptic estimates, the functions $x\mapsto u_\infty(x+x^n)$ converge in $C^2_{loc}(\R^N)$, up to extraction of a subsequence, to a~$C^2(\R^N)$ solution $U$ of
$$\Delta U+f(U)=0\ \hbox{ in $\R^N$},$$
such that $\theta<\inf_{\R^N}U\le\sup_{\R^N}U\le1$. Since $f(1)=0$ and $f>0$ in $(\theta,1)$, it then easily follows that $U\equiv1$ in $\R^N$. Similarly, if $(x^n)_{n\in\N}$ is any sequence in $\overline{\Omega}$ such that $x^n_1\to+\infty$ and $\limsup_{n\to+\infty}(h(x^n_1)-|(x^n)'|)<+\infty$, then there are an open half-space $H$ of $\R^N$ and a $C^2(\overline{H})$ function $U$ such that, up to extraction of a subsequence, $\|u_\infty(\cdot+x^n)-U\|_{C^2(K\cap(\overline{\Omega}-x^n))}\to0$ as~$n\to+\infty$ for every compact set $K\subset\overline{H}$. Hence, $U$ obeys $\Delta U+f(U)=0$ in $\overline{H}$ and $\nu\cdot\nabla U=0$ on $\partial H$, together with $\theta<\inf_{H}U\le\sup_{H}U\le1$. As above, one infers that $U\equiv1$ in $\overline{H}$. From the previous observations, it follows that
$$u_\infty(x)\to1\hbox{ as $x_1\to+\infty$}$$
with $x\in\overline{\Omega}$, that is, uniformly with respect to the variables $(x_2,\cdots,x_N)$.

Now, from~\eqref{f-bistable-1}-\eqref{f-bistable-2} and the affine $C^1$ extension of $f$ outside $[0,1]$, there is $\epsilon>0$ small enough such that the $C^1(\R)$ function $f-\epsilon$ satisfies $f(\alpha_\epsilon)-\epsilon=f(\theta_\epsilon)-\epsilon=f(\beta_\epsilon)-\epsilon=0$ for some
$$\alpha_\epsilon<0<\theta<\theta_\epsilon<\beta_\epsilon<1,$$
with $f'(\alpha_\epsilon)<0$, $f'(\theta_\epsilon)>0$, $f'(\beta_\epsilon)<0$, $f-\epsilon<0$ in $(\alpha_\epsilon,\theta_\epsilon)$, $f-\epsilon>0$ in $(\theta_\epsilon,\beta_\epsilon)$, and $\int_{\alpha_\epsilon}^{\beta_\epsilon}(f-\epsilon)>0$. Therefore, there exist $c_\epsilon>0$ and a $C^2(\R)$ function $\phi_\epsilon:\R\to(\alpha_\epsilon,\beta_\epsilon)$ solving
$$\phi_\epsilon''+c_\epsilon\phi_\epsilon'+f(\phi_\epsilon)-\epsilon=0\hbox{ in $\R$},\ \hbox{ and }\ \phi_\epsilon(-\infty)=\beta_\epsilon,\ \phi_\epsilon(+\infty)=\alpha_\epsilon.$$
Since $u_\infty$ is positive in $\overline{\Omega}$ and converges to $1$ as $x_1\to+\infty$, there exists $A>0$ such that $u_\infty(x)\ge\phi_\epsilon(-x_1+A)$ for all $x\in\overline{\Omega}$. But the function $\phi_\epsilon$ is decreasing in $\R$ and the function~$h$ in~\eqref{h} is nondecreasing. Hence, for each $t\in\R$, the function $x\mapsto\phi_\epsilon(-x_1-c_\epsilon t+A)$ has a nonpositive normal derivative at any point $x\in\partial\Omega$. Furthermore, $\underline{u}(t,x):=\phi_\epsilon(-x_1-c_\epsilon t+A)$ satisfies $\underline{u}_t=\Delta\underline{u}+f(\underline{u})-\epsilon\le\Delta\underline{u}+f(\underline{u})$ in $\R\times\R^N$ by definition of $\phi_\epsilon$. Remembering that $u_\infty\ge\underline{u}(0,\cdot)$ in $\overline{\Omega}$, the parabolic maximum principle then implies that $u_\infty\ge\underline{u}(t,\cdot)$ in $\overline{\Omega}$ for all~$t\ge0$. The limit as $t\to+\infty$ and the positivity of $c_\epsilon$ yield
$$u_\infty\ge\phi_\epsilon(-\infty)=\beta_\epsilon\ \hbox{ in $\overline{\Omega}$}.$$
Since $\theta<\beta_\epsilon\le u_\infty\le1$ and $f>0$ in $(\theta,1)$, one then gets that $u_\infty\equiv1$ in $\overline{\Omega}$, which is the desired conclusion.
\end{proof}

The next result, which can be viewed as a corollary of Lemmas~\ref{lemsub} and~\ref{lemliouville0}, will also be a key-ingredient in the proof of Theorems~\ref{thm5} and~\ref{thm7}. 

\begin{lemma}\label{lemliouville1}
Let $\Omega$ be a funnel-shaped domain satisfying~\eqref{defOmega}-\eqref{h}, and let $0<u_\infty\le1$ be a~$C^2(\overline{\Omega})$ solution of~\eqref{u-infty} in $\overline{\Omega}$. Let $R_0>0$ and $\psi\in C^2(\overline{B_{R_0}})$ be as in Lemma~$\ref{lemsub}$. If there is a point~$x_0\in\Omega$ such that $\overline{B_{R_0}(x_0)}\subset\overline{\Omega}$ and $u_\infty\ge\psi(\cdot-x_0)$ in $\overline{B_{R_0}(x_0)}$, then $u_\infty\equiv 1$ in $\overline{\Omega}$.
\end{lemma}

\begin{proof}
Write
$$x_0=(x_{0,1},x'_0)$$
with $x_{0,1}\in\R$ and $x'_0\in\R^{N-1}$. Owing to the properties~\eqref{defOmega}-\eqref{h} satisfied by $\Omega$, one has $\overline{B_{R_0}(x_{0,1},sx'_0)}\subset\overline{\Omega}$ for all $s\in[0,1]$. Since $\psi$ satisfies~\eqref{sub sol in ball} and vanishes on $\partial B_{R_0}$, and since the solution $u_\infty$ of~\eqref{u-infty} is positive in $\overline{\Omega}$, the strong maximum principle implies that $u_\infty>\psi(\cdot-x_0)$ in $\overline{B_{R_0}(x_0)}$ and, by continuity, $u_\infty>\psi(\cdot-(x_{0,1},sx'_0))$ in $\overline{B_{R_0}(x_{0,1},sx'_0)}$ for all $s\in[\eta,1]$, for some $0\le\eta<1$. We then claim that
\be\label{claimmp}
u_\infty>\psi(\cdot-(x_{0,1},sx'_0))\hbox{ in $\overline{B_{R_0}(x_{0,1},sx'_0)}$ for all $s\in[0,1]$}.
\ee
Indeed, otherwise, there exists $s^*\in[0,1)$ such that $u_\infty\ge\psi(\cdot-(x_{0,1},s^*x'_0))$ in $\overline{B_{R_0}(x_{0,1},s^*x'_0)}$ with equality at a point $x^*\in\overline{B_{R_0}(x_{0,1},s^*x'_0)}$. The point $x^*$ can not lie on the boun\-dary $\partial B_{R_0}(x_{0,1},s^*x'_0)$, since $\psi(\cdot-(x_{0,1},s^*x'_0))$ vanishes there whereas $u_\infty$ is positive. Hence, $x^*$ is in the open ball $B_{R_0}(x_{0,1},s^*x'_0)$ and the strong maximum principle yields $u_\infty\equiv\psi(\cdot-(x_{0,1},s^*x'_0))$ in~$\overline{B_{R_0}(x_{0,1},s^*x'_0)}$, which is impossible on $\partial B_{R_0}(x_{0,1},s^*x'_0)$. Therefore,~\eqref{claimmp} holds and, in particular, $u_\infty>\psi(\cdot-(x_{0,1},0))$ in $\overline{B_{R_0}(x_{0,1},0)}$.

Similarly, since $\overline{B_{R_0}(s,0)}\subset\overline{\Omega}$ for all $s\ge x_{0,1}$, one then infers that $u_\infty>\psi(\cdot-(s,0))$ in~$\overline{B_{R_0}(s,0)}$ for all $s\ge x_{0,1}$. Consider then any
$$s\ge\max(x_{0,1},L\cos\alpha+R_0),$$
with $L>R>0$ and $\alpha\in[0,\pi/2)$ given in~\eqref{h}, and any unit vector $e'$ of $\R^{N-1}$. For each $\sigma\in[0,h(s)]$, two cases may occur, owing to~\eqref{defOmega}-\eqref{h}:
$$\left\{\baa{l}
\hbox{either $\overline{B_{R_0}(s,\sigma e')}\subset\overline{\Omega}$},\vspace{3pt}\\
\hbox{or $B_{R_0}(s,\sigma e')\cap\partial\Omega\neq\emptyset$ and $\nu(x)\cdot(x-(s,\sigma e'))\ge0$ for every $x\in B_{R_0}(s,\sigma e')\cap\partial\Omega$,}\eaa\right.$$
where $\nu(x)$ denotes the outward unit normal to $\Omega$ at $x$. In the latter case, one then has $\nu(x)\cdot\nabla\psi(x-(s,\sigma e'))\le0$, since the function $y\mapsto\psi(y)$ is radially symmetric and nonincreasing with respect to $|y|$ in $\overline{B_{R_0}}$. In all cases, for each $\sigma\in[0,h(s)]$, the function $\psi(\cdot-(s,\sigma e'))$ is a subsolution of~\eqref{u-infty} in $\overline{B_{R_0}(s,\sigma e')\cap\Omega}$ (this closed set is actually equal to $\overline{B_{R_0}(s,\sigma e')}\cap\overline{\Omega}$ from the definition of $\Omega$, since $\sigma\in[0,h(s)]$), and the open set $B_{R_0}(s,\sigma e')\cap\Omega$ is connected and not empty, with its boundary meeting $\partial B_{R_0}(s,\sigma e')\cap\overline{\Omega}$. Hence, the function $\psi(\cdot-(s,\sigma e'))$ can not be identically equal to $u_\infty$ in $\overline{B_{R_0}(s,\sigma e')\cap\Omega}$. Since $u_\infty>\psi(\cdot-(s,0))$ in $\overline{B_{R_0}(s,0)}$, one then gets as in the previous paragraph, by sliding $\psi$ below $u_\infty$ in the direction $(0,e')$ and using the strong interior maximum principle and the Hopf lemma, that $u_\infty>\psi(\cdot-(s,\sigma e'))$ in~$\overline{B_{R_0}(s,\sigma e')}\cap\overline{\Omega}$ for all $\sigma\in[0,h(s)]$. As a consequence,
$$u_\infty(s,\sigma e')>\psi(0)\ \hbox{ for all $s\ge\max(x_{0,1},L\cos\alpha+R_0)$ and $\sigma\in[0,h(s)]$}.$$
Since this holds for every unit vector $e'$ of $\R^{N-1}$, one infers that $u_\infty(x)>\psi(0)$ for all $x\in\overline{\Omega}$ with $x_1\ge\max(x_{0,1},L\cos\alpha+R_0)$. Since $u_\infty\le1$ in $\overline{\Omega}$ together with $\psi(0)>\theta$, one concludes from Lemma~\ref{lemliouville0} that $u_\infty\equiv1$ in $\overline{\Omega}$.  The proof of Lemma~\ref{lemliouville1} is thereby complete.
\end{proof}


\subsection{Auxiliary Liouville type results for stable solutions of $\Delta U\!+\!f(U)\!=\!0$}\label{Sec-3.2}

The last three auxiliary results for the proof of Theorem~\ref{thm2} are still Liouville type results for semilinear elliptic equations $\Delta U+f(U)=0$ in $\overline{\omega}$. But these results, of independent interest, deal with other geometric configurations: $\omega$ will be the two-dimensional plane, or a two-dimensional half-plane, or the whole space $\R^N$. In all these statements, we are concerned with stable solutions, in the sense of~\eqref{stable}.

\begin{proposition}\label{proliouville2}
Let $0\le U\le 1$ be a $C^2(\R^2)$ stable solution of $\Delta U+f(U)=0$ in $\R^2$. Then, either~$U\equiv0$ in $\R^2$ or $U\equiv 1$ in $\R^2$.
\end{proposition}

\begin{proof}
The proof uses some properties of the principal eigenvalues of some elliptic operators, together with some results of~\cite{BCN1997}. First of all, since $f\in C^{1,1}([0,1])$, standard elliptic estimates imply that $U$ is of class $C^3(\R^2)$ and has bounded partial derivatives up to the third order.

Now, for any $R>0$, let
\be\label{lambdaR}
\lambda(-\Delta-f'(U),B_R)=\min_{\psi\in H^1_0(B_R),\,\|\psi\|_{L^2(B_R)}=1}\int_{B_R}|\nabla\psi|^2-f'(U)\psi^2
\ee
and
$$\lambda(-\Delta,B_R)=\min_{\psi\in H^1_0(B_R),\,\|\psi\|_{L^2(B_R)}=1}\int_{B_R}|\nabla\psi|^2$$
be the principal eigenvalues of the operators $-\Delta-f'(U)$ and $-\Delta$ in $B_R$ (the two-dimensional Euclidean disc) with Dirichlet boundary conditions on $\partial B_R$. One has $\lambda(-\Delta-f'(U),B_R)\ge0$ by assumption, and
$$\lambda(-\Delta-f'(U),B_R)\le\max_{[0,1]}|f'|+\lambda(-\Delta,B_R)=\max_{[0,1]}|f'|+\frac{\lambda(-\Delta,B_1)}{R^2}.$$
Hence $\sup_{R\ge1}|\lambda(-\Delta-f'(U),B_R)|<+\infty$. Furthermore, the map $R\mapsto\lambda(-\Delta-f'(U),B_R)$ is nonincreasing (and even actually decreasing) in $(0,+\infty)$, and there exists
$$\lambda_\infty=\lim_{R\to+\infty}\lambda(-\Delta-f'(U),B_R)\in[0,+\infty).$$
Notice also that the map $x\mapsto f'(U(x))$ is Lipschitz continuous from the $C^{1,1}$ regularity of $f$ and the Lipschitz continuity of $U$. For each $n\in\N$ with $n\ge1$, there exists a unique principal eigenfunction $\varphi_n\in C^2(\overline{B_n})$ solving
$$-\Delta\varphi_n-f'(U)\varphi_n=\lambda(-\Delta-f'(U),B_n)\varphi_n\ \hbox{ in $\overline{B_n}$},$$
with $\varphi_n=0$ on $\partial B_n$, $\varphi_n>0$ in $B_n$ and $\varphi_n(0)=1$. The Harnack inequality and standard elliptic estimates then imply that, up to extraction of a subsequence, the functions $\varphi_n$ converge in $C^2_{loc}(\R^2)$ to a positive function $\varphi$ solving $-\Delta\varphi-f'(U)\varphi=\lambda_\infty\varphi\ge0$ in $\R^2$ (together with $\varphi(0)=1$). Since the space dimension is here equal to $2$, and since each function $e\cdot\nabla U$ (with a unit  vector $e$ of $\R^2$) is bounded in $\R^2$ and solves
$$\Delta(e\cdot\nabla U)+f'(U)(e\cdot\nabla U)=0\ \hbox{ in $\R^2$},$$
it follows from~\cite[Theorem~1.8]{BCN1997} that $e\cdot\nabla U\equiv C_e\varphi$ in $\R^2$ for some real number $C_e$. In particular, each partial derivative $e\cdot\nabla U$ is either identically $0$ or has a strict constant sign in~$\R^2$. As a consequence, either the function $U$ is constant, or it depends on one variable only and it is strictly monotone in that variable.

If $U$ is constant, it may be equal to $0$, $\theta$ or $1$, from~\eqref{f-bistable-1}-\eqref{f-bistable-2}. However, if $U$ were equal to~$\theta$, then
$$0\le\lambda(-\Delta-f'(U),B_R)=-f'(\theta)+\lambda(-\Delta,B_R)=-f'(\theta)+R^{-2}\lambda(-\Delta,B_1)\mathop{\longrightarrow}_{R\to+\infty}-f'(\theta)<0,$$
a contradiction. Thus, if $U$ is constant, then either $U\equiv 0$ or $U\equiv1$ in $\R^2$.

If $U$ were one-dimensional and strictly monotone, that is $U(x)=V(x\cdot e)$ for some unit vector~$e$ and~$V$ increasing in $\R$, then $V$ would solve $V''+f(V)=0$ in $\R$ with $(V(-\infty),V(+\infty))\in\{(0,\theta),(0,1),(\theta,1)\}$, but the integration of this equation against $V'$ over $\R$ would lead to $\int_{V(-\infty)}^{V(+\infty)}f(s)ds=0$, contradicting~\eqref{f-bistable-1}-\eqref{f-bistable-2}. Thus, this monotone one-dimensional case is ruled out.

As a conclusion, one has shown that $U$ is constant in $\R^2$, and identically equal to $0$ or $1$. The proof of Proposition~\ref{proliouville2} is thereby complete.
\end{proof}

From Proposition~\ref{proliouville2}, the following analogue in a half-plane easily follows.

\begin{proposition}\label{proliouville2bis}
Let $H$ be an open half-plane and let $0\le U\le 1$ be a $C^2(\overline{H})$ stable solution of~$\Delta U+f(U)=0$ in $\overline{H}$ with Neumann boundary condition $\nu\cdot\nabla U=0$ on $\partial H$. Then, either~$U\equiv0$ in $\overline{H}$ or $U\equiv 1$ in $\overline{H}$.
\end{proposition}

\begin{proof}
Up to translation and rotation, one can assume that $H=\{(x_1,x_2)\in\R^2:x_2<0\}$ without loss of generality. Thus, $\partial H=\R\times\{0\}$ and $\partial_{x_2}U(x_1,0)=0$ for all $x_1\in\R$. Consider now the function $V$ in $\R^2$ defined by
$$V(x_1,x_2)=\left\{\baa{ll} U(x_1,x_2) & \hbox{if }x_2\le0,\vspace{3pt}\\
U(x_1,-x_2) & \hbox{if }x_2>0.\eaa\right.$$
It is of class $C^2(\R^2)$ and it solves $\Delta V+f(V)=0$ in $\R^2$, together with $0\le V\le1$ in $\R^2$. Furthermore, for any $\psi\in C^1(\R^2)$ with compact support, one has
$$\baa{rcl}
\displaystyle\int_{\R^2}\big(|\nabla\psi|^2-f'(V)\psi^2\big) & = & \displaystyle\int_{\{x_2<0\}}\big(|\nabla\psi|^2-f'(U)\psi^2\big)+\int_{\{x_2>0\}}\big(|\nabla\psi|^2-f'(V)\psi^2\big)\vspace{3pt}\\
& = & \displaystyle\int_H\big(|\nabla\psi|^2-f'(U)\psi^2\big)+\int_H\big(|\nabla\tilde{\psi}|^2-f'(U)\tilde{\psi}^2\big),\eaa$$
where $\tilde{\psi}(x_1,x_2)=\psi(x_1,-x_2)$. But the restrictions of the functions $\psi$ and $\tilde{\psi}$ in $\overline{H}$ are of class~$C^1(\overline{H})$ with compact support in $\overline{H}$. Therefore, the two terms of the right-hand side of the previous formula are nonnegative by assumption. Hence, 
$$\int_{\R^2}\big(|\nabla\psi|^2-f'(V)\psi^2\big)\ge0$$
for any $\psi\in C^1(\R^2)$ with compact support. Proposition~\ref{proliouville2} implies that $V$ is identically equal to either $0$ or $1$ in $\R^2$, which leads to the desired conclusion for $U$ in $\overline{H}$.
\end{proof}

The last Liouville type result is Proposition~\ref{proliouville}, which was stated in Section~\ref{Sec-1.3}. It is concerned with stable axisymmetric solutions in $\R^N$, and it also follows from Proposition~\ref{proliouville2}, as well as from some arguments inspired by~\cite{BCN1997}.

\begin{proof}[Proof of Proposition~$\ref{proliouville}$]
Throughout the proof, $U:\R^N\to[0,1]$ is a stable $C^2(\R^N)$ solution of $\Delta U+f(U)=0$ in $\R^N$, which is axisymmetric with respect to the $x_1$-axis. Let us first show that
\be\label{claimunif}
\hbox{either $U(x_1,x')\to0\ $ or $\ U(x_1,x')\to1\ $ as $|x'|\to+\infty,\ $ uniformly in $x_1\in\R$.}
\ee
To show this property, since $U$ is continuous and axisymmetric with respect to the $x_1$-axis, it is sufficient to show that, for any sequence
$$(x^n)_{n\in\N}=((x^n_1,x^n_2,0,\cdots,0))_{n\in\N}$$
in $\R^N$ such that $x^n_2\to+\infty$ as $n\to+\infty$, one has, up to extraction of a subsequence, either $U(x^n)\to0$ or $U(x^n)\to1$. Consider any such sequence $(x^n)_{n\in\N}$. Up to extraction of a subsequence, the functions $U(\cdot+x^n)$ converge in $C^2_{loc}(\R^N)$ to a solution $U_\infty$ of $\Delta U_\infty+f(U_\infty)=0$ in $\R^N$, with~$0\le U_\infty\le1$ in $\R^N$. Furthermore, since $U$ is axisymmetric with respect to the $x_1$-axis, and $x^n=(x^n_1,x^n_2,0,\cdots,0)$ with $x^n_2\to+\infty$, there is a $C^2(\R^2)$ function $V_\infty$ such that~$U_\infty(x)=V_\infty(x_1,x_2)$ for all $x\in\R^N$. Notice that $V_\infty$ then obeys $\Delta V_\infty+f(V_\infty)=0$ in $\R^2$. Let us now show that~$V_\infty$ is stable, in the sense of~\eqref{stable} with $\omega=\R^2$. Consider any $C^1(\R^2)$ function $\psi$ with compact support. For $n\in\N$, let us define
$$\psi_n(x)=\psi(x_1-x^n_1,|x'|-x^n_2)$$
for $x=(x_1,x')\in\R^N$. Since $\psi$ is compactly supported in $\R^2$ and since $x^n_2\to+\infty$ as $n\to+\infty$, the function $\psi_n$ is of class $C^1(\R^N)$ with compact support for all $n$ large enough. Together with the semistability of the solution $U$ of $\Delta U+f(U)=0$ in $\R^N$, one gets that, for all $n$ large enough,
$$\int_{\R^N}|\nabla\psi_n|^2-f'(U)\psi_n^2\ge0.$$
But since both $U$ and $\psi_n$ are axisymmetric with respect to the $x_1$-axis, the above inequality means that, for all $n$ large enough,
$$\int_{\R^2}\Big[|\nabla\psi(x_1-x^n_1,x_2-x^n_2)|^2-f'(U(x_1,x_2,0,\cdots,0))\psi(x_1-x^n_1,x_2-x^n_2)^2\Big]\,dx_1dx_2\ge0,$$
that is,
$$\int_{\R^2}\Big[|\nabla\psi(x_1,x_2)|^2-f'(U(x_1+x^n_1,x_2+x^n_2,0,\cdots,0))\psi(x_1,x_2)^2\Big]\,dx_1dx_2\ge0.$$
Since $U(x_1+x^n_1,x_2+x^n_2,0,\cdots,0)\to U_\infty(x_1,x_2,0,\cdots,0)=V_\infty(x_1,x_2)$
locally uniformly in~$(x_1,x_2)\in\R^2$ as $n\to+\infty$, since $\psi$ has a compact support, and since $f$ is of class $C^1(\R)$, one gets that
$$\int_{\R^2}|\nabla\psi|^2-f'(V_\infty)\psi^2\ge0.$$
As a consequence, the $C^2(\R^2)$ function $V_\infty$ is a stable solution of $\Delta V_\infty+f(V_\infty)=0$ in $\R^2$ such that $0\le V_\infty\le1$ in $\R^2$. Proposition~\ref{proliouville2} then implies that either $V_\infty\equiv0$ in $\R^2$, or $V_\infty\equiv1$ in~$\R^2$. In particular, either $U(x^n)\to0$ or $U(x^n)\to1$ as $n\to+\infty$, at least for a subsequence. But as already emphasized, this is sufficient to infer~\eqref{claimunif}.

Let us now show that $|\nabla U(x_1,x')|$ decays to $0$ exponentially as $|x'|\to+\infty$, uniformly in~$x_1\in\R$. To do so, let us consider only the limit $0$ in~\eqref{claimunif} (the limit $1$ can be handled similarly even if it means changing $U$ into $1-U$ and $f(s)$ into $-f(1-s)$). Since $f'(0)<0=f(0)$, there is~$\delta\in(0,1)$ such that
\be\label{fdelta}
f(s)\le\frac{f'(0)}{2}\,s\ \hbox{ for all $s\in[0,\delta]$}
\ee
and there is then $A>0$ such that
\be\label{defA2}
0\le U(x_1,x')\le\delta\ \hbox{ for all $|x'|\ge A$ and $x_1\in\R$}.
\ee
Take $\gamma>0$ small enough such that $\gamma^2+f'(0)/2<0$. The function $\overline{U}(x)=\delta\,e^{-\gamma(|x'|-A)}$ obeys
$$\baa{rcl}
\displaystyle\Delta\overline{U}(x)+\frac{f'(0)}{2}\,\overline{U}(x) & = & \displaystyle\delta\Big(\gamma^2-\frac{(N-2)\gamma}{|x'|}\Big)e^{-\gamma(|x'|-A)}+\frac{f'(0)\,\delta\,e^{-\gamma(|x'|-A)}}{2}\vspace{3pt}\\
& \le & \displaystyle\delta\Big(\gamma^2+\frac{f'(0)}{2}\Big)e^{-\gamma(|x'|-A)}<0\eaa$$
for all $|x'|\ge A$ and $x_1\in\R$. Since $U(x_1,x')\le\delta=\overline{U}(x_1,x')$ for all~$|x'|=A$ and $x_1\in\R$, together with~\eqref{fdelta}-\eqref{defA2} and~\eqref{claimunif} with limit $0$, it then easily follows from the maximum principle that $U(x_1,x')\le\overline{U}(x_1,x')=\delta\,e^{-\gamma(|x'|-A)}$ for all $|x'|\ge A$ and~$x_1\in\R$. From standard elliptic estimates, the function $|\nabla U|$ is bounded in $\R^N$ and moreover there is a positive real number~$B$ such that
\be\label{nablaU}
|\nabla U(x)|\le B\,e^{-\gamma|x'|}\hbox{ for all $x\in\R^N$}.
\ee

Now, as in the proof of Proposition~\ref{proliouville2}, from the semistability of $U$, one gets the existence of a positive $C^2(\R^N)$ function $\varphi$ and of a nonnegative real number $\lambda_\infty$ such that
$$-\Delta\varphi-f'(U)\varphi=\lambda_\infty\varphi\ge0\ \hbox{ in }\R^N.$$
Consider any unit vector $e$ of $\R^N$ and denote
$$w=\frac{e\cdot\nabla U}{\varphi}.$$
From standard elliptic estimates and the $C^{1,1}$ smoothness of $f$, the function $U$ is of class $C^3(\R^N)$ and it is elementary to check that the $C^2(\R^N)$ function $w$ obeys
$$w\,\nabla\cdot(\varphi^2\nabla w)=\lambda_\infty\varphi^2w^2\ge0\ \hbox{ in }\R^N.$$
Take a $C^\infty(\R)$ function $\zeta$ such that $0\le\zeta\le1$ in $\R$, $\zeta=1$ in $[-1,1]$ and $\zeta=0$ in $\R\setminus(-2,2)$. For $R\ge1$ and $x=(x_1,x')\in\R^N$, we define
$$\zeta_R(x)=\zeta\Big(\frac{x_1}{R^{N-1}}\Big)\times\zeta\Big(\frac{|x'|}{R}\Big).$$
Each function $\zeta_R$ is of class $C^\infty(\R^N)$ with compact support and there is a positive real number $C$ such that, for every $R\ge1$ and $x=(x_1,x')\in\R^N$, $|\partial_{x_1}\zeta_R(x)|\le C\,R^{1-N}$ and $|\nabla\zeta_R(x)|\le C\,R^{-1}$. For any $R\ge1$, let us define
$$\left\{\baa{l}
E_R=\big\{x=(x_1,x')\in\R^N:|x'|\le R,\,R^{N-1}\le|x_1|\le 2R^{N-1}\big\},\vspace{3pt}\\
F_R=\big\{x=(x_1,x')\in\R^N:R\le |x'|\le 2R,\,|x_1|\le 2R^{N-1}\big\},\vspace{3pt}\\
G_R=E_R\cup F_R.\eaa\right.$$
Observe that $|\nabla\zeta_R|=0$ in $\R^N\setminus G_R$ and that $|\nabla\zeta_R(x)|=|\partial_{x_1}\zeta_R(x)|\le C\,R^{1-N}$ for all $x\in E_R$. By integrating the inequation $w\,\nabla\cdot(\varphi^2\nabla w)\ge0$ against $\zeta_R^2$ (notice that all integrals below converge since all involved functions are continuous and $\zeta_R$ is compactly supported), one gets that
\be\label{integrals}\baa{rcl}
\displaystyle\int_{\R^N}\!\varphi^2\zeta_R^2|\nabla w|^2\le-2\int_{\R^N}\!w\,\varphi^2\zeta_R\,\nabla\zeta_R\cdot\nabla w & \!\!=\!\! & \displaystyle-2\int_{G_R}\!w\,\varphi^2\zeta_R\,\nabla\zeta_R\cdot\nabla w\vspace{3pt}\\
& \!\!\le\!\! & \displaystyle2\,\sqrt{\int_{G_R}\!\varphi^2\zeta_R^2|\nabla w|^2}\,\sqrt{\int_{G_R}\!w^2\varphi^2|\nabla\zeta_R|^2}.\eaa
\ee
Furthermore, from the above estimates on $\nabla\zeta_R$ and from~\eqref{nablaU}, one has
$$\baa{rcl}
\displaystyle\int_{G_R}w^2\varphi^2|\nabla\zeta_R|^2 & = & \displaystyle\int_{E_R}|e\cdot\nabla U|^2|\nabla\zeta_R|^2+\int_{F_R}|e\cdot\nabla U|^2|\nabla\zeta_R|^2\vspace{3pt}\\
& \le & \displaystyle 2B^2C^2\omega_{N-1}+B^2e^{-2\gamma R}C^2R^{-2}\omega_{N-1}((2R)^{N-1}-R^{N-1})\times(4R^{N-1}),\eaa$$
where $\omega_{N-1}$ denotes the $(N-1)$-dimensional Lebesgue measure of the unit Euclidean ball in~$\R^{N-1}$. Therefore, there is a positive real number $D$ such that
\be\label{D}
\int_{G_R}w^2\varphi^2|\nabla\zeta_R|^2\le D
\ee
for all $R\ge1$, hence
$$\int_{\R^N}\varphi^2\zeta_R^2|\nabla w|^2\le 4D$$
by~\eqref{integrals}. Therefore, owing to the definition of $\zeta_R$, the integral $\int_{\R^N}\varphi^2|\nabla w|^2$ converges and
$$\int_{G_R}\varphi^2\zeta_R^2|\nabla w|^2\to0\ \hbox{ as }R\to+\infty.$$
Together with~\eqref{integrals}-\eqref{D}, one infers that
$$\int_{\R^N}\varphi^2|\nabla w|^2=0,$$
hence $w$ is constant in $\R^N$. Owing to the definition of $w=(e\cdot\nabla U)/\varphi$, this implies that $e\cdot\nabla U$ is either of a strict constant sign, or is identically $0$ in $\R^N$. By taking now $e=(0,e')$ with a unit vector $e'$ of $\R^{N-1}$, and remembering that $U(x_1,x')\to0$ as $|x'|\to+\infty$, one infers that $e\cdot\nabla U\equiv0$ in $\R^N$ for any such $e=(0,e')$, and finally $U\equiv0$ in $\R^N$. As already underlined, the case of the limit $1$ in~\eqref{claimunif} can be handled similarly, and the proof of Proposition~\ref{proliouville} is thereby complete.
\end{proof}


\subsection{Proof of Theorem~\ref{thm2}}\label{Sec-3.3}

Throughout this section, we consider a domain $\Omega$ of the type~\eqref{defOmega}-\eqref{h}, for any $R>0$ and $\alpha\in[0,\pi/2)$, and we call~$u$ the time-increasing solution of~\eqref{1} and~\eqref{initial} given in Proposition~\ref{thm1}. Let $0<u_\infty\le1$ be its $C^2_{loc}(\overline{\Omega})$ limit as $t\to+\infty$. The function $u_\infty$ solves~\eqref{u-infty}, and
\be\label{inequuinfty}
0<u(t,x)<u_\infty(x)\le 1\ \hbox{ for all $(t,x)\in\R\times\overline{\Omega}$}.
\ee

Let us first notice that~\eqref{defw-} and~\eqref{uw} imply that $u(t,x)\to1$ as $x_1\to-\infty$, at least for every $t$ negative enough. Since $u$ is increasing in $t$ and $u<1$ in $\R\times\overline{\Omega}$, one infers that, for every $\tau\in\R$,
\be\label{unif1t}
u(t,x)\to1\hbox{ as $x_1\to-\infty$, uniformly with respect to $t\ge\tau$}.
\ee
Together with~\eqref{inequuinfty}, it follows that, if the solution $u$ is blocked in the sense of~\eqref{blocking}, then the convergence of $u(t,\cdot)$ to $u_\infty$ as $t\to+\infty$ is actually uniform in $\overline{\Omega}$.

After this preliminary observation, the first main step of the proof of Theorem~\ref{thm2} consists in showing that $u_\infty$ is a stable solution of~\eqref{u-infty} in $\overline{\Omega}$ in the sense of~\eqref{stable} with $\omega=\Omega$, whether~$u_\infty$ be identically $1$ or less than $1$ in $\overline{\Omega}$.

\begin{lemma}\label{lemstable}
The function $u_\infty$ is a stable solution of~\eqref{u-infty} in $\overline{\Omega}$ in the sense of~\eqref{stable}. 
\end{lemma}

\begin{proof}
Consider any $C^1(\overline{\Omega})$ function $\psi$ with compact support. The function $u$ satisfies
\be\label{uuinfty}
0\le u_t=\Delta(u-u_\infty)+f(u)-f(u_\infty)
\ee
in $\R\times\overline{\Omega}$. Since $\nu(x)\cdot\nabla(u(t,x)-u_\infty(x))=0$ for all $(t,x)\in\R\times\partial\Omega$ and since $\psi$ has compact support, multiplying~\eqref{uuinfty} by the nonnegative function $\psi^2/(u_\infty-u(t,\cdot))$ and integrating by parts over $\overline{\Omega}$, at a fixed time $t\in\R$, leads to
$$\baa{rcl}
0 & \le & \displaystyle\int_{\overline{\Omega}}\nabla(u_\infty-u(t,\cdot))\cdot\nabla\Big(\frac{\psi^2}{u_\infty-u(t,\cdot)}\Big)-\frac{f(u(t,\cdot))-f(u_\infty)}{u(t,\cdot)-u_\infty}\,\psi^2\vspace{3pt}\\
& = & \displaystyle\int_{\overline{\Omega}}2\,\frac{\psi\,\nabla(u_\infty-u(t,\cdot))\cdot\nabla\psi}{u_\infty-u(t,\cdot)}-\frac{|\nabla(u_\infty-u(t,\cdot))|^2\psi^2}{(u_\infty-u(t,\cdot))^2}-\frac{f(u(t,\cdot))-f(u_\infty)}{u(t,\cdot)-u_\infty}\,\psi^2\vspace{3pt}\\
& \le & \displaystyle\int_{\overline{\Omega}}|\nabla\psi|^2-\frac{f(u(t,\cdot))-f(u_\infty)}{u(t,\cdot)-u_\infty}\,\psi^2,\eaa$$
where all the above integrals converge since $\psi$ has compact support and all integrated functions or fields are at least continuous in $\overline{\Omega}$. But since $\psi$ has compact support and $u(t,\cdot)\to u_\infty$ as $t\to+\infty$ at least locally uniformly in $\overline{\Omega}$, the passage to the limit as $t\to+\infty$ in the above formula yields
$$0\le\int_{\overline{\Omega}}|\nabla\psi|^2-f'(u_\infty)\,\psi^2.$$
From the arbitrariness of $\psi\in C^1(\overline{\Omega})$ with compact support, the proof is complete.
\end{proof}

\begin{proof}[Proof of Theorem~$\ref{thm2}$] In order to show that $u$ either propagates completely in the sense of~\eqref{complete} or is blocked in the sense of~\eqref{blocking}, we have to show that either $u_\infty\equiv1$ in $\overline{\Omega}$, or $u_\infty(x)\to0$ as~$x_1\to+\infty$. Since the case $\alpha=0$ is trivial, as already noticed in the introduction ($u_\infty\equiv1$ in~$\overline{\Omega}$ in this case), one can assume that
$$\alpha>0$$
in the sequel. From Lemma~\ref{lemliouville0}, it is then sufficient to show that either $u_\infty(x)\to1$ as $x_1\to+\infty$ or $u_\infty(x)\to0$ as $x_1\to+\infty$. Since, for each $B\in\R$, the set $\{x\in\overline{\Omega}:x_1\ge B\}$ is connected and since $u_\infty$ is continuous in $\overline{\Omega}$, it is sufficient to show that, for any sequence $(x^n)_{n\in\N}$ with
$$x^n_1\to+\infty\ \hbox{ as $n\to+\infty$},$$
then, up to extraction of a subsequence, either $u_\infty(x^n)\to0$ or $u_\infty(x^n)\to1$ as $n\to+\infty$. Consider such a sequence $(x^n)_{n\in\N}$ in the sequel. Since the functions $u$ and $u_\infty$ are axisymmetric with respect to the $x_1$-axis, one can assume without loss of generality that
$$x^n=(x^n_1,x^n_2,0,\cdots,0),\ \hbox{ with }0\le x^n_2\le h(x^n_1),$$
for each $n\in\N$. Up to extraction of a subsequence, three cases can occur: either $\sup_{n\in\N}x^n_2<+\infty$, or $x^n_2\to+\infty$ and $h(x^n_1)-x^n_2\to+\infty$ as $n\to+\infty$, or $\sup_{n\in\N}(h(x^n_1)-x^n_2)<+\infty$. We consider these three cases separately.

Let us firstly consider the case $\sup_{n\in\N}x^n_2<+\infty$. Call
$$y^n=(x^n_1,0,\cdots,0).$$
Here, up to extraction of a subsequence, the functions $u_\infty(\cdot+y^n)$ converge in $C^2_{loc}(\R^N)$ to a~$C^2(\R^N)$ solution $U$ of $\Delta U+f(U)=0$ in $\R^N$ which is axisymmetric with respect to the $x_1$-axis (since so is $u_\infty$). Furthermore, $0\le U\le1$ in $\R^N$. Let us now show that $U$ is stable in the sense of~\eqref{stable} (Proposition~\ref{proliouville} will then yield the desired conclusion). Pick any $C^1(\R^N)$ function $\psi$ with compact support $K$. For $n\in\N$, denote $\psi_n(x)=\psi(x-y^n)$ for $x\in\overline{\Omega}$. Each function $\psi_n$ is of class $C^1(\overline{\Omega})$ with compact support, hence
$$\int_{\overline{\Omega}}|\nabla\psi_n|^2-f'(u_\infty)\psi_n^2\ge0$$
by the semistability of $u_\infty$ established in Lemma~\ref{lemstable}. But, for every $n$ large enough, the support~$y^n+K$ of $\psi_n$ is included in $\overline{\Omega}$, and the previous inequality then means that
$$\int_K|\nabla\psi|^2-f'(u_\infty(\cdot+y^n))\psi^2\ge0.$$
Since $u_\infty(\cdot+y^n)\to U$ as $n\to+\infty$ at least locally uniformly in $\R^N$ and $f$ is of class $C^1(\R)$, one concludes by passing to the limit $n\to+\infty$
$$\int_{\R^N}|\nabla\psi|^2-f'(U)\psi^2=\int_K|\nabla\psi|^2-f'(U)\psi^2\ge0.$$
Therefore, $U$ is a stable solution of $\Delta U+f(U)=0$ in $\R^N$ and it satisfies the other assumptions of Proposition~\ref{proliouville}. One then deduces that either $U\equiv0$ in $\R^N$ or $U\equiv1$ in~$\R^N$. In particular, since the sequence $(x^n_2)_{n\in\N}$ was assumed to be bounded, one concludes that either $u_\infty(x^n)\to0$ or $u_\infty(x^n)\to1$, up to extraction of a subsequence.

In the second case, we assume that $x^n_2\to+\infty$ and $h(x^n_1)-x^n_2\to+\infty$ as $n\to+\infty$. Define $U_n(x)=u_\infty(x+x^n)$ for $x\in\overline{\Omega}-x^n$. From standard elliptic estimates, together with the axisymmetry of $u_\infty$ with respect to the $x_1$-axis, the functions~$U_n$ converge in $C^2_{loc}(\R^N)$, up to extraction of a subsequence, to a $C^2(\R^N)$ function $\tilde{U}$, which actually depends on $(x_1,x_2)$ only, that is,
$$\tilde{U}(x_1,\cdots,x_N)=U(x_1,x_2)$$
for some $C^2(\R^2)$ function $U$, and there holds $\Delta U+f(U)=0$ in $\R^2$. Furthermore, $0\le U\le1$ in~$\R^2$. Let us now show that $U$ satisfies the condition~\eqref{stable} with $\omega=\R^2$, and Proposition~\ref{proliouville2} will then yield the desired conclusion. So, consider any $C^1(\R^2)$ function $\psi$ with compact support~$K$. For $n\in\N$, define the following function $\psi_n$ in $\overline{\Omega}$ by:
$$\psi_n(x)=\psi_n(x_1,x')=\left\{\baa{ll}
\psi(x_1-x^n_1,|x'|-x^n_2) & \hbox{if }(x_1,|x'|)\in K+(x^n_1,x^n_2),\vspace{3pt}\\
0 & \hbox{otherwise}.\eaa\right.$$
Since $\lim_{n\to+\infty}x^n_2=+\infty$, it follows that, for every $n$ large enough, $\psi_n$ is a $C^1(\overline{\Omega})$ function with compact support. Lemma~\ref{lemstable} implies that, for all $n$ large enough,
$$\int_{\overline{\Omega}}|\nabla\psi_n|^2-f'(u_\infty)\psi_n^2\ge0.$$
But since both $u_\infty$ and $\psi_n$ are axisymmetric with respect to the $x_1$-axis, the above inequality means that, for all $n$ large enough,
\be\label{ineqpsin1}\baa{l}
\displaystyle\int_{\{x_1\in\R,\,0\le x_2\le h(x_1)\}}\Big[|\nabla\psi(x_1-x^n_1,x_2-x^n_2)|^2\vspace{3pt}\\
\qquad\qquad\qquad\qquad\qquad-f'(u_\infty(x_1,x_2,0,\cdots,0))\psi(x_1-x^n_1,x_2-x^n_2)^2\Big]\,dx_1dx_2\ge0.\eaa
\ee
Since both sequences $(x^n_2)_{n\in\N}$ and $(h(x^n_1)-x^n_2)_{n\in\N}$ converge to $+\infty$ and since $\psi$ has compact support, denoted by $K$, the previous inequality means that, for all $n$ large enough,
\be\label{ineqpsin2}
\int_K\Big[|\nabla\psi(x_1,x_2)|^2-f'(u_\infty(x_1+x^n_1,x_2+x^n_2,0,\cdots,0))\psi(x_1,x_2)^2\Big]\,dx_1dx_2\ge0.
\ee
Since $u_\infty(x_1+x^n_1,x_2+x^n_2,0,\cdots,0)\to\tilde{U}(x_1,x_2,0,\cdots,0)=U(x_1,x_2)$ locally uniformly in $(x_1,x_2)\in\R^2$ as $n\to+\infty$, and since $f$ is of class $C^1(\R)$, one gets that
$$\int_{\R^2}|\nabla\psi|^2-f'(U)\psi^2\ge0.$$
As a consequence, the $C^2(\R^2)$ function $U$ is a stable solution of $\Delta U+f(U)=0$ in $\R^2$ such that $0\le U\le1$ in $\R^2$. Proposition~\ref{proliouville2} then implies that either $U\equiv0$ in $\R^2$ or $U\equiv1$ in $\R^2$, that is, either $\tilde{U}\equiv0$ or $\tilde{U}\equiv1$ in $\R^N$. Hence, either $u_\infty(x^n)\to0$ or $u_\infty(x^n)\to1$, up to extraction of a subsequence.

Consider thirdly the case $\sup_{n\in\N}h(x^n_1)-x^n_2<+\infty$. Define
$$y^n=(x^n_1,h(x^n_1),0,\cdots,0),\ \ U_n(x)=u_\infty(x+y^n)\hbox{ for }x\in\overline{\Omega}-y^n,$$
and
$$H=\{(x_1,x_2)\in\R^2:x_2<x_1\tan\alpha\},$$
which is an open half-plane of $\R^2$. From standard elliptic estimates, together with the definitions~\eqref{defOmega}-\eqref{h} and the axisymmetry of $u_\infty$ with respect to the $x_1$-axis, there is a $C^2(\overline{H}\times\R^{N-2})$ function $\tilde{U}$, which actually depends on $(x_1,x_2)$ only, that is,
$$\tilde{U}(x_1,\cdots,x_N)=U(x_1,x_2)$$
for some $U\in C^2(\overline{H})$, such that, up to extraction of a subsequence, $\|U_n-\tilde{U}\|_{C^2(\mathcal{K}\cap(\overline{\Omega}-y^n))}\to0$ as $n\to+\infty$ for every compact set $\mathcal{K}\subset\R^N$ (notice that, for each such $\mathcal{K}$, there holds $\mathcal{K}\cap(\overline{\Omega}-y^n)\subset\overline{H}\times\R^{N-2}$ for all $n$ large enough). The function $U$ then satisfies
$$\Delta U+f(U)=0\hbox{ in $\overline{H}$},$$
together with $\nu\cdot\nabla U=0$ on $\partial H$ and $0\le U\le1$ in $\overline{H}$. Let us now show that $U$ satisfies the condition~\eqref{stable} with $\omega=H$, and Proposition~\ref{proliouville2bis} will then yield the desired conclusion. So, consider any $C^1(\overline{H})$ function $\psi$ with compact support $K$. For $n\in\N$, define the following function $\psi_n$ in $\overline{\Omega}$ by:
$$\psi_n(x)=\psi_n(x_1,x')=\left\{\baa{ll}
\psi(x_1-x^n_1,|x'|-h(x^n_1)) & \hbox{if }(x_1,|x'|)\in K+(x^n_1,h(x^n_1)),\vspace{3pt}\\
0 & \hbox{otherwise}.\eaa\right.$$
Since $\lim_{n\to+\infty}h(x^n_1)=+\infty$, it follows that, for all $n$ large enough, $\psi_n$ is a $C^1(\overline{\Omega})$ function with compact support. Lemma~\ref{lemstable} implies that, for all $n$ large enough,
$$\int_{\overline{\Omega}}|\nabla\psi_n|^2-f'(u_\infty)\psi_n^2\ge0.$$
But since both $u_\infty$ and $\psi_n$ are axisymmetric with respect to the $x_1$-axis, and since $h(x_1)=x_1\tan\alpha$ for all $x_1\ge L\cos\alpha$, together with $x^n_1\to+\infty$ as $n\to+\infty$, the definition of $\psi_n$ and the previous inequality then yield~\eqref{ineqpsin1}-\eqref{ineqpsin2}, with $x^n_2$ replaced by $h(x^n_1)$, for all $n$ large enough. Since $u_\infty(x_1+x^n_1,x_2+h(x^n_1),0,\cdots,0)\to\tilde{U}(x_1,x_2,0,\cdots,0)=U(x_1,x_2)$ uniformly in~$K$ (because~$K\times\{0\}^{N-2}\subset\overline{\Omega}-y^n$ for all $n$ large enough), and since $f$ is of class $C^1(\R)$, one gets that
$$\int_{\overline{H}}|\nabla\psi|^2-f'(U)\psi^2\ge0.$$
As a consequence, the $C^2(\R^2)$ function $U$ is a stable solution of $\Delta U+f(U)=0$ in $\R^2$ such that $0\le U\le1$ in $\R^2$. Proposition~\ref{proliouville2} then implies that either $U\equiv0$ in $\R^2$ or $U\equiv1$ in $\R^2$, that is, either $\tilde{U}\equiv0$ or $\tilde{U}\equiv1$ in $\R^N$. Hence, either $u_\infty(x^n)\to0$ or $u_\infty(x^n)\to1$, up to extraction of a subsequence.

As a conclusion, for any sequence $(x^n)_{n\in\N}$ in $\overline{\Omega}$ such that $x^n_1\to+\infty$ as $n\to+\infty$, one has, up to extraction of a subsequence, either $u_\infty(x^n)\to0$ or $u_\infty(x^n)\to1$ as $n\to+\infty$. As already emphasized, this leads to the desired conclusion, and the proof of Theorem~\ref{thm2} is thereby complete.
\end{proof}


\section{Transition fronts and long-time behavior of the level sets: proofs of Theorems~\ref{thm3} and~\ref{thm4}}\label{Sec-4}

This section is devoted to the proofs of Theorems~\ref{thm3} and~\ref{thm4}. For any given $R>0$ and $\alpha\in[0,\pi/2)$, we especially show that the solution $0<u<1$ of~\eqref{1} emanating from the planar front $\phi(x_1-ct)$ in the sense of~\eqref{initial} is a transition front connecting $1$ and $0$, and we show further more precise estimates on the position of the level sets at large time in case of complete propagation. But we start in the next subsection with the, immediate, proof of Theorem~\ref{thm4}.


\subsection{Proof of Theorem~\ref{thm4}}

We here assume that $u$ is blocked, in the sense of~\eqref{blocking}. Since $u(t,x)-\phi(x_1-ct)\to0$ as~$t\to-\infty$ uniformly in $x\in\overline{\Omega}$, and since $\phi(-\infty)=1$ and $\phi(+\infty)=0$, one infers that
$$\sup_{t\le-A,\,x\in\overline{\Omega},\,x_1-ct\le-A}|u(t,x)-1|\to0\ \hbox{ and }\ \sup_{t\le-A,\,x\in\overline{\Omega},\,x_1-ct\ge A}u(t,x)\to0\ \hbox{ as $A\to+\infty$}.$$
Furthermore, from~\eqref{unif1t}, one knows that, for every $\tau\in\R$, $u(t,x)\to1$ as $x_1\to-\infty$, uniformly with respect to $t\ge\tau$. Since $0<u(t,x)<u_\infty(x)$ for all $(t,x)\in\R\times\overline{\Omega}$ and $u_\infty(x)\to0$ as~$x_1\to+\infty$, there also holds
$$u(t,x)\to0\ \hbox{ as $x_1\to+\infty$, uniformly in $t\ge\tau$},$$
for every $\tau\in\R$. All these properties, owing to the definition~\eqref{defOmega}-\eqref{h} of $\Omega$, imply that $u$ is a transition front connecting $1$ and $0$, with the sets $\Omega^\pm_t$ and $\Gamma_t$ given for instance by~\eqref{Gamma_tbis}. In particular, $u$ does not have any global mean speed in the sense of Definition~\ref{Def1} (but one can still say that it has a ``past'' speed equal to $c$, and a ``future'' speed equal to $0$, following the terminology used in~\cite{HR2016}).


\subsection{Proof of Theorem~\ref{thm3}}

We here assume that $\alpha>0$ and the solution $u$ of~\eqref{1} with past condition~\eqref{initial} propagates completely, namely $u(t,\cdot)\to 1$ as $t\to+\infty$ locally uniformly in $\overline\Omega$. We will prove, thanks to a comparison argument, that the level sets of $u$ can be sandwiched between two expanding spherical surfaces at large time in $\overline{\Omega^+}$, and that $u$ is a transition front with sets $\Gamma_t$ and $\Omega^\pm_t$ defined by~\eqref{Gamma_t}-\eqref{Omega_t}. Moreover, we show that along each level set the function $u$ converges locally to the planar traveling front at large time, in which a Liouville type theorem of Berestycki and the first author in~\cite{BH2007} for entire solutions of the bistable equation plays an essential role.

\subsubsection*{Large time estimates of $u$ for $x\in\overline{\Omega^+}$ with $|x|$ large}

We aim at proving the key Lemma~\ref{lemma-super sub} below, which gives refined bounds, for large $t$ and for~$x\in\overline{\Omega^+}$ with large norm $|x|$, of the solution $u$ of~\eqref{1} satisfying the complete propagation condition \eqref{complete}. This lemma is based on the construction, inspired by Fife and McLeod~\cite{FM1977} and Uchiyama~\cite{U1985}, of suitable sub- and supersolutions. For this purpose, let us first define a function $\vartheta$ in~$[0,+\infty)$ by
\begin{equation*}
\label{vartheta}
\vartheta(t)=\frac{2\,(\ln(t+1))^{3/2}}{3}. 
\end{equation*} 
Notice that 
\begin{equation}
\label{vartheta-1}
\vartheta(t)\ge 0\ \hbox{ and }\ 0 \le \vartheta'(t)=\frac{\sqrt{\ln(t+1)}}{t+1}<1~~\text{for all}~ t\ge 0,
\end{equation}
and 
$$\int_{0}^{+\infty} e^{-r\vartheta(t)}dt<+\infty~~\text{for all}~ r> 0.$$
We also recall that $L>0$ is given in \eqref{h}.

\begin{lemma}
\label{lemma-super sub}
There exist $\tau>0$, $\tau_1\in\R$, $\tau_2\in\mathbb{R}$, $z_1\in\mathbb{R}$, $z_2\in\mathbb{R}$, $\delta>0$ and $\mu>0$  such that 
\begin{equation}
\label{sup-1}\baa{l}
\displaystyle u(t,x)\le \phi\Big(|x|-c(t-\tau_1+\tau)+\frac{N-1}{c}\ln (t-\tau_1+\tau)+z_1\Big)+\delta e^{-\delta \vartheta(t-\tau_1)}+\delta e^{-\mu(|x|-L)}\vspace{3pt}\\
\qquad\qquad\qquad\qquad\hbox{for all $t\ge \tau_1$ and $x\in\overline{\Omega^+}$ with $|x|\ge L$},\eaa
\ee
and 
\begin{equation}
\label{sub-1}\baa{l}
\displaystyle u(t,x)\ge \phi\Big(|x|-c(t-\tau_2+\tau)+\frac{N-1}{c}\ln (t-\tau_2+\tau) + z_2\Big)-\delta e^{-\delta\vartheta (t-\tau_2)}-\delta e^{-\mu(|x|-L)}\vspace{3pt}\\
\qquad\qquad\qquad\qquad\hbox{for all $t\ge \tau_2$ and $x\in\overline{\Omega^+}$ with $|x|\ge L$}.\eaa
\ee
\end{lemma}

\begin{proof}
\textit{Step 1: choice of some parameters.} Choose first $\mu>0$ and then $\delta\in(0,1/2)$ such that
\begin{equation}
\label{mu-1}
0<\mu<\sqrt{\min\Big(\frac{|f'(0)|}{2},\frac{|f'(1)|}{2}\Big)},
\end{equation}
and 
\begin{equation}
\label{delta-1}
0<\delta<\min\Big(\frac{\mu c}{2}, \frac{\mu_*}{2},\frac{\mu^*}{2},\frac{\mu^2}{2}\Big),\ \ f'\le\frac{f'(0)}{2}\ \text{in}\ [0,3\delta],\ \ f'\le \frac{f'(1)}{2}\ \text{in}\ [1-3\delta,1],	
\end{equation}
with $\mu_*>0$ and $\mu^*>0$ as in~\eqref{phi}. From~\eqref{TW}-\eqref{phi'}, there are $C>0$ and $K>0$ such that 
\begin{equation}
\label{phi-1}
\phi\ge 1-\delta\ \text{in}\ (-\infty,-C],\ \phi\le \delta \ \text{in}\  [C,+\infty),\ |\phi'(z)|\le K\min\big(e^{\mu_*z/2}, e^{-\mu^*z/2}\big)\ \text{for all}\ z\in \mathbb{R}.
\end{equation}
Since $\phi'$ is continuous and negative in $\mathbb{R}$, there exists a constant  $\kappa>0$ such that 
\begin{equation}
\label{kappa-1}
\phi'\le -\kappa \ \text{in}\ [-C,C].
\end{equation}
We then choose $\sigma>0$ such that
\begin{equation}
\label{kappa-sigma}
\max_{[0,1]}|f'|+\mu^2\le \kappa \sigma.
\end{equation}
Let then $\tau_0>0$ be such that
\begin{equation}
\label{tau}
\frac{N-1}{c}\ln t\le \frac{c}{2}\,t~~\text{for all}~t\ge \tau_0,
\end{equation}
and $\eta>0$ such that
\be\label{defeta}
e^{-\mu_*\eta/2}\le\frac{L}{c\tau_0}.
\ee
From~\eqref{delta-1}, there exist some constants $M_1, M_2\ge C$ such that
\begin{equation}
\label{A-inf-1}
\max\Big(\frac{(N-1)Ke^{-\mu_*(M_1+\vartheta(t))/2}}{c\tau_0},\frac{(N-1)K e^{-\mu^*(M_2+\vartheta(t))/2}}{L}\Big)\le\delta^2 e^{-\delta\vartheta(t)}~~\text{for all}~ t\ge 0
\end{equation}
Now define
$$\omega=\int_0^{+\infty}\sigma \delta( e^{-\delta\vartheta(s)}+ e^{-\delta s})\,ds\ \in(0,+\infty),$$
and
$$M=\max\big(M_1+\omega+1+\eta,M_2+\omega+1\big)>0\ \hbox{ and }\ B=2C+\omega+1>0.$$
For every $\tau\ge\tau_0$,~\eqref{tau} implies that $c(t+\tau)-((N-1)/c)\ln(t+\tau)+L+B-C+M+\vartheta(t)\ge L$ for all $t\ge0$, hence the function $\Lambda$ defined in $[0,+\infty)$ by
\begin{equation}\label{Lambda}
\Lambda(t)=\sup_{\substack{\big||x|-c (t+\tau)+((N-1)/c)\ln (t+\tau)-L-B+C\big|\le M+\vartheta(t)\\ x\in\overline{\Omega^+},~ |x|\ge L}} \bigg|\frac{N-1}{|x|}-\frac{N-1}{c(t+\tau)}\bigg|\,\ge0
\end{equation}
is well defined, nonnegative and continuous in $[0,+\infty)$. Furthermore, it is easy to see that it is integrable over $[0,+\infty)$, and that
$$\lim_{\tau\to+\infty}\int_0^{+\infty}\!\!\Lambda(t)\,dt\,=\,0.$$
Let us then fix $\tau\ge\max(\tau_0,L/c)$ large enough so that $\int_0^{+\infty}\Lambda(t)\,dt<1$ and let us introduce a nonnegative function $\varrho$ defined in $[0,+\infty)$ by
\begin{equation}
\label{varrho-1}
\varrho(t)=\int_0^t\big(\Lambda(s)+\sigma \delta( e^{-\delta\vartheta(s)}+ e^{-\delta s})\big)\,ds\ \ge0.
\end{equation} 
One then has $0<\varrho(+\infty)<1+\omega$. Hence,
$$M\ge\max\big(M_1+\varrho(+\infty)+\eta,M_2+\varrho(+\infty)\big)\ge C+\varrho(+\infty)\ \hbox{ and }\ B\ge2C+\varrho(+\infty)$$
and, from~\eqref{defeta}-\eqref{A-inf-1} and the inequality $c\tau\ge L$, there holds
\begin{equation}
\label{A-inf-2}
\!\!\max\!\Big(\!\frac{(N\!-\!1)K e^{-\mu_*(M\!+\!\vartheta(t)\!-\!\varrho(+\infty))/2}}{L},\frac{(N\!-\!1)Ke^{-\mu^*(M\!+\!\vartheta(t)\!-\!\varrho(+\infty))/2}}{c\tau}\!\Big)\!\le\!\delta^2 e^{-\delta\vartheta(t)}~ \text{for all}~ t\!\ge\!0.
\end{equation}
For notational convenience, let us finally define, for $s\ge0$ and $x\in\overline{\Omega^+}$ with $|x|\ge L$,
$$A(s,x)=\frac{N-1}{|x|}-\frac{N-1}{c(s+\tau)}$$
and
$$\zeta(s,x)=|x|-c(s+\tau)+\frac{N-1}{c}\ln(s+\tau)-L-B+C.$$

\vskip 0.3cm
\noindent\textit{Step 2: proof of \eqref{sup-1}.} Since $u(t,x)-\phi(x_1-ct)\to 0$ as $t\to-\infty$ uniformly in $\overline\Omega$, and since~$\phi(+\infty)=0$, there exists~$\tau_1<0$ such that $\phi(-c\tau_1)\le \delta/2$ and
\begin{equation}
\label{super-t1-1}
u(\tau_1,x)\le \phi(x_1-c\tau_1)+\frac{\delta}{2}\le\phi(-c\tau_1)+\frac{\delta}{2}\le \delta
\end{equation} 
for all $x\in\overline{\Omega^+}$. For $t\ge \tau_1$ and $x\in\overline{\Omega^+}$ with $|x|\ge L$, let us set
\begin{equation*}
\overline u(t,x)=\min\big(\phi(\overline\xi(t,x))+\delta e^{-\delta\vartheta(t-\tau_1)}+\delta e^{-\mu(|x|-L)}, 1\big),
\end{equation*}
where 
\begin{equation*}
\overline\xi(t,x)=\zeta(t-\tau_1,x)-\varrho(t-\tau_1)=|x|-c(t-\tau_1+\tau)+\frac{N-1}{c}\ln (t-\tau_1+\tau)-L-B+C-\varrho(t-\tau_1).
\end{equation*}
Let us now check that $\overline u(t,x)$ is a supersolution of the problem satisfied by $u(t,x)$ for $t\ge \tau_1$ and $x\in\overline{\Omega^+}$ with $|x|\ge L$.

We first verify the initial and boundary conditions. On the one hand, at time $t=\tau_1$,  from~\eqref{super-t1-1} it follows that $\overline u(\tau_1,x)\ge \delta\ge u(\tau_1,x)$ for all $x\in\overline{\Omega^+}$ with $|x|\ge L$. On the other hand, for $t\ge \tau_1$ and for all $x\in\overline{\Omega^+}$ with $|x|=L$, one infers from \eqref{tau}, \eqref{varrho-1} and the choice of~$B=2C+\omega+1\ge 2C$, that $\overline \xi(t,x)\le -B+C\le -C$, hence~\eqref{phi-1} gives $\phi(\overline\xi(t,x))\ge 1-\delta$, which yields $\overline u(t,x)\ge\min(1-\delta+\delta e^{-\delta\vartheta(t-\tau_1)}+\delta,1)=1>u(t,x)$. Lastly, owing to~\eqref{defOmega}-\eqref{h}, one has $\nu(x)\cdot \nabla\overline u(t,x)=0$ for every $t\ge \tau_1$ and $x\in\partial\Omega^+$ such that $|x|>L$ and $\overline u(t,x)<1$, since $\nu(x)\cdot x/|x|=0$ at any such~$x$.
	
Next, let us check that
$$\mathcal{L}\overline u(t,x)=\overline u_t(t,x)-\Delta\overline u(t,x)-f(\overline u(t,x))\ge 0$$
for all $t\ge \tau_1$ and $x\in\overline{\Omega^+}$ such that $|x|\ge L$ and $\overline u(t,x)<1$. After a straightforward computation, we get, for such a $(t,x)$,
$$\baa{rcl}
\mathcal{L}\overline u(t,x) & = & f(\phi(\overline \xi(t,x)))-f(\overline u(t,x))-\delta^2 \vartheta'(t-\tau_1) e^{-\delta\vartheta(t-\tau_1)}-\mu^2\delta e^{-\mu(|x|-L)}\vspace{3pt}\\
& & \displaystyle+\frac{N-1}{|x|}\mu\delta e^{-\mu(|x|-L)}-\bigg(\varrho'(t-\tau_1)+\underbrace{ \frac{N-1}{|x|}-\frac{N-1}{c(t-\tau_1+\tau)}}_{=A(t-\tau_1,x)\ge-\frac{N-1}{c\tau}\ge-\frac{N-1}{c\tau_0}}\bigg)\phi'(\overline\xi(t,x)).\eaa$$
Three cases can occur, namely: either $\zeta(t-\tau_1,x)<-M-\vartheta(t-\tau_1)$, or $\zeta(t-\tau_1,x)>M+\vartheta(t-\tau_1)$, or $|\zeta(t-\tau_1,x)|\le M+\vartheta(t-\tau_1)$.

Consider firstly the case
$$\zeta(t-\tau_1,x)<-M-\vartheta(t-\tau_1).$$
One then has $\overline\xi(t,x)\le\zeta(t-\tau_1,x)<-M-\vartheta(t-\tau_1)<-M_1-\vartheta(t-\tau_1)\le -C$, hence $1>\phi(\overline\xi(t,x))\ge 1-\delta$ and $\overline u(t,x)\ge 1-\delta$ (remember also that $(t,x)$ is assumed to be such that~$1>\overline u(t,x)$). By~\eqref{delta-1} one gets that
$$f(\phi(\overline\xi(t,x)))-f(\overline u(t,x))\ge-\frac{f'(1)}{2}\,(\delta e^{-\delta\vartheta(t-\tau_1)}+\delta e^{-\mu(|x|-L)}).$$
Notice also from~\eqref{phi-1} that $0<-\phi'(\overline\xi(t,x))\le K e^{\mu_*\overline\xi(t,x)/2}\le K e^{\mu_*(-M_1-\vartheta(t-\tau_1))/2}$, which yields
$$-A(t-\tau_1,x)\,\phi'(\overline\xi(t,x))\ge-\frac{N-1}{c\tau_0}\,K e^{\mu_*(-M_1-\vartheta(t-\tau_1))/2}\ge-\delta^2 e^{-\delta\vartheta(t-\tau_1)}$$
thanks to~\eqref{A-inf-1}. Hence, it follows from \eqref{vartheta-1}, \eqref{mu-1}-\eqref{delta-1}, \eqref{varrho-1}, as well as the negativity of $\phi'$ and $f'(1)$, that
\begin{align*}
\mathcal{L}\overline u(t,x)\ge& \ -\frac{f'(1)}{2}(\delta e^{-\delta\vartheta(t-\tau_1)}+\delta e^{-\mu(|x|-L)})-\delta^2\vartheta'(t-\tau_1) e^{-\delta\vartheta(t-\tau_1)}-\mu^2\delta e^{-\mu(|x|-L)}\\
&\ +\frac{N-1}{|x|}\mu\delta e^{-\mu(|x|-L)}-\varrho'(t-\tau_1)\phi'(\overline\xi(t,x))-\delta^2 e^{-\delta\vartheta(t-\tau_1)}\\
\ge & \ \Big(\!-\frac{f'(1)}{2}-\delta\vartheta'(t-\tau_1)-\delta\Big)\delta e^{-\delta\vartheta(t-\tau_1)}+\Big(\!-\frac{f'(1)}{2}-\mu^2\Big)\delta e^{-\mu(|x|-L)}\ge 0.
\end{align*}

Consider secondly the case
$$\zeta(t-\tau_1,x)>M+\vartheta(t-\tau_1).$$
One then has $\overline\xi(t,x)>M\!+\!\vartheta(t\!-\!\tau_1)\!-\!\varrho(+\infty)\ge C$, hence $0<\phi(\overline\xi(t,x))\le\delta$ and $0<\overline u(t,x)\le 3\delta$. From~\eqref{delta-1} one gets that $f(\phi(\overline\xi(t,x)))-f(\overline u(t,x))\ge -(f'(0)/2)(\delta e^{-\delta\vartheta(t-\tau_1)}\!+\!\delta e^{-\mu(|x|-L)})$. By noticing that $0<-\phi'(\overline\xi(t,x))\le K e^{-\mu^*\overline\xi(t,x)/2}\le K e^{-\mu^*(M+\vartheta(t-\tau_1)-\varrho(+\infty))/2}$ from~\eqref{phi-1}, one gets that
$$-A(t-\tau_1)\,\phi'(\overline\xi(t,x))\ge-\frac{N-1}{c\tau}\,K e^{-\mu^*(M+\vartheta(t-\tau_1)-\varrho(+\infty))/2}\ge -\delta^2 e^{-\delta\vartheta(t-\tau_1)},$$
from~\eqref{A-inf-2}. It then follows from \eqref{vartheta-1}, \eqref{mu-1}-\eqref{delta-1}, \eqref{varrho-1}, as well as the negativity of $\phi'$ and~$f'(0)$, that
\begin{align*}
\mathcal{L}\overline u(t,x)\ge& \ -\frac{f'(0)}{2}(\delta e^{-\delta\vartheta(t-\tau_1)}+\delta e^{-\mu(|x|-L)})-\delta^2 \vartheta'(t-\tau_1) e^{-\delta\vartheta(t-\tau_1)}-\mu^2\delta e^{-\mu(|x|-L)}\\
&\ +\frac{N-1}{|x|}\mu\delta e^{-\mu(|x|-L)}-\varrho'(t-\tau_1)\phi'(\overline\xi(t,x))-\delta^2 e^{-\delta\vartheta(t-\tau_1)}\\
\ge& \ \Big(\!-\frac{f'(0)}{2}-\delta\vartheta'(t-\tau_1)-\delta\Big)\delta e^{-\delta\vartheta(t-\tau_1)}+\Big(\!-\frac{f'(0)}{2}-\mu^2\Big)\delta e^{-\mu(|x|-L)}\ge 0.    		
\end{align*}

Lastly, we consider the case
$$|\zeta(t-\tau_1,x)|\le M+\vartheta(t-\tau_1).$$
One observes from the definitions of $\Lambda$ and $\varrho$ in~\eqref{Lambda}-\eqref{varrho-1} that, in this range, there holds
\begin{equation}
\label{A-1}
\varrho'(t-\tau_1)+A(t-\tau_1,x)\ge\sigma\delta\big( e^{-\delta\vartheta(t-\tau_1)}+e^{-\delta(t-\tau_1)}\big)>0.
\end{equation} 
Three subcases may then occur. If $-C\le\overline \xi(t,x)\le C$, then $-\phi'(\overline\xi(t,x))\ge \kappa>0$ by~\eqref{kappa-1} and $f(\phi(\overline\xi(t,x)))-f(\overline u(t,x))\ge -\big(\max_{[0,1]}|f'|\big)(\delta e^{-\delta\vartheta(t-\tau_1)}+\delta e^{-\mu(|x|-L)})$. Moreover, from the expression of $\overline\xi(t,x)$ and \eqref{tau}, one obtains
\begin{equation*}
|x|-L\ge c(t-\tau_1+\tau)-\frac{N-1}{c}\ln (t-\tau_1+\tau)+B-2C\ge \frac{c}{2}(t-\tau_1+\tau)+B-2C>\frac{c}{2}(t-\tau_1)+B-2C.
\end{equation*}
This reveals that $e^{-\mu(|x|-L)}\le e^{-\mu(c(t-\tau_1)/2+B-2C)}$. Therefore, since $B>2C$, and by virtue of~\eqref{vartheta-1},~\eqref{delta-1},~\eqref{kappa-1}-\eqref{kappa-sigma} and~\eqref{A-1}, one gets that
\begin{align*}
\mathcal{L}\overline u(t,x)\ge& \ -\Big(\max_{[0,1]}|f'|\Big)\,(\delta e^{-\delta\vartheta(t-\tau_1)}+\delta e^{-\mu(|x|-L)})-\delta^2 \vartheta'(t-\tau_1) e^{-\delta\vartheta(t-\tau_1)}-\mu^2\delta e^{-\mu(|x|-L)}\\
&\ +\frac{N-1}{|x|}\mu\delta e^{-\mu(|x|-L)}+\kappa\sigma \delta( e^{-\delta\vartheta(t-\tau_1)}+ e^{-\delta(t-\tau_1)})\\
\ge& \ \Big(\!-\max_{[0,1]}|f'|-\delta\vartheta'(t-\tau_1)+\kappa \sigma\Big)\delta e^{-\delta\vartheta(t-\tau_1)}+\kappa \sigma\delta e^{-\delta(t-\tau_1)}\\
&\ -\Big(\max_{[0,1]}|f'|+\mu^2\Big)\delta e^{-\mu(c(t-\tau_1)/2+B-2C)}\\
\ge& \ \Big[\kappa \sigma-\Big(\max_{[0,1]}|f'|+\mu^2\Big) e^{-\mu(B-2C)}\Big]\delta e^{-\mu c(t-\tau_1)/2}\ge0.
\end{align*}
If $\overline\xi(t,x)\ge C$, one has $0<\phi(\overline \xi(t,x))\le \delta$ and then $0<\overline u(t,x)\le 3\delta$. Due to \eqref{vartheta-1}, \eqref{mu-1}-\eqref{delta-1}, as well as~\eqref{A-1}, an analogous argument as above leads to 
\begin{align*}
\mathcal{L}\overline u(t,x)\ge\Big(\!-\frac{f'(0)}{2}-\delta\vartheta'(t-\tau_1)\Big)\delta e^{-\delta\vartheta(t-\tau_1)}+\Big(\!-\frac{f'(0)}{2}-\mu^2\Big)\delta e^{-\mu(|x|-L)}\ge 0.    		
\end{align*}
If $\overline \xi(t,x)\le -C$, it follows that 
$1>\phi(\overline\xi(t,x))\ge 1-\delta$ and then $\overline u(t,x)\ge 1-\delta$ (remember also that $(t,x)$ is assumed to be such that $1>\overline u(t,x)$). Finally, one infers from~\eqref{vartheta-1}, \eqref{mu-1}-\eqref{delta-1} as well as~\eqref{A-1} that
\begin{align*}
\mathcal{L}\overline u(t,x)
\ge\Big(\!-\frac{f'(1)}{2}-\delta\vartheta'(t-\tau_1)\Big)\delta e^{-\delta\vartheta(t-\tau_1)}+\Big(\!-\frac{f'(1)}{2}-\mu^2\Big)\delta e^{-\mu(|x|-L)}\ge 0.    		
\end{align*}

As a consequence, we conclude that $\mathcal{L}\overline u(t,x)=\overline u_t(t,x)-\Delta\overline u(t,x)-f(\overline u(t,x))\ge 0$ for all~$t\ge \tau_1$ and $x\in\overline{\Omega^+}$ such that $|x|\ge L$ and $\overline u(t,x)<1$. Since $f(1)=0$ and $u<1$ in $\R\times\overline{\Omega}$, the maximum principle then implies that 
$$\baa{r}
\displaystyle u(t,x)\le\overline{u}(t,x)\le\phi\Big(|x|-c(t-\tau_1+\tau)+\frac{N-1}{c}\ln (t-\tau_1+\tau)-L-B+C-\varrho(t-\tau_1)\Big)\vspace{3pt}\\
+\delta e^{-\delta\vartheta(t-\tau_1)}+\delta e^{-\mu(|x|-L)}\eaa$$
for all $t\ge \tau_1$ and $x\in\overline{\Omega^+}$ such that $|x|\ge L$. Finally, since $\phi$ is decreasing,~\eqref{sup-1} holds by taking~$z_1=-L-B+C-\varrho(+\infty)$ .

\vskip 0.3cm
\noindent\textit{Step 3: proof of \eqref{sub-1}.}
Since $u(t,\cdot)\to 1$ as $t\to +\infty$ locally uniformly in $\overline\Omega$ by the complete propagation condition \eqref{complete}, there exists $\tau_2>0$ such that 
\begin{equation}
\label{sup-t2-1}
u(t,x)\ge 1-\delta \ \text{for all} \ t\ge \tau_2\ \text{and}\ x\in\overline{\Omega^+} \ \text{with} \ L\le|x|\le L+B+c\tau-\frac{N-1}{c}\ln\tau.
\end{equation}
 For $t\ge \tau_2$ and $x\in\overline{\Omega^+}$ with $|x|\ge L$, let us set
\begin{equation*}
\underline u(t,x)=\max\big(\phi(\underline\xi(t,x))-\delta e^{-\delta\vartheta(t-\tau_2)}-\delta e^{-\mu(|x|-L)}, 0\big),
\end{equation*}
where 
\begin{equation*}
\underline\xi(t,x)=\zeta(t-\tau_2,x)+\varrho(t-\tau_2)=|x|-c(t-\tau_2+\tau)+\frac{N-1}{c}\ln (t-\tau_2+\tau)-L-B+C+\varrho(t-\tau_2).
\end{equation*}
Let us now check that $\underline u(t,x)$ is a subsolution of the problem satisfied by $u(t,x)$ for $t\ge \tau_2$ and~$x\in\overline{\Omega^+}$ with $|x|\ge L$.

Let us first check the initial and boundary conditions.  At time $t=\tau_2$, on the one hand, it follows  from~\eqref{sup-t2-1} that, for every $x\in\overline{\Omega^+}$ with $L\le |x|\le L+B+c\tau-((N-1)/c)\ln\tau$, there holds
$$u(\tau_2,x)\ge 1-\delta\ge 1-\delta-\delta e^{-\mu(|x|-L)}\ge \underline u(\tau_2,x).$$
On the other hand, for every $x\in\overline{\Omega^+}$ such that $|x|\ge L+B+c\tau-((N-1)/c)\ln\tau$, one has
$$\underline\xi(\tau_2,x)=|x|-c\tau+((N-1)/c)\ln\tau-L-B+C\ge C,$$
and it then follows from~\eqref{phi-1} that $\underline u(\tau_2,x)\le \max(\delta-\delta-\delta e^{-\mu(|x|-L)},0)=0< u(\tau_2,x)$. Therefore, $\underline u(\tau_2,x)\le u(\tau_2,x)$ for all $x\in\overline{\Omega^+}$ with~$|x|\ge L$. Next, for $t\ge \tau_2$ and $x\in\overline{\Omega^+}$ with~$|x|=L$, one has $\underline u(t,x)\le 1-\delta e^{-\delta\vartheta(t-\tau_2)}-\delta<1-\delta\le u(t,x)$ due to \eqref{sup-t2-1}. Moreover, it can be easily deduced that $\nu(x)\cdot\nabla \underline u(t,x)=0$ for every $t\ge \tau_2$ and $x\in\partial\Omega^+$ such that $|x|>L$ and $\underline u(t,x)>0$.

Let us now check that $\mathcal{L}\underline u(t,x)=\underline u_t(t,x)-\Delta\underline u(t,x)-f(\underline u(t,x))\le 0$ for all $t\ge \tau_2$ and~$x\in\overline{\Omega^+}$ such that $|x|\ge L$ and $\underline u(t,x)>0$. A straightforward computation shows that, for such a~$(t,x)$,
\begin{align*}
\mathcal{L}\underline u(t,x)= &\ f(\phi(\underline\xi(t,x)))-f(\underline u(t,x))+\delta^2 \vartheta'(t-\tau_2) e^{-\delta\vartheta(t-\tau_2)}+\mu^2\delta e^{-\mu(|x|-L)}\\
&\ -\frac{N-1}{|x|}\mu\delta e^{-\mu(|x|-L)}+\bigg(\varrho'(t-\tau_2)-\Big(\underbrace{\frac{N-1}{|x|}-\frac{N-1}{c(t-\tau_2+\tau)}}_{=A(t-\tau_2,x)\le(N-1)/L}\Big) \bigg)\phi'(\underline\xi(t,x)).
\end{align*}
As in Step~ 2, three cases can occur, namely: either $\zeta(t-\tau_2,x)>M+\vartheta(t-\tau_2)$, or $\zeta(t-\tau_2,x)<-M-\vartheta(t-\tau_2)$, or $|\zeta(t-\tau_2,x)|\le M+\vartheta(t-\tau_2)$.

Consider firstly the case
$$\zeta(t-\tau_2,x)>M+\vartheta(t-\tau_2).$$
One then has $\underline\xi(t,x)\ge \zeta(t-\tau_2,x)\!>\!M+\vartheta(t-\tau_2)\!>\!M_2+\vartheta(t-\tau_2)\ge C$. Hence, $0<\phi(\underline \xi(t,x))\le \delta$ and then $\underline u(t,x)\le \delta$ (remember also that $(t,x)$ is assumed to be such that $0<\underline{u}(t,x)$). One deduces from \eqref{delta-1} that $f(\phi(\underline \xi(t,x)))-f(\underline u(t,x))\le (f'(0)/2)(\delta e^{-\delta\vartheta(t-\tau_2)}+\delta e^{-\mu(|x|-L)})$. Moreover,  by virtue of \eqref{phi-1} and \eqref{A-inf-1} one has
$$-A(t-\tau_2,x)\,\phi'(\underline\xi(t,x))\le\frac{N-1}{L}\,Ke^{-\mu^*\underline\xi(t,x)/2}\le\frac{N-1}{L}\, Ke^{-\mu^*(M_2+\vartheta(t-\tau_2))/2}\le \delta^2 e^{-\delta\vartheta(t-\tau_2)}.$$
Therefore, it follows from \eqref{vartheta-1}, \eqref{mu-1}-\eqref{delta-1}, \eqref{varrho-1}, as well as the negativity of $\phi'$ and $f'(0)$, that 
\begin{align*}
\mathcal{L}\underline u(t,x)\le&\ \frac{f'(0)}{2}(\delta e^{-\delta\vartheta(t-\tau_2)}+\delta e^{-\mu(|x|-L)})+\delta^2 \vartheta'(t-\tau_2) e^{-\delta\vartheta(t-\tau_2)}+\mu^2\delta e^{-\mu(|x|-L)}\\
&\ -\frac{N-1}{|x|}\mu\delta e^{-\mu(|x|-L)}+\varrho'(t-\tau_2)\phi'(\underline \xi(t,x))+\delta^2  e^{-\delta\vartheta(t-\tau_2)}\\
\le &\ \Big(\frac{f'(0)}{2}+\delta\vartheta'(t-\tau_2)+\delta\Big)\delta e^{-\delta\vartheta(t-\tau_2)}+\Big(\frac{f'(0)}{2}+\mu^2\Big)\delta e^{-\mu(|x|-L)}\le 0.
\end{align*}

Consider secondly the case
$$\zeta(t-\tau_2,x)<-M-\vartheta(t-\tau_2).$$
One then has $\underline\xi(t,x)<-M-\vartheta(t-\tau_2)+\varrho(+\infty)\le -C$, which implies $1>\phi(\underline\xi(t,x))\ge 1-\delta$ and then $1>\underline u(t,x)\ge 1-3\delta$. By \eqref{delta-1} there holds $f(\phi(\underline \xi(t,x)))-f(\underline u(t,x))\le (f'(1)/2)(\delta e^{-\delta\vartheta(t-\tau_2)}+\delta e^{-\mu(|x|-L)})$. One also infers from~\eqref{phi-1} and~\eqref{A-inf-2} that
$$-A(t-\tau_2,x)\,\phi'(\underline\xi(t,x))\le\frac{N-1}{L}\,Ke^{\mu_*\underline\xi(t,x)/2}\le\frac{N-1}{L}\, Ke^{\mu_*(-M-\vartheta(t-\tau_2)+\varrho(+\infty))/2}\le \delta^2 e^{-\delta\vartheta(t-\tau_2)}.$$
It then follows from \eqref{vartheta-1}, \eqref{mu-1}-\eqref{delta-1}, \eqref{varrho-1}, as well as the negativity of $\phi'$ and $f'(1)$, that 
$$\baa{rcl}
\mathcal{L}\underline u(t,x) & \le & \displaystyle\frac{f'(1)}{2}(\delta e^{-\delta\vartheta(t-\tau_2)}+\delta e^{-\mu(|x|-L)})+\delta^2 \vartheta'(t-\tau_2) e^{-\delta\vartheta(t-\tau_2)}+\mu^2\delta e^{-\mu(|x|-L)}\vspace{3pt}\\
& & \displaystyle-\frac{N-1}{|x|}\mu\delta e^{-\mu(|x|-L)}+\varrho'(t-\tau_2)\phi'(\underline \xi(t,x))+\delta^2 e^{-\delta\vartheta(t-\tau_2)}\vspace{3pt}\\
& \le & \displaystyle\Big(\frac{f'(1)}{2}+\delta\vartheta'(t-\tau_2)+\delta\Big)\delta e^{-\delta \vartheta(t-\tau_2)}+\Big(\frac{f'(1)}{2}+\mu^2\Big)\delta e^{-\mu(|x|-L)}\le 0.\eaa$$

Eventually, let us consider the case that
$$|\zeta(t-\tau_2,x)|\le M+\vartheta(t-\tau_2).$$
One then observes from the definitions of $\Lambda$ and $\varrho$ in~\eqref{Lambda}-\eqref{varrho-1} that in this range there holds
\begin{equation}
\label{A-2}
\varrho'(t-\tau_2)-A(t-\tau_2,x)\ge \sigma\delta ( e^{-\delta\vartheta(t-\tau_2)}+e^{-\delta(t-\tau_2)})>0.
\end{equation} 
 Similarly as the preceding step, three subcases may occur. If $-C\le\underline\xi(t,x)\le C$, one then has
$$\baa{rcl}
\displaystyle |x|-L & \ge & \displaystyle c(t-\tau_2+\tau)-\frac{N-1}{c}\ln(t-\tau_2+\tau)+B-2C-\varrho(t-\tau_2)\\
& \ge & \displaystyle\frac{c}{2}(t-\tau_2+\tau)+B-2C-\varrho(+\infty)\ge\frac{c}{2}(t-\tau_2)+B-2C-\varrho(+\infty)\eaa$$
thanks to \eqref{tau}, whence $e^{-\mu(|x|-L)}\le e^{-\mu(c(t-\tau_2)/2+B-2C-\varrho(+\infty))}$. Moreover, $\phi'(\underline\xi(t,x))\le -\kappa<0$ and $f(\phi(\underline \xi(t,x)))-f(\underline u(t,x))\le \big(\max_{[0,1]}|f'|\big)(\delta e^{-\delta\vartheta(t-\tau_2)}+\delta e^{-\mu(|x|-L)})$. Therefore, since $B\ge 2C+\varrho(+\infty)$, one infers from~\eqref{vartheta-1},~\eqref{delta-1},~\eqref{kappa-sigma} and~\eqref{A-2} that 
\begin{align*}
\mathcal{L}\underline u(t,x)\le& \ \Big(\max_{[0,1]}|f'|\Big)\,(\delta e^{-\delta\vartheta(t-\tau_2)}+\delta e^{-\mu(|x|-L)})+\delta^2\vartheta'(t-\tau_2) e^{-\delta\vartheta(t-\tau_2)} +\mu^2\delta e^{-\mu(|x|-L)}\\
&\ -\frac{N-1}{|x|}\mu\delta e^{-\mu(|x|-L)}-\kappa\sigma\delta( e^{-\delta\vartheta(t-\tau_2)}+ e^{-\delta(t-\tau_2)})\\
\le& \ \Big(\max_{[0,1]}|f'|+\delta\vartheta'(t-\tau_2)-\kappa \sigma\Big)\delta e^{-\delta\vartheta(t-\tau_2)}+\!\Big(\!\max_{[0,1]}|f'|+\mu^2\Big)\delta e^{-\mu(|x|-L)}\!-\kappa \sigma\delta e^{-\delta(t-\tau_2)}\\
\le&\ \Big(\max_{[0,1]}|f'|+\mu^2\Big)\delta e^{-\mu(c(t-\tau_2)/2+B-2C-\varrho(+\infty))}-\kappa \sigma\delta e^{-\delta(t-\tau_2)}\\
\le& \ \Big[\Big(\max_{[0,1]}|f'|+\mu^2\Big)e^{-\mu(B-2C-\varrho(+\infty))}-\kappa \sigma\Big]\delta e^{-\delta(t-\tau_2)}\le 0.
\end{align*}
If $\underline \xi(t,x)\ge C$, one has $0<\phi(\underline\xi(t,x))\le \delta$ and $0<\underline u(t,x)\le \delta$. Moreover, one infers from~\eqref{delta-1} that $f(\phi(\underline \xi(t,x)))-f(\underline u(t,x))\le (f'(0)/2)(\delta e^{-\delta\vartheta(t-\tau_2)}+\delta e^{-\mu(|x|-L)})$. Therefore it follows from~\eqref{vartheta-1},~\eqref{mu-1}-\eqref{delta-1} and~\eqref{A-2}, as well as the negativity of $\phi'$ and $f'(0)$, that 
\begin{align*}
\mathcal{L}\underline u(t,x)\le& \ \frac{f'(0)}{2}(\delta e^{-\delta\vartheta(t-\tau_2)}+\delta e^{-\mu(|x|-L)})+\delta^2 \vartheta'(t-\tau_2) e^{-\delta\vartheta (t-\tau_2)}+\mu^2\delta e^{-\mu(|x|-L)}\\
\le &\ \Big(\frac{f'(0)}{2}+\delta\vartheta'(t-\tau_2)\Big)\delta e^{-\delta\vartheta (t-\tau_2)}+\Big(\frac{f'(0)}{2}+\mu^2\Big)\delta e^{-\mu(|x|-L)}\le 0.
\end{align*}
If $\underline \xi(t,x)\le -C$, one has $1>\phi(\underline\xi(t,x))\ge 1-\delta$ and then $1>\underline u(t,x)\ge 1-3\delta$. By virtue of~\eqref{vartheta-1},~\eqref{mu-1}-\eqref{delta-1} and~\eqref{A-2}, as well as the negativity of $\phi'$ and $f'(1)$, one gets that
\begin{align*}
\mathcal{L}\underline u(t,x)\le\Big(\frac{f'(1)}{2}+\delta\vartheta'(t-\tau_2)\Big)\delta e^{-\delta\vartheta(t-\tau_2)}+\Big(\frac{f'(1)}{2}+\mu^2\Big)\delta e^{-\mu(|x|-L)}\le 0.   		
\end{align*}

Consequently, we conclude that $\mathcal{L}\underline u(t,x)=\underline u_t(t,x)-\Delta\underline u(t,x)-f(\underline u(t,x))\le 0$ for all $t\ge \tau_2$ and $x\in\overline{\Omega^+}$ such that $|x|\ge L$ and $\underline u(t,x)>0$. Since $f(0)=0$ and $u>0$ in $\R\times\overline{\Omega}$, the maximum principle then implies that
$$\baa{r}
\displaystyle u(t,x)\ge\underline{u}(t,x)\ge\phi\Big(|x|-c(t-\tau_2+\tau)+\frac{N-1}{c}\ln (t-\tau_2+\tau)-L-B+C+\varrho(t-\tau_2)\Big)\vspace{3pt}\\
-\delta e^{-\delta\vartheta(t-\tau_2)}-\delta e^{-\mu(|x|-L)}\eaa$$
for all $t\ge \tau_2$ and $x\in\overline{\Omega^+}$ such that $|x|\ge L$. Choosing $z_2=-L$, property~\eqref{sub-1} then follows from the fact that $B\ge2C+\varrho(+\infty)\ge C+\varrho(+\infty)$ and the negativity of $\phi'$. The proof of Lemma~\ref{lemma-super sub} is thereby complete.
\end{proof}

\subsubsection*{Proof of Theorem~\ref{thm3}}

Let $\Omega$ be a funnel-shaped domain satisfying~\eqref{defOmega}-\eqref{h} with $\alpha>0$, and let $u$ be the solution of~\eqref{1} with past condition~\eqref{initial}, given in Proposition~\ref{thm1}. One assumes that $u$ propagates completely in the sense of~\eqref{complete}. First of all, we recall from~\eqref{unif1t} that, for every $\tau\in\R$, $u(t,x)\to1$ as $x_1\to-\infty$ uniformly with respect to $t\ge\tau$. Together with~\eqref{complete}, one infers that
$$\inf_{\overline{\Omega^-}\cup\{x\in\overline{\Omega^+}:|x|\le L\}}u(t,\cdot)\to1\ \hbox{ as $t\to+\infty$}.$$
With Lemma~\ref{lemma-super sub} and the limits $\phi(-\infty)=1$, $\phi(+\infty)=0$, it follows that, for any $\lambda\in(0,1)$, there is $r_0>0$ such that the upper level set $\mathcal{U}_\lambda(t)$ defined in~\eqref{def-level set} satisfies~\eqref{upper level set} for all $t$ large enough. In other words, the Hausdorff distance between the level set $\mathcal{E}_\lambda(t)$ and the expanding spherical surface of radius $ct-((N-1)/c)\ln t$ in $\overline{\Omega^+}$ remains bounded as $t\to+\infty$. 

Furthermore, from~\eqref{defw+},~\eqref{inequw+} and the positivity of $u$, one gets that, for any $\eta>0$, there is $t_\eta<0$ such that $0<u(t,\cdot)\le\eta$ in $\overline{\Omega^+}$ for all $t\le t_\eta$. In particular, by choosing any~$\eta$ small enough so that $f<0$ in $(0,\eta]$, it then easily follows from the maximum principle and parabolic estimates that $u(t,x)\to0$ as $x_1\to+\infty$, uniformly with respect to $t\le t_\eta$, and then also locally uniformly in $t\in\R$ again from parabolic estimates. Since $u(t,x)\to1$ as $x_1\to-\infty$ (at least) locally uniformly in $t\in\R$ by~\eqref{unif1t}, and since~\eqref{initial} and~\eqref{upper level set} hold, it is elementary to check that $u$ is a transition front in the sense of Definition~\ref{Def1} with sets $\Gamma_t$ and $\Omega^\pm_t$ defined by~\eqref{Gamma_t}-\eqref{Omega_t}. Moreover, $u$ then has a global mean speed equal to $c$.

To complete the proof of Theorem~\ref{thm3}, it remains to show that $u$ converges locally uniformly along any of its level sets to planar front profiles as $t\to+\infty$. To do so, let $\tau>0$, $\tau_1\in \mathbb{R}$, $\tau_2\in\mathbb{R}$, $z_1\in\mathbb{R}$, $z_2\in\mathbb{R}$, $\delta>0$ and $\mu>0$ be as in Lemma \ref{lemma-super sub}. For $t\ge\max\{\tau_1,\tau_2\}$ and $x\in\overline{\Omega^+}$ with $|x|\ge L$, there holds 
\begin{equation}
\label{super-sub-1}\baa{l}
\displaystyle\phi\Big(|x|-c(t-\tau_2+\tau)+\frac{N-1}{c}\ln (t-\tau_2+\tau) + z_2\Big)-\delta e^{-\delta\vartheta (t-\tau_2)}-\delta e^{-\mu(|x|-L)}\vspace{3pt}\\
\qquad\displaystyle\le u(t,x)\le\phi\Big(|x|\!-\!c(t\!-\!\tau_1\!+\!\tau)+\frac{N\!-\!1}{c}\ln(t\!-\!\tau_1\!+\!\tau)\!+\!z_1\!\Big)\!+\!\delta e^{-\delta\vartheta(t\!-\!\tau_1)}\!+\!\delta e^{-\mu(|x|\!-\!L)}.\eaa
\end{equation}
Consider now any $\lambda\in(0,1)$, any sequence $(t_n)_{n\in\mathbb{N}}$ such that $t_n\to+\infty$ as $n\to+\infty$, and any sequence $(x_n)_{n\in\mathbb{N}}$ in $\overline{\Omega}$ such that $u(t_n,x_n)=\lambda$. From the properties of the previous paragraphs, one infers that $x_n\in \overline{\Omega^+}$ for all $n$ large enough, and $|x_n|\to+\infty$ as $n\to+\infty$. Therefore, up to extraction of a subsequence, two cases can occur: either $d(x_n,\partial\Omega)\to+\infty$ as $n\to+\infty$, or~$\sup_{n\in\N}d(x_n,\partial\Omega)<+\infty$.

\textit{Case 1: $d(x_n,\partial \Omega)\to+\infty$ as $n\to+\infty$}. Up to extraction of a subsequence, there is a unit vector $e$ such that $x_n/|x_n|\to e$ as $n\to+\infty$. From standard parabolic estimates, the functions 
\begin{equation*}
u_n(t,x)=u(t+t_n,x+x_n)
\end{equation*}
converge in $C^{1,2}_{(t,x);loc}(\mathbb{R}\times\mathbb{R}^N)$, up to extraction of a subsequence, to a solution $u_\infty$ of 
$$(u_\infty)_t=\Delta u_\infty+f(u_\infty)~~\text{in}~\mathbb{R}\times\mathbb{R}^N,$$
satisfying $u_\infty(0,0)=\lambda$. It also follows from~\eqref{super-sub-1} that, for every $(t,x)\in\R\times\R^N$, one has
\begin{equation}
\label{tn-xn-1}
\baa{l}
\displaystyle\phi\Big(|x+x_n|-c(t+t_n-\tau_2+\tau)+\frac{N-1}{c}\ln (t+t_n-\tau_2+\tau) + z_2\Big)\vspace{3pt}\\
\qquad\qquad\qquad -\delta e^{-\delta\vartheta (t+t_n-\tau_2)}-\delta e^{-\mu(|x+x_n|-L)}\vspace{3pt}\\
\qquad\le u_n(t,x)\vspace{3pt}\\
\qquad\displaystyle\le\phi\Big(|x+x_n|-c(t+t_n-\tau_1+\tau)+\frac{N-1}{c}\ln (t+t_n-\tau_1+\tau)+z_1\Big)\vspace{3pt}\\
\qquad\qquad\qquad+\delta e^{-\delta\vartheta (t+t_n-\tau_1)}+\delta e^{-\mu(|x+x_n|-L)}\eaa
\end{equation} 
for all $n$ large enough. Since $u_n(0,0)=\lambda\in(0,1)$ for all $n\in\N$, one gets that, for every $t_0\in\R$, the sequence $(|x_n|-c(t_n+t_0)+((N-1)/c)\ln(t_n+t_0))_{n\in\mathbb{N}}$ is bounded. Moreover, since $\ln(t+t_n+t_0)-\ln(t_n+t_0)\to0$ as $n\to+\infty$ for every $(t_0,t)\in\R^2$, and since $|x+x_n|=|x_n|+x\cdot x_n/|x_n|+o(1)=|x_n|+x\cdot e+o(1)$ as $n\to+\infty$ for every $x\in\R^N$, the passage to the limit as $n\to+\infty$ in~\eqref{tn-xn-1} yields the existence of some real numbers $A$ and $B$ such that, for all $(t,x)\in\R\times\R^N$,
\begin{equation*}
\phi(x\cdot e-ct +A)\le u_\infty(t,x)\le \phi(x\cdot e-ct+B).
\end{equation*}
One concludes from~\cite[Theorem 3.1]{BH2007} and the property $u_\infty(0,0)=\lambda$, that 
\begin{equation*}
u_\infty(t,x)=\phi(x\cdot e-ct+\phi^{-1}(\lambda))~\text{ for all}~(t,x)\in\mathbb{R}\times\mathbb{R}^N.
\end{equation*}
Consequently, 
\begin{equation*}
u(t+t_n,x+x_n)-\phi(x\cdot e-c t+\phi^{-1}(\lambda))\to0~~\text{in}~C^{1,2}_{(t,x);loc}(\mathbb{R}\times\mathbb{R}^N)~\text{as}~n\to+\infty.
\end{equation*}
The previous limit, together with standard parabolic estimates and the compactness of the unit sphere of $\R^N$, yields the desired conclusion~\eqref{convtaun}.

\textit{Case 2: $\sup_{n\in\N}d(x_n,\partial\Omega)<+\infty$.} Up to extraction of a subsequence, one has $x_n/|x_n|\to e$ as~$n\to+\infty$, where $e=(e_1,e')$ is a unit vector such that $e_1>0$ and $|e'|=e_1\tan\alpha$. From standard parabolic estimates, there are then an open half-space $H$ of $\R^N$ such that $e$ is parallel to $\partial H$ and a~$C^{1,2}_{(t,x)}(\R\times\overline{H})$ solution~$u_\infty$ of
\begin{equation}
\label{half space}\left\{\baa{ll}
(u_\infty)_t=\Delta u_\infty+f(u_\infty) & \text{in}~\mathbb{R}\times\overline{H},\vspace{3pt}\\
\nu\cdot\nabla u_\infty=0 & \text{on}~\mathbb{R}\times\partial H,\eaa\right.
\end{equation}
such that, up to extraction of a subsequence, $\|u(\cdot+t_n,\cdot+x_n)-u_\infty\|_{C^{1,2}_{(t,x)}(K\,\cap\,(\R\times(\overline{\Omega}-x_n)))}\to0$ as $n\to+\infty$ for every compact set $K\subset\R\times\overline{H}$. Following an analogous analysis as the preceding case, it comes that 
\begin{equation}\label{phiuinfty}
\phi(x\cdot e-ct+A)\le u_\infty(t,x)\le \phi(x\cdot e-ct+B)\hbox{ for all $(t,x)\in\mathbb{R}\times\overline{H}$},
\end{equation}
for some real numbers $A$ and $B$. Let us now call $\mathcal{R}$ the orthogonal reflection of $\R^N$ with respect to the hyperplane $\partial H$, and let us define
$$v_\infty(t,x)=\left\{\baa{ll}
u_\infty(t,x) & \hbox{ if }(t,x)\in\R\times\overline{H},\vspace{3pt}\\
u_\infty(t,\mathcal{R}x) &  \hbox{ if }(t,x)\in\R\times(\R^N\!\setminus\!\overline{H}),\eaa\right.$$
Thanks to the Neumann boundary conditions satisfied by $u_\infty$ on $\R\times\partial H$, the function $v_\infty$ is then a $C^{1,2}_{(t,x)}(\R\times\R^N)$ solution of the equation $(v_\infty)_t=\Delta v_\infty+f(v_\infty)$ in $\R\times\R^N$. Since $e$ is parallel to $\partial H$, one also gets from~\eqref{phiuinfty} that $\phi(x\cdot e-ct+A)\le v_\infty(t,x)\le \phi(x\cdot e-ct+B)$ for all $(t,x)\in\R\times\R^N$. It follows as in the previous case that $v_\infty(t,x)=\phi(x\cdot e-ct+\phi^{-1}(\lambda))$ for all $(t,x)\in\R\times\R^N$, that is,
$$u_\infty(t,x)=\phi(x\cdot e-ct+\phi^{-1}(\lambda))\ \hbox{ for all $(t,x)\in\R\times\overline{H}$},$$
which yields the desired conclusion. The proof of Theorem~\ref{thm3} is thereby complete.\hfill$\Box$


\section{Sufficient conditions on $(R,\alpha)$ for complete propagation or for blocking}\label{Sec-5}

In this section, we show Theorems~\ref{thm5} and~\ref{thm6}, which provide sufficient conditions on the parameters $(R,\alpha)$ such that the solutions $u$ of~\eqref{1} and~\eqref{initial} given in Proposition~\ref{thm1} propagate completely or are blocked.


\subsection{Complete propagation for $R\ge R_0$ and $\alpha\in(0,\pi/2)$: proof of Theo\-rem~\ref{thm5}}\label{Sec-51}

Consider any $\alpha\in(0,\pi/2)$, and assume that $R\ge R_0$, where $R_0$ is given in Lemma~\ref{lemsub}. Remember that the limit $0<u_\infty(x)\le1$ of $u(t,x)$ as $t\to+\infty$ solves~\eqref{u-infty} and $u_\infty(x)\to1$ as $x_1\to-\infty$. It then follows from Lemmas~\ref{lemsub} and~\ref{lemliouville1}, with $x_0=(-A,0,\cdots,0)$ and $A>0$ large enough, that $u_\infty\equiv1$ in $\overline{\Omega}$, that is, $u$ propagates completely.\hfill$\Box$


\subsection{Blocking for $R\ll 1$ and $\alpha$ not too small: proof of Theorem~\ref{thm6}}\label{Sec-52}

This subsection is devoted to the proof of Theorem \ref{thm6}. Throughout this subsection, we assume that $N\ge3$, and we are given
$$\alpha_*\in\Big(0,\frac{\pi}{2}\Big)\ \hbox{ and }\ L_*>0.$$
We will consider domains~$\Omega$ satisfying~\eqref{defOmega}-\eqref{h} whose left parts $\Omega^-=\{x\in\Omega:x_1\le0\}$ have cross sections of small radius $R$, whereas the angles $\alpha$ of the right parts $\Omega^+=\{x\in\Omega:x_1>0\}$ are not too small, namely
$$\alpha_*\le\alpha<\frac{\pi}{2}.$$
We also always assume that
$$0<R<L\le L_*$$
in~\eqref{defOmega}-\eqref{h}. We then aim at establishing the existence of a non-constant supersolution of~\eqref{u-infty} that will block the propagation of the solution $u$ of~\eqref{1} satisfying the past condition~\eqref{initial}. More precisely, we will prove that, when the measure $|\{x\in\Omega: -1<x_1<0\}|$ is sufficiently small (that is, when $R>0$ is small enough), then there exists a supersolution~$\overline{u}$ of~\eqref{u-infty} such that $\overline{u}(x)=1$ for all $x\in\overline{\Omega}$ with $x_1\le -1$ and $\overline{u}(x)\to 0$ as~$x_1\to +\infty$. The proof is based on the construction of solutions of reduced problems in truncated domains, which is itself based on variational arguments as in~\cite{BHM2009,BBC2016}.

\subsubsection*{Some notations}

To apply this scheme, let us first list the definitions of some sets that will be used in the sequel:
\be\label{sets}\left\{\baa{llll}
\Omega^-_{R,1} & = & \{x\in\Omega:-1<x_1\le0\}, & \vspace{3pt}\\
\Gamma^-_{R,1} & = & \{x\in\overline{\Omega}:x_1=-1\}, & \vspace{3pt}\\
\Omega'_{R,\alpha} & = & \{x\in\Omega: x_1>-1\}, & \vspace{3pt}\\
\Omega'_{R,\alpha,r} & = & \Omega^-_{R,1}\cup\big\{x\in\Omega^+:|x|<r\big\} & \hbox{for }r\ge L_*,\vspace{3pt}\\
C^+_{\alpha,r} & = & \big\{x\in\Omega^+:|x'|<x_1\tan\alpha,\,|x|<r\big\} & \hbox{for $r\ge L_*$},\vspace{3pt}\\
\Gamma^+_{\alpha,r} & = & \big\{x\in\overline{\Omega^+}:|x|=r\big\} & \hbox{for $r\ge L_*$}.\eaa\right.
\ee
Notice that $\Omega^-_{R,1}$ and $\Gamma^-_{R,1}$ are actually independent of $\alpha$, and that $C^+_{\alpha,r}$ (a conical sector) and~$\Gamma^+_{\alpha,r}$ are independent of $R$ and $L$ with $0<R<L\le L_*$ in~\eqref{defOmega}-\eqref{h}.
\begin{figure}[H]
	\centering
	\includegraphics[scale=1.7]{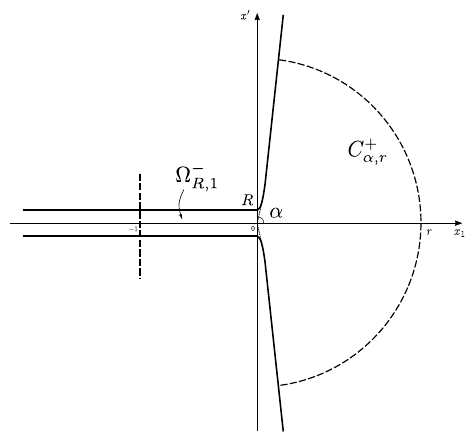}
	\caption{\small Illustration of some domains that will be used in the proof of Theorem \ref{thm6}.}
\end{figure}

We will consider a reduced elliptic problem in $\Omega'_{R,\alpha}$:
\begin{align}
\label{reduced}
\begin{cases}
\Delta w+f(w)=0 ~&\text{in}~\Omega'_{R,\alpha},\\
\nu\cdot\nabla w=0~&\text{on}~\partial\Omega'_{R,\alpha}\backslash \Gamma^-_{R,1},\\
w=1 ~&\text{on}~\Gamma^-_{R,1},\\
w(x)\to 0~&\text{as}~|x|\to+\infty~\text{in}~\overline{\Omega'_{R,\alpha}}.
\end{cases}
\end{align}
We shall prove the existence of a positive $C^2(\overline{\Omega'_{R,\alpha}})$ solution $w$ of \eqref{reduced}. Such a solution $w$, extended by $1$ in $\overline{\Omega}\!\setminus\!\overline{\Omega'_{R,\alpha}}$, will give rise to a supersolution $\overline{u}$ of~\eqref{u-infty} which will block the propagation of the solution $u$ of~\eqref{1} with past condition~\eqref{initial}.

For this purpose, we first consider the corresponding truncated problem in the domain $\Omega'_{R,\alpha,r}$ (for $r\ge L_*$), and show that the elliptic problem
\begin{align}
\label{truncated}
\begin{cases}
\Delta w_r+f(w_r)=0 ~&\text{in}~\Omega'_{R,\alpha,r},\vspace{3pt}\\
\nu\cdot\nabla w_r=0~&\text{on}~\partial\Omega'_{R,\alpha,r}\!\setminus\!(\Gamma^-_{R,1}\cup\Gamma^+_{\alpha,r}),\vspace{3pt}\\
w_r=1 ~&\text{on}~\Gamma^-_{R,1},\vspace{3pt}\\
w_r=0~&\text{on}~\Gamma^+_{\alpha,r},
\end{cases}
\end{align}
admits a $C^2(\overline{\Omega'_{R,\alpha,r}})$ solution $w_r$ such that $0<w_r<1$ in $\overline{\Omega'_{R,\alpha,r}}\!\setminus\!(\Gamma^-_{R,1}\cup\Gamma^+_{\alpha,r})$. Then, we will prove that $w_r\to w$ as $r\to+\infty$ locally uniformly in $\overline{\Omega'_{R,\alpha}}$, with $w$ satisfying~\eqref{reduced}.

\subsubsection*{Truncated problem~\eqref{truncated} in $\Omega'_{R,\alpha,r}$}

For any bounded measurable subset $D$ of $\R^N$, let us define the functional
$$H^1(D)\ni w\mapsto J_D(w)=\int_D\frac{|\nabla w|^2}{2}+F(w),\footnote{We equip $H^1(D)$ with the norm $\|w\|_{H^1(D)}=\sqrt{\|\,|\nabla w|\,\|_{L^2(D)}^2+\|w\|_{L^2(D)}^2}$.}$$
where $F(t)=\int_t^1f(s)ds$. From \eqref{f-bistable-1}-\eqref{f-bistable-2} and the affine extension of $f$ outside the interval~$[0,1]$, there exists $\kappa>0$ such that
\be\label{defkappa}
0\le\kappa(t-1)^2\le F(t)\le\frac{1+t^2}{\kappa}
\ee
for all $t\in\mathbb{R}$ (hence, $J_D$ is well defined in $H^1(D)$ for every bounded measurable subset $D$ of~$\R^N$). For $r\ge L_*$, and $0<R<L\le L_*$ and $\alpha_*\le\alpha<\pi/2$ in~\eqref{defOmega}-\eqref{h}, define now
\be\label{defHRalphar}
H_{R,\alpha,r}=\big\{w\in H^1(\Omega'_{R,\alpha,r}): w=1~\text{on}~\Gamma^-_{R,1}\text{ and }w=0~\text{on}~\Gamma^+_{\alpha,r}\big\},
\ee
where the equalities on $\Gamma^-_{R,1}$ and $\Gamma^+_{\alpha,r}$ are understood in the sense of trace. We aim at finding a local minimizer of $J_{\Omega'_{R,\alpha,r}}$ belonging to $H_{R,\alpha,r}$. That will lead to the existence of a solution to~\eqref{truncated}. 

We start with the following result on the functional $J_{C^+_{\alpha,r}}$, where the conical sectors~$C^+_{\alpha,r}$ are defined in~\eqref{sets} (we recall that these sets $C^+_{\alpha,r}$ are independent of $R$ and $L$).

\begin{lemma}
\label{lem5.1}
The function $0$ is a strict local minimum of $J_{C^+_{\alpha,r}}$ in the space $H^1(C^+_{\alpha,r})$ and, more precisely, there exist $\sigma>0$ and $\delta>0$ such that, for all $\alpha\in[\alpha_*,\pi/2)$, $r\ge L_*$, and $w\in H^1(C^+_{\alpha,r})$ with $\|w\|_{H^1(C^+_{\alpha,r})}\le \delta$, there holds
$$J_{C^+_{\alpha,r}}(w)\ge J_{C^+_{\alpha,r}}(0)+\sigma\|w\|^2_{H^1(C^+_{\alpha,r})}.$$
\end{lemma}

\begin{proof}
Throughout the proof, $\alpha\in[\alpha_*,\pi/2)$ and $r\ge L_*$ are arbitrary. First observe, from the Taylor expansion and the affine expansion of $f$ outside $[0,1]$, that there exist a continuous bounded function $\eta:\R\to\R$ such that $\eta(0)=0$ and
$$F(t)=F(0)+F'(0)t+\frac{F''(0)}{2}\,t^2+\eta(t)\,t^2=F(0)-\frac{f'(0)}{2}\,t^2+\eta(t)\,t^2$$
for all $t\in\R$. Therefore, by setting
$$\sigma=\min\Big(\frac16,\frac{-f'(0)}{6}\Big)>0,$$
we have
\begin{equation}
\label{estimate1}
J_{C^+_{\alpha,r}}(w)-J_{C^+_{\alpha,r}}(0)=\!\int_{C^+_{\alpha,r}}\!\!\frac{|\nabla w|^2}{2}-\frac{f'(0)}{2}\,w^2+\eta(w)\,w^2\ge3\sigma\|w\|_{H^1(C^+_{\alpha,r})}^2-\int_{C^+_{\alpha,r}}\!\!|\eta(w)|\,w^2
\end{equation}
for all $w\in H^1(C^+_{\alpha,r})$.

Define
$$p^*=\frac{2N}{N-2}\in(2,+\infty).$$
From Sobolev embedding theorem and the uniform (with respect to $\alpha\in[\alpha_*,\pi/2)$) Lipschitz continuity of the conical sectors $C^+_{\alpha,L_*}$, there is a positive constant $C$ (depending on $\alpha_*$, $L_*$ and~$N$, but independent of $\alpha\in[\alpha_*,\pi/2)$ and $r\ge L_*$) such that
$$\|v\|_{L^{p^*}(C^+_{\alpha,L_*})}\le C\|v\|_{H^1(C^+_{\alpha,L_*})}$$
for all $\alpha\in[\alpha_*,\pi/2)$ and $v\in H^1(C^+_{\alpha,L_*})$.\footnote{We here use the assumption $N\ge3$. In dimension $N=2$, the sets $H^1(C^+_{\alpha,L_*})$ are not embedded into $L^\infty(C^+_{\alpha,L_*})$, and the following arguments would not work as such in dimension $N=2$.} On the other hand, since the function $\eta$ is continuous, bounded and vanishes at $0$, there is a positive constant $C'$ (independent of $\alpha$ and $r$) such that
$$|\eta(t)|\le\sigma+C'|t|^{p^*-2}$$
for all $t\in\R$. Hence, for all $\alpha\in[\alpha_*,\pi/2)$, $r\ge L_*$ and $w\in H^1(C^+_{\alpha,r})$, there holds, with $v=w(r\cdot/L_*)\in H^1(C^+_{\alpha,L_*})$,
$$\baa{rcl}
\displaystyle\int_{C^+_{\alpha,r}}|\eta(w)|\,w^2 & \le & \displaystyle\sigma\int_{C^+_{\alpha,r}}w^2+C'\int_{C^+_{\alpha,r}}|w|^{p^*}\vspace{3pt}\\
& \le & \displaystyle\sigma\|w\|_{H^1(C^+_{\alpha,r})}^2+\frac{C'r^N}{L_*^N}\|v\|_{L^{p^*}(C^+_{\alpha,L_*})}^{p^*}\vspace{3pt}\\
& \le & \displaystyle\sigma\|w\|_{H^1(C^+_{\alpha,r})}^2+\frac{C'C^{p^*}r^N}{L_*^N}\|v\|_{H^1(C^+_{\alpha,L_*})}^{p^*}\vspace{3pt}\\
& = & \displaystyle\sigma\|w\|_{H^1(C^+_{\alpha,r})}^2+\frac{C'C^{p^*}r^N}{L_*^N}\Big(\int_{C^+_{\alpha,r}}\frac{L_*^{N-2}}{r^{N-2}}|\nabla w|^2+\int_{C^+_{\alpha,r}}\frac{L_*^N}{r^N}w^2\Big)^{p^*/2}\vspace{3pt}\\
& \le & \displaystyle\sigma\|w\|_{H^1(C^+_{\alpha,r})}^2+C'C^{p^*}\Big(\frac{L_*}{r}\Big)^{(N-2)p^*/2-N}\|w\|_{H^1(C^+_{\alpha,r})}^{p^*}\vspace{3pt}\\
& = & \sigma\|w\|_{H^1(C^+_{\alpha,r})}^2+C'C^{p^*}\|w\|_{H^1(C^+_{\alpha,r})}^{p^*}.\eaa$$
Together with~\eqref{estimate1}, one gets that
$$J_{C^+_{\alpha,r}}(w)-J_{C^+_{\alpha,r}}(0)\ge2\sigma\|w\|_{H^1(C^+_{\alpha,r})}^2-C'C^{p^*}\|w\|_{H^1(C^+_{\alpha,r})}^{p^*}\ge\sigma\|w\|_{H^1(C^+_{\alpha,r})}^2$$
for all $\alpha\in[\alpha_*,\pi/2)$, $r\ge L_*$ and $w\in H^1(C^+_{\alpha,r})$ such that
$$\|w\|_{H^1(C^+_{\alpha,r})}\le\delta:=\Big(\frac{\sigma}{C'C^{p^*}}\Big)^{1/(p^*-2)}.$$
Since the positive constants $\sigma$ and $\delta$ do not depend on $\alpha\in[\alpha_*,\pi/2)$ and $r\ge L_*$, the proof of Lemma~\ref{lem5.1} is thereby complete.
\end{proof}

Next, let us focus on the domain $\Omega'_{R,\alpha,r}$, with $0<R<L\le L_*$, $\alpha\in[\alpha_*,\pi/2)$, and $r\ge L_*$. We define the following function $w_0$ in $\overline{\Omega'_{R,\alpha,r}}$ by:
\begin{align*}
w_0(x)=
\begin{cases}
-x_1~~&\text{if}~x\in\overline{\Omega'_{R,\alpha,r}}\hbox{ with }x_1\le0,\vspace{3pt}\\
0~~&\text{if}~x\in\overline{\Omega'_{R,\alpha,r}}\hbox{ with }x_1>0.
\end{cases}
\end{align*}
It is immediate to see that $w_0\in H_{R,\alpha,r}$, with $H_{R,\alpha,r}$ defined in~\eqref{defHRalphar}.

\begin{lemma}
\label{lem5.3}
Let $\delta>0$ be as in Lemma~$\ref{lem5.1}$. Then there exist $R_*\in(0,L_*)$ and $\gamma>0$ such that, for every funnel-shaped domain $\Omega$ satisfying~\eqref{defOmega}-\eqref{h} with $\alpha\in[\alpha_*,\pi/2)$, $0<R\le R_*$ and $0<R<L\le L_*$, for every $r\ge L_*$, and for every $w\in H_{R,\alpha,r}$ with $\|w-w_0\|_{H^1(\Omega'_{R,\alpha,r})}=\delta$, there holds 
$$J_{\Omega'_{R,\alpha,r}}(w)\ge J_{\Omega'_{R,\alpha,r}}(w_0)+\gamma.$$
\end{lemma}

\begin{proof} 
Let $\delta>0$ and $\sigma>0$ be as in Lemma~\ref{lem5.1}. Throughout the proof, $\alpha\in[\alpha_*,\pi/2)$ and $r\ge L_*$ are arbitrary. We consider funnel-shaped domains $\Omega$ satisfying~\eqref{defOmega}-\eqref{h}, with a parameter $R$ satisfying $0<R<L\le L_*$, and some further restrictions on $R$ will appear later. Consider any $w\in H_{R,\alpha,r}$ with
$$\|w-w_0\|_{H^1(\Omega'_{R,\alpha,r})}=\delta.$$
In order to estimate $J_{\Omega'_{R,\alpha,r}}(w)-J_{\Omega'_{R,\alpha,r}}(w_0)$, we decompose the integrals over two disjoint subsets of~$\Omega'_{R,\alpha,r}$, namely $\Omega^-_{R,1}\cup S^+_{R,\alpha}$ and $C^+_{\alpha,r}$, with
\be\label{defS+}
S^+_{R,\alpha}=\Omega'_{R,\alpha,r}\!\setminus\!\big(\Omega^-_{R,1}\cup C^+_{\alpha,r})=\big\{(x_1,x')\in\R^N:0\!<\!x_1\!<\!L\cos\alpha,\,x_1\tan\alpha\le|x'|<h(x_1)\big\},
\ee
where the function $h$ is as in~\eqref{defOmega}-\eqref{h} (notice that $S^+_{R,\alpha}$ depends on $\Omega$ but not on $r$, since $L\le L_*\le r$). One has
$$J_{\Omega'_{R,\alpha,r}}(w)-J_{\Omega'_{R,\alpha,r}}(w_0)=J_{\Omega^-_{R,1}\cup S^+_{R,\alpha}}(w)-J_{\Omega^-_{R,1}\cup S^+_{R,\alpha}}(w_0)+J_{C^+_{\alpha,r}}(w)-J_{C^+_{\alpha,r}}(w_0).$$
Since $\|w\|_{H^1(C^+_{\alpha,r})}=\|w-w_0\|_{H^1(C^+_{\alpha,r})}\le\|w-w_0\|_{H^1(\Omega'_{R,\alpha,r})}=\delta$, Lemma~\ref{lem5.1} yields
$$J_{C^+_{\alpha,r}}(w)-J_{C^+_{\alpha,r}}(w_0)=J_{C^+_{\alpha,r}}(w)-J_{C^+_{\alpha,r}}(0)\ge\sigma\|w\|_{H^1(C^+_{\alpha,r})}^2=\sigma\|w-w_0\|_{H^1(C^+_{\alpha,r})}^2,$$
hence
\be\label{ineq11}
J_{\Omega'_{R,\alpha,r}}(w)-J_{\Omega'_{R,\alpha,r}}(w_0)\ge J_{\Omega^-_{R,1}\cup S^+_{R,\alpha}}(w)-J_{\Omega^-_{R,1}\cup S^+_{R,\alpha}}(w_0)+\sigma\|w-w_0\|_{H^1(C^+_{\alpha,r})}^2.
\ee

Let us now estimate $J_{\Omega^-_{R,1}\cup S^+_{R,\alpha}}(w)-J_{\Omega^-_{R,1}\cup S^+_{R,\alpha}}(w_0)$. On the one hand, with
$$\rho:=\min\Big(\frac12,\kappa\Big)>0$$
and $\kappa>0$ as in~\eqref{defkappa}, there holds
\be\label{ineq12}\baa{rcl}
\displaystyle J_{\Omega^-_{R,1}\cup S^+_{R,\alpha}}(w)=\int_{\Omega^-_{R,1}\cup S^+_{R,\alpha}}\frac{|\nabla(w-1)|^2}{2}+F(w) & \!\!\!\ge\!\!\! & \displaystyle\int_{\Omega^-_{R,1}\cup S^+_{R,\alpha}}\frac{|\nabla(w-1)|^2}{2}+\kappa\,(w-1)^2\vspace{3pt}\\
& \!\!\!\ge\!\!\! & \displaystyle\rho\|w-1\|_{H^1(\Omega^-_{R,1}\cup S^+_{R,\alpha})}^2.\eaa
\ee
On the other hand,
$$\baa{rcl}
\displaystyle J_{\Omega^-_{R,1}\cup S^+_{R,\alpha}}(w_0) & = & \displaystyle\Big(\int_{\Omega^-_{R,1}}\frac{|\nabla w_0|^2}{2}+F(w_0)\Big)+F(0)\,|S^+_{R,\alpha}|\vspace{3pt}\\
& \le & \displaystyle\Big(\frac12+\max_{[0,1]}F\Big)\,|\Omega^-_{R,1}|+F(0)\,|S^+_{R,\alpha}|=\Big(\frac12+F(\theta)\Big)\omega_{N-1}R^{N-1}+F(0)\,|S^+_{R,\alpha}|,\eaa$$
where $\omega_{N-1}$ denotes the $(N-1)$-dimensional Lebesgue measure of the unit Euclidean ball in~$\R^{N-1}$. Since $0<R<L\le L_*$ and $h$ satisfies~\eqref{defOmega}-\eqref{h}, one has $0\le h(x_1)-x_1\tan\alpha\le h(0)=R$ for all $x_1\in[0,L\cos\alpha]$, hence
$$\baa{rcl}
|S^+_{R,\alpha}| & = & \displaystyle\int_0^{L\cos\alpha}\omega_{N-1}\big(h(x_1)^{N-1}-(x_1\tan\alpha)^{N-1}\big)dx_1\vspace{3pt}\\
& \le & \displaystyle\int_0^{L\cos\alpha}\omega_{N-1}\big((x_1\tan\alpha+R)^{N-1}-(x_1\tan\alpha)^{N-1}\big)dx_1\vspace{3pt}\\
& = & \displaystyle\frac{\omega_{N-1}}{N\tan\alpha}\big((L\sin\alpha+R)^N-R^N-(L\sin\alpha)^N\big)\vspace{3pt}\\
& \le & \displaystyle\omega_{N-1}\cot\alpha_*(L_*\sin\alpha+L_*)^{N-1}R\le\omega_{N-1}2^{N-1}L_*^{N-1}R\cot\alpha_*.\eaa$$
Therefore,
$$J_{\Omega^-_{R,1}\cup S^+_{R,\alpha}}(w_0)\le\Big(\frac12+F(\theta)\Big)\omega_{N-1}R^{N-1}+F(0)\omega_{N-1}2^{N-1}L_*^{N-1}R\cot\alpha_*$$
and, together with~\eqref{ineq12}, 
$$\baa{l}
J_{\Omega^-_{R,1}\cup S^+_{R,\alpha}}(w)-J_{\Omega^-_{R,1}\cup S^+_{R,\alpha}}(w_0)\vspace{3pt}\\
\quad\ge\displaystyle\rho\|w-1\|_{H^1(\Omega^-_{R,1}\cup S^+_{R,\alpha})}^2-\Big(\frac12+F(\theta)\Big)\omega_{N-1}R^{N-1}-F(0)\omega_{N-1}2^{N-1}L_*^{N-1}R\cot\alpha_*\vspace{3pt}\\
\quad\ge\displaystyle\frac{\rho}{2}\|w-w_0\|_{H^1(\Omega^-_{R,1}\cup S^+_{R,\alpha})}^2\!-\rho\|w_0-1\|_{H^1(\Omega^-_{R,1}\cup S^+_{R,\alpha})}^2\!-\!\Big(\frac12\!+\!F(\theta)\Big)\omega_{N-1}R^{N-1}\vspace{3pt}\\
\quad\qquad-F(0)\omega_{N-1}2^{N-1}L_*^{N-1}R\cot\alpha_*\vspace{3pt}\\
\quad\ge\displaystyle\frac{\rho}{2}\|w\!-\!w_0\|_{H^1(\Omega^-_{R,1}\cup S^+_{R,\alpha})}^2\!-\!\Big(\frac{4\rho}{3}\!+\!\frac12\!+\!F(\theta)\Big)\omega_{N-1}R^{N-1}\!-\!(\rho\!+\!F(0))\omega_{N-1}2^{N-1}L_*^{N-1}R\cot\alpha_*.\eaa$$
Putting the previous inequality into~\eqref{ineq11}, one gets that
$$\baa{l}
J_{\Omega'_{R,\alpha,r}}(w)-J_{\Omega'_{R,\alpha,r}}(w_0)\vspace{3pt}\\
\quad\displaystyle\ge\beta\|w-w_0\|_{H^1(\Omega'_{R,\alpha,r})}^2-\Big(\frac{4\rho}{3}\!+\!\frac12\!+\!F(\theta)\Big)\omega_{N-1}R^{N-1}-(\rho\!+\!F(0))\omega_{N-1}2^{N-1}L_*^{N-1}R\cot\alpha_*\vspace{3pt}\\
\quad\displaystyle=\beta\delta^2-\Big(\frac{4\rho}{3}+\frac12+F(\theta)\Big)\omega_{N-1}R^{N-1}-(\rho+F(0))\omega_{N-1}2^{N-1}L_*^{N-1}R\cot\alpha_*\eaa$$
with $\beta:=\min(\sigma,\rho/2)>0$.

Finally, since the positive constants $\beta$, $\delta$, $\rho$ are independent of $\alpha$, $R$, $L$ and $r$ with $\alpha\in[\alpha_*,\pi/2)$ and $0<R<L\le L_*\le r$, there are then some positive real numbers  $R_*\in(0,L_*)$ and $\gamma>0$ such that $J_{\Omega'_{R,\alpha,r}}(w)-J_{\Omega'_{R,\alpha,r}}(w_0)\ge\gamma$ for all $\alpha\in[\alpha_*,\pi/2)$, $0<R\le R_*$, $0<R<L\le L_*\le r$ and $w\in H_{R,\alpha,r}$ with $\|w-w_0\|_{H^1(\Omega'_{R,\alpha,r})}=\delta$. The proof of Lemma~\ref{lem5.3} is thereby complete.
\end{proof}

\subsubsection*{End of the proof of Theorem~\ref{thm6}}

Let $\delta>0$, $R_*\in(0,L_*)$ and $\gamma>0$ be as in Lemma~\ref{lem5.3}. Let us then fix any funnel-shaped domain $\Omega$ satisfying~\eqref{defOmega}-\eqref{h} with $\alpha\in[\alpha_*,\pi/2)$, $0<R\le R_*$ and $0<R<L\le L_*$, and let us show that the solution $u$ of~\eqref{1} with the past condition~\eqref{initial}, given in Proposition~\ref{thm1}, is blocked in the sense of~\eqref{blocking}.

First of all, from Lemma~\ref{lem5.3}, for any $r\ge L_*$, the nonnegative functional $J_{\Omega'_{R,\alpha,r}}$ admits a local minimizer $w_r$ in $H_{R,\alpha,r}$ satisfying $\|w_r-w_0\|_{H^1(\Omega'_{R,\alpha,r})}<\delta$. This function $w_r$ is then a weak solution of the elliptic  problem~\eqref{truncated} and, since $f>0$ in $(-\infty,0)$ and $f<0$ in $(1,+\infty)$, one has $0\le w_r\le1$ almost everywhere in $\Omega'_{R,\alpha,r}$ and standard elliptic estimates imply that $w_r$ is a classical $C^2(\overline{\Omega'_{R,\alpha,r}})$ solution of~\eqref{truncated}, with $0<w_r<1$ in $\overline{\Omega'_{R,\alpha,r}}\setminus(\Gamma^-_{R,1}\cup\Gamma^+_{\alpha,r})$ (notice that $\Gamma^-_{R,1}$ and $\Gamma^+_{\alpha,r}$ meet $\partial\Omega$ orthogonally).

Remembering the definition of $\Omega'_{R,\alpha}$ in~\eqref{sets}, it follows from standard elliptic estimates that there is a sequence $(r_n)_{n\in\N}$ diverging to $+\infty$ such that the functions $w_{r_n}$ converge in $C^2_{loc}(\overline{\Omega'_{R,\alpha}})$ to a $C^2(\overline{\Omega'_{R,\alpha}})$ function $w$ solving
\begin{align}
\begin{cases}
\Delta w+f(w)=0 ~&\text{in}~\Omega'_{R,\alpha},\\
\nu\cdot\nabla w=0~&\text{on}~\partial\Omega'_{R,\alpha}\!\setminus\!\Gamma^-_{R,1},\\
w=1~&\text{on}~\Gamma^-_{R,1},
\end{cases}
\end{align}
and $0<w\le1$ in $\overline{\Omega'_{R,\alpha}}$ from the strong maximum principle. Furthermore, for any bounded measurable set $D\subset\Omega^+$, one has, for all $n$ large enough, $D\subset\Omega'_{R,\alpha,r_n}$ and $\|w_{r_n}\|_{L^2(D)}=\|w_{r_n}-w_0\|_{L^2(D)}\le\|w_{r_n}-w_0\|_{H^1(\Omega'_{R,\alpha,r_n})}<\delta$. Hence, $\|w\|_{L^2(D)}\le\delta$, and, since $\delta$ is independent of $D$, one gets that $\|w\|_{L^2(\Omega^+)}\le\delta$ by the monotone convergence theorem. Since $|\nabla w|$ is bounded from standard elliptic estimates, one infers that $w(x)\to0$ as $|x|\to+\infty$ in $\overline{\Omega'_{R,\alpha}}$. In other words, $w$ solves~\eqref{reduced}.

We now extend $w$ in $\overline{\Omega}\!\setminus\!\overline{\Omega'_{R,\alpha}}$ by $1$, namely we define
\begin{align*}
\overline u(x)=\begin{cases}
w(x)~&\text{if}~x\in\overline{\Omega'_{R,\alpha}},\vspace{3pt}\\
1,~&\text{if}~x\in\overline{\Omega}\!\setminus\!\overline{\Omega'_{R,\alpha}}.
\end{cases}
\end{align*}
Since $f(1)=0$ and $0<w\le1$ is a classical solution of~\eqref{reduced} in $\overline{\Omega'_{R,\alpha}}$, the function $\overline u$ is a supersolution of \eqref{1}. Finally, from the construction of $u$ in the proof of Proposition~\ref{thm1}, and in particular from~\eqref{defw-},~\eqref{defunn} and the fact that $w^-(t,\cdot)\to0$ as $t\to-\infty$ locally uniformly in~$\overline{\Omega}$, one has
$$u_n(-n,\cdot)\le\overline{u}\ \hbox{ in $\overline{\Omega}$}$$
for all $n$ large enough, hence $u_n(t,\cdot)\le\overline{u}$ in $\overline{\Omega}$ for all $t\ge-n$ and all $n$ large enough, by the maximum principle. As a consequence, $u(t,\cdot)\le\overline{u}$ in $\overline{\Omega}$ for all $t\in\R$, and the large time limit~$u_\infty$ of $u(t,\cdot)$ satisfies $0<u_\infty\le\overline{u}$ in $\overline{\Omega}$. Thus, $0<u_\infty\le w$ in $\overline{\Omega'_{R,\alpha}}$ and $u_\infty(x)\to0$ as $x_1\to+\infty$ in~$\overline{\Omega}$. The proof of Theorem~\ref{thm6} is thereby complete.


\section{The set of $(R,\alpha)$ with complete propagation property is open in $(0,+\infty)\times(0,\pi/2)$: proof of Theorem~\ref{thm7}}\label{Sec-6}

This section is devoted to the proof of Theorem~\ref{thm7}. The main strategy is to argue by way of contradiction and make use of Corollary~\ref{cor2} and Lemma~\ref{lemliouville1}. So, let $(R,\alpha)\in(0,+\infty)\times(0,\pi/2)$ be such that the solution $u$ of~\eqref{1} with past condition~\eqref{initial} propagates completely in the sense of~\eqref{complete}, and let us assume that there is a sequence $(R_n,\alpha_n)_{n\in\N}$ in $(0,+\infty)\times(0,\pi/2)$ converging to $(R,\alpha)$, such that the solutions $u_n$ of~\eqref{1} (in $\R\times\overline{\Omega_{R_n,\alpha_n}}$) with past conditions~\eqref{initial} do not propagate completely. From the dichotomy result of Theorem~\ref{thm2}, this means that each solution~$u_n$ is blocked, that is, there is a $C^2(\overline{\Omega_{R_n,\alpha_n}})$ solution $0<u_{\infty,n}<1$ of~\eqref{u-infty} in $\overline{\Omega_{R_n,\alpha_n}}$ such that $u_n(t,x)\to u_{\infty,n}(x)$ in $C^2_{loc}(\overline{\Omega_{R_n,\alpha_n}})$ as $t\to+\infty$ and
$$u_{\infty,n}(x)\to0\ \hbox{ as $x_1\to+\infty$ with $x\in\overline{\Omega_{R_n,\alpha_n}}$}.$$
On the other hand, by assumption of the theorem, the functions $h_n$ involved in the definitions~\eqref{defOmega}-\eqref{h} of the sets $\Omega_{R_n,\alpha_n}$ converge (in $C^{2,\beta}_{loc}(\R)$) to the function $h$ involved in the definition of the set $\Omega_{R,\alpha}$. In particular, since $\alpha>0$, there is a point $x_0\in\R^N$ (independent of~$n\in\N$) such that $\overline{B_{R_0}(x_0)}\subset\overline{\Omega_{R,\alpha}}$ and $\overline{B_{R_0}(x_0)}\subset\overline{\Omega_{R_n,\alpha_n}}$ for all $n\in\N$, where $R_0>0$ is given as in Lemma~\ref{lemsub}. It then follows from Lemmas~~\ref{lemsub} and~\ref{lemliouville1} (the latter 	applied in $\Omega_{R_n,\alpha_n}$) that, for each $n\in\N$,
$$\min_{\overline{B_{R_0}(x_0)}}u_{\infty,n}<\max_{\overline{B_{R_0}}}\psi=\psi(0)<1,$$
where the $C^2(\overline{B_{R_0}})$ function $\psi$ is as in Lemma~\ref{lemsub}. From standard elliptic estimates, there is a $C^2(\overline{\Omega_{R,\alpha}})$ solution $0\le U\le1$ of~\eqref{u-infty} in $\overline{\Omega_{R,\alpha}}$ such that, up to extraction of a subsequence, $\|u_{\infty,n}-U\|_{C^2(K\cap\overline{\Omega_{R_n,\alpha_n}})}\to0$ as $n\to+\infty$ for every compact set $K\subset\overline{\Omega_{R,\alpha}}$. In particular, one has
\be\label{U<1}
\min_{\overline{B_{R_0}(x_0)}}U\le\psi(0)<1.
\ee
Finally, remember that the functions $u_{\infty,n}(x)$ converge to $1$ as $x_1\to-\infty$ with $x\in\overline{\Omega_{R_n,\alpha_n}}$, uniformly with respect to $n\in\N$, from property~\eqref{cv1unif} in the construction of the solutions $u$ in Section~\ref{Sec-2}. Therefore, $U(x)\to1$ as $x_1\to-\infty$ with $x\in\overline{\Omega_{R,\alpha}}$. Corollary~\ref{cor2} then implies that the solution $u$ of~\eqref{1} in $\R\times\overline{\Omega_{R,\alpha}}$ with past condition~\eqref{initial} satisfies $u(t,x)\le U(x)$ for all $(t,x)\in\R\times\overline{\Omega_{R,\alpha}}$. The condition~\eqref{U<1} then means that $u$ does not propagate completely, which is a contradiction. The proof of Theorem~\ref{thm7} is thereby complete.\hfill$\Box$


\end{document}